\documentclass[a4paper]{amsart}
\usepackage{amscd}
\usepackage{amsmath}
\usepackage{amssymb}
\usepackage{mathrsfs}
\usepackage{amsthm}
\usepackage{bbm}
\usepackage{stmaryrd}
\usepackage{tikz-cd}

\usepackage[T1]{fontenc}
\usepackage[utf8]{inputenc}

\usepackage{lipsum}

\makeatletter
\renewcommand\subsection{\@startsection{subsection}{2}%
  \z@{-.5\linespacing\@plus-.7\linespacing}{.5\linespacing}%
  {\normalfont\scshape}}
\renewcommand\subsubsection{\@startsection{subsubsection}{3}%
  \z@{.5\linespacing\@plus.7\linespacing}{-.5em}%
  {\normalfont\scshape}}
\makeatother

\makeatletter
\@namedef{subjclassname@2020}{%
  \textup{2020} Mathematics Subject Classification}
\makeatother

\numberwithin{equation}{section} \swapnumbers

\newtheorem{satz}{Satz}[section]

\newtheorem{theorem}[satz]{Theorem}
\newtheorem{proposition}[satz]{Proposition}
\newtheorem{corollary}[satz]{Corollary}
\newtheorem{lemma}[satz]{Lemma}

\newtheorem{definition}[satz]{Definition}

\newtheorem{remark}[satz]{Remark}

\newcommand{\bbb}{\mathbb{B}}

\newcommand{\bbe}{\mathbb{E}}

\newcommand{\bbn}{\mathbb{N}}
\newcommand{\bbp}{\mathbb{P}}

\newcommand{\bbr}{\mathbb{R}}

\newcommand{\bbx}{\mathbb{X}}

\newcommand{\calc}{\mathcal{C}}

\newcommand{\calf}{\mathcal{F}}
\newcommand{\calg}{\mathcal{G}}

\newcommand{\call}{\mathcal{L}}

\newcommand{\scrc}{\mathscr{C}}
\newcommand{\scrd}{\mathscr{D}}

\newcommand{\bfx}{\mathbf{X}}

\newcommand{\Id}{{\rm Id}}

\newcommand{\Lip}{{\rm Lip}}
\newcommand{\tr}{{\rm tr}}

\newcommand{\la}{\langle}
\newcommand{\ra}{\rangle}

\begin{document}

\title[Mild solutions to rough partial differential equations]{Mild solutions to semilinear rough partial differential equations}
\author{Stefan Tappe}
\address{Albert Ludwig University of Freiburg, Department of Mathematical Stochastics, Ernst-Zermelo-Stra\ss{}e 1, D-79104 Freiburg, Germany}
\email{stefan.tappe@math.uni-freiburg.de}
\date{14 March, 2024}
\thanks{The author gratefully acknowledges financial support from the Deutsche Forschungsgemeinschaft (DFG, German Research Foundation) -- project number 444121509.}
\begin{abstract}
We provide an existence and uniqueness result for mild solutions to rough partial differential equations in the framework of the semigroup approach. Applications to stochastic partial differential equations driven by infinite dimensional Wiener processes and infinite dimensional fractional Brownian motion are presented as well.
\end{abstract}

\maketitle\thispagestyle{empty}

\section{Introduction}

In this paper we deal with rough partial differential equations (RPDEs) with time-inhomogeneous coefficients of the form
\begin{align}\label{RPDE-intro}
\left\{
\begin{array}{rcl}
dY_t & = & (A Y_t + f_0(t,Y_t)) dt + f(t,Y_t) d \bfx_t
\\ Y_0 & = & \xi.
\end{array}
\right.
\end{align}
The state space of the RPDE (\ref{RPDE-intro}) is a Banach space $W$, and the driving signal $\bfx = (X,\bbx) \in \scrc^{\alpha}([0,T],V)$ is a rough path with values in a Banach space $V$ for some index $\alpha \in (\frac{1}{3},\frac{1}{2}]$. Furthermore, the operator $A$ is the generator of a $C_0$-semigroup $(S_t)_{t \geq 0}$ on $W$, and $f_0 : [0,T] \times W \to W$ and $f : [0,T] \times W \to L(V,W)$ are appropriate mappings. We are interested in the existence of mild solutions, which means that the variations of constants formula
\begin{align}\label{mild-intro}
Y_t = S_t \xi + \int_0^t S_{t-s} f_0(s,Y_s) \, ds + \int_0^t S_{t-s} f(s,Y_s) \, d \bfx_s, \quad t \in [0,T]
\end{align}
is satisfied.

By now, there is a well-established theory of rough paths, including rough integration and rough differential equations (RDEs); see the textbook \cite{Friz-Hairer-2020}, and also the lecture notes \cite{Allan}. The goal of the present paper is to extend the existing RDE results to RPDEs with with time-inhomogeneous coefficients of the form (\ref{RPDE-intro}), and to apply these findings to stochastic partial differential equations (SPDEs) driven by infinite dimensional Wiener processes, and to SPDEs driven by infinite dimensional fractional Brownian motion.

To some extent, RPDEs and various related aspects have already been studied in the literature; we refer, for example, to \cite{Gubinelli, Hairer-PDEs, Friz-Oberhauser, Teichmann, Friz-Oberhauser-2, Gerasimovics, Gerasimovics-et-al} for results about RPDEs, to \cite{Schmalfuss-00, Hesse-Neamtu-local, Hesse-Neamtu-global-2} for RPDE results with a particular focus on fractional Brownian motion, to \cite{Schmalfuss-0, Hesse-Neamtu-global, Kuehn, Riedel} for RPDE results with a viewpoint to random dynamical systems, to \cite{Atma-book, Da_Prato, Liu-Roeckner} for results about SPDEs driven by Wiener processes, to \cite{Grecksch-Anh, Maslowski, Duncan-Maslowski, Duncan-Maslowski-2, Duncan-Maslowski-3, Maslowski-Neerven} for results concerning SPDEs driven by fractional Brownian motion, and to \cite{Duncan, Grecksch, Schmalfuss} for various related aspects of infinite dimensional fractional Brownian motion.

One of the principle challenges when dealing with mild solutions to RPDEs is an appropriate definition of the rough convolution
\begin{align}\label{conv-intro}
\int_0^t S_{t-s} Y_s \, d \mathbf{X}_s, \quad t \in [0,T]
\end{align}
for controlled rough paths $(Y,Y')$, which shows up in the variation of constants formula (\ref{mild-intro}). For this reason,  most of the aforementioned papers about RPDE results assume that the semigroup $(S_t)_{t \geq 0}$ is analytic. Considering scales of Hilbert spaces, this assumption allows to establish a mild Sewing Lemma, and then to provide a definition of the rough convolution (\ref{conv-intro}) in the spirit of the Gubinelli integral.

In this paper, we pursue a different approach. The essential idea is to provide a direct definition of rough convolutions (\ref{conv-intro}) by using the well-known Gubinelli integral, without the need of establishing a mild Sewing Lemma and a new integral definition. This is possible for controlled rough paths $(Y,Y') \in \scrd_X^{2 \alpha}([0,T],L(V,D(A^2)))$, where $D(A^2)$ denotes the second order domain of the generator; see Section \ref{sec-rough-conv} for details. Of course, this imposes additional conditions on the coefficients $f_0$ and $f$ of the RPDE (\ref{RPDE-intro}), but on the other hand, our results (see in particular Theorem \ref{thm-RDGL-main}) hold true for arbitrary $C_0$-semigroups $(S_t)_{t \geq 0}$. Moreover, in contrast to most of the aforementioned papers, we treat RPDEs with time-inhomogeneous coefficients.

The remainder of this paper is organized as follows. In Section \ref{sec-notation} we briefly explain some frequently used notation. In Section \ref{sec-rough-paths} we provide the required background about rough path theory. In Section \ref{sec-semigroups} we provide the required results about strongly continuous semigroups. In Section \ref{sec-comp-reg} we treat compositions of regular paths with functions, and in Section \ref{sec-comp-rough} we deal with compositions of controlled rough paths with functions. In Section \ref{sec-reg-conv} we consider regular convolution integrals, and in Section \ref{sec-rough-conv} we investigate rough convolution integrals. After these preparations, in Section \ref{sec-RPDEs} we deal with rough partial differential equations; in particular, we establish the announced existence and uniqueness result. Afterwards, in Section \ref{sec-Wiener} we apply our findings to stochastic partial differential equations driven by infinite dimensional Wiener processes, and in Section \ref{sec-fractional} we apply our findings to stochastic partial differential equations driven by infinite dimensional fractional Brownian motion.

\section{Frequently used notation}\label{sec-notation}

In this section we briefly explain some frequently used notation. The symbols $V$ and $W$ typically denote Banach spaces equipped with their respective norms, always written as $| \cdot |$. The space $L(V,W)$ of all continuous linear operators $T : V \to W$ is also a Banach space. We will also use the notation $L(V) = L(V,V)$. On the product space $V \times W$ we will always consider the norm
\begin{align*}
|v,w| = |v| + |w|,
\end{align*}
which makes $V \times W$ a Banach space. For another Banach space $E$ we denote by $L^{(2)}(V \times W, E)$ the space of all continuous bilinear operators $T : V \times W \to E$.

For a bounded function $f : V \to W$ we denote by $\| f \|_{\infty}$ the supremum norm. For $n \in \bbn$ we denote by $C^n = C^n(V,W)$ the space of all mappings $f : V \to W$ which are $n$-times continuously Fr\'{e}chet differentiable. Furthermore, we denote by $C_b^n = C_b^n(V,W)$ the subspace of all $f \in C^n$ such that
\begin{align*}
\| f \|_{C_b^n} := \sum_{k=0}^n \| D^k f \|_{\infty} < \infty.
\end{align*}
The space $\Lip(V,W)$ denotes the space of all Lipschitz continuous functions; that is, the space of all functions $f : V \to W$ such that
\begin{align*}
| f |_{\Lip} := \sup_{x,y \in V \atop x \neq y} \frac{|f(x) - f(y)|}{|x-y|} < \infty.
\end{align*}
Let $T \in \bbr_+$ and $\alpha \in (0,1]$ be arbitrary. We denote by $\calc^{\alpha}([0,T],V)$ the space of all $\alpha$-H\"{o}lder continuous functions; that is, the space of all functions $X : [0,T] \to V$ such that
\begin{align*}
\| X \|_{\alpha} := \sup_{s,t \in [0,T] \atop s \neq t} \frac{|X_{s,t}|}{|t-s|^{\alpha}} < \infty.
\end{align*}
Let $X \in \calc^{\alpha}([0,T],V)$ be arbitrary. For a subinterval $I \subset [0,T]$ we use the notation
\begin{align*}
\| X \|_{\alpha;I} := \sup_{s,t \in I \atop s \neq t} \frac{|X_{s,t}|}{|t-s|^{\alpha}}.
\end{align*}
Moreover, for any $h > 0$ the quantity $\| X \|_{\alpha;h}$ denotes the supremum over all $\| X \|_{\alpha;I}$, where $I \subset [0,T]$ is any subinterval with $|I| \leq h$.

\section{Rough path theory}\label{sec-rough-paths}

In this section we provide the required background about rough path theory. Further details can be found, for example, in \cite{Friz-Hairer-2020} or \cite{Allan}.

Let $V$ be a Banach space, and let $T \in \bbr_+$ be a time horizon. In the sequel, a function $X : [0,T] \to V$ will be called a \emph{path}. For a path $X : [0,T] \to V$ we agree on the notation
\begin{align*}
X_{s,t} := X_t - X_s, \quad s,t \in [0,T].
\end{align*}

\begin{definition}
Let $\alpha \in (\frac{1}{3},\frac{1}{2}]$ be arbitrary. We define the space $\scrc^{\alpha}([0,T],V)$ of all \emph{$\alpha$-H\"{o}lder rough paths} (over $V$) as the space of all pairs $\mathbf{X} = (X,\bbx)$ with paths $X : [0,T] \to V$ and $\bbx : [0,T]^2 \to V \otimes V$ such that
\begin{align*}
\| X \|_{\alpha} &:= \sup_{s,t \in [0,T] \atop s \neq t} \frac{|X_{s,t}|}{|t-s|^{\alpha}} < \infty,
\\ \| \bbx \|_{2\alpha} &:= \sup_{s,t \in [0,T] \atop s \neq t} \frac{|\bbx_{s,t}|}{|t-s|^{2\alpha}} < \infty,
\end{align*}
and \emph{Chen's relation}
\begin{align}\label{Chen-relation}
 \bbx_{s,t} - \bbx_{s,u} - \bbx_{u,t} = X_{s,u} \otimes X_{u,t} \quad \text{for all $s,u,t \in [0,T]$}
\end{align}
is satisfied.
\end{definition}

\begin{remark}
The path $\bbx$ of a rough path $\mathbf{X} = (X,\bbx)$ is also called the associated \emph{second order process} or \emph{L\'{e}vy area}.
\end{remark}

\begin{definition}
For a rough path $\mathbf{X} = (X,\bbx) \in \scrc^{\alpha}([0,T],V)$ for some index $\alpha \in (\frac{1}{3},\frac{1}{2}]$ we introduce
\begin{align*}
\interleave \mathbf{X} \interleave_{\alpha} := \| X \|_{\alpha} + \| \bbx \|_{2 \alpha}.
\end{align*}
\end{definition}

\begin{lemma}\label{lemma-alpha-beta}
Let $\beta \in (\frac{1}{3},\frac{1}{2}]$ and $\alpha \in (0,\beta)$ be arbitrary. Then for each rough path $\mathbf{X} = (X,\bbx) \in \scrc^{\beta}([0,T],V)$ we also have $\mathbf{X} \in \scrc^{\alpha}([0,T],V)$ and the estimate
\begin{align*}
\interleave \mathbf{X} \interleave_{\alpha} \leq \| X \|_{\beta} T^{\beta-\alpha} + \| \mathbb{X} \|_{2\beta} T^{2(\beta-\alpha)}.
\end{align*}
\end{lemma}

\begin{proof}
Since $\alpha < \beta$, we have $\mathbf{X} \in \scrc^{\alpha}([0,T],V)$. Hence, the calculation
\begin{align*}
\interleave \mathbf{X} \interleave_{\alpha} &= \| X \|_{\alpha} + \| \bbx \|_{2\alpha}
\\ &\leq \| X \|_{\beta} T^{\beta-\alpha} + \| \mathbb{X} \|_{2\beta} T^{2(\beta-\alpha)}
\end{align*}
establishes the proof.
\end{proof}

For what follows, we fix an index $\alpha \in (\frac{1}{3},\frac{1}{2}]$. The following result is well-known.

\begin{lemma}
The space $\scrc^{\alpha}([0,T],V)$, equipped with the metric
\begin{align*}
d_{\alpha}(\mathbf{X},\mathbf{Y}) = |X_0 - Y_0| + \interleave \mathbf{X} - \mathbf{Y} \interleave_{\alpha},
\end{align*}
is a complete metric space.
\end{lemma}

\begin{definition}
We define the space $\scrc_g^{\alpha}([0,T],V)$ of \emph{weakly geometric rough paths} as the set of all $(X,\bbx) \in \scrc^{\alpha}([0,T],V)$ such that
\begin{align*}
{\rm Sym}(\bbx_{s,t}) = \frac{1}{2} X_{s,t} \otimes X_{s,t} \quad \text{for all $s,t \in [0,T]$.}
\end{align*}
\end{definition}

We list two well-known auxiliary results about weakly geometric rough paths.

\begin{lemma}\label{lemma-geom-1}
Let $X : [0,T] \to V$ be a continuous path which is piecewise of class $C^1$. More precisely, there are $n \in \bbn$ and $0 = t_0 < t_1 < \ldots < t_n = T$ such that $X|_{(t_{i-1},t_i)} \in C^1((t_{i-1},t_i),V)$ for all $i=1,\ldots,n$. Then we have $\mathbf{X} = (X,\bbx) \in \scrc_g^{\alpha}([0,T],V)$, where the second order process $\bbx : [0,T]^2 \to V \otimes V$ is given by
\begin{align*}
\bbx_{s,t} = \int_s^t X_{s,r} \otimes dX_r, \quad s,t \in [0,T].
\end{align*}
\end{lemma}

\begin{lemma}\label{lemma-geom-2}
$\scrc_g^{\alpha}([0,T],V)$ is a closed subset of $\scrc^{\alpha}([0,T],V)$ with respect to the metric $d_{\alpha}$.
\end{lemma}

For what follows, we fix a rough path $\mathbf{X} = (X,\bbx) \in \scrc^{\alpha}([0,T],V)$. Furthermore, let $W$ be another Banach space.

\begin{definition}
We introduce the following notions:
\begin{enumerate}
\item[(a)] We say that a path $Y \in \calc^{\alpha}([0,T],W)$ is \emph{controlled} by $X$ if there exists $Y' \in \calc^{\alpha}([0,T], L(V,W))$ such that $\| R^Y \|_{2 \alpha} < \infty$, where the remainder term $R^Y$ is given by
\begin{align*}
R_{s,t}^Y &:= Y_{s,t} - Y_s' X_{s,t} \quad \text{for all $s,t \in [0,T]$.}
\end{align*}
\item[(b)] This defines the space of \emph{controlled rough paths}, and we write
\begin{align*}
(Y,Y') \in \scrd_X^{2 \alpha}([0,T], W).
\end{align*}
\item[(c)] We call $Y'$ a \emph{Gubinelli derivative} of $Y$ (with respect to $X$).
\end{enumerate}
\end{definition}

We endow the space $\scrd_X^{2 \alpha}([0,T], W)$ with the seminorm
\begin{align*}
\| Y,Y' \|_{X,2\alpha} := \| Y' \|_{\alpha} + \| R^Y \|_{2 \alpha}.
\end{align*}
Then the space $\scrd_X^{2 \alpha}([0,T], W)$ endowed with the norm
\begin{align*}
\interleave Y,Y' \interleave_{X,2\alpha} := |Y_0| + |Y_0'| + \| Y,Y' \|_{X,2\alpha}
\end{align*}
is a Banach space. We will also consider the seminorm
\begin{align*}
| Y,Y' |_{X,2\alpha} := |Y_0'| + \| Y,Y' \|_{X,2\alpha}.
\end{align*}

\begin{lemma}\label{lemma-reg-derivative-zero}
For each $Y \in \calc^{2 \alpha}([0,T],W)$ we have $(Y,0) \in \scrd_X^{2 \alpha}([0,T],W)$ with
\begin{align*}
\| Y,Y' \|_{X,2\alpha} = \| Y \|_{2 \alpha}.
\end{align*}
\end{lemma}

\begin{proof}
Since $Y' = 0$, we have $R^Y = Y$, proving the statement.
\end{proof}

\begin{lemma}\label{lemma-norm-of-Y}
For all $(Y,Y') \in \scrd_X^{2 \alpha}([0,T],W)$ we have
\begin{align}\label{norm-Y-1}
\| Y' \|_{\infty} &\leq C | Y,Y' |_{X,2\alpha},
\\ \label{norm-Y-2} \| Y \|_{\alpha} &\leq C | Y,Y' |_{X,2\alpha} ( \interleave \mathbf{X} \interleave_{\alpha} + T^{\alpha} ),
\\ \label{norm-Y-3} \| Y \|_{\infty} &\leq C \interleave Y,Y' \interleave_{X,2\alpha} ( \interleave \mathbf{X} \interleave_{\alpha} + 2 ),
\end{align}
where the constant $C > 0$ depends on $\alpha$ and $T$. Furthermore, for $T \leq 1$ the estimates (\ref{norm-Y-1})--(\ref{norm-Y-3}) hold true with $C=1$.
\end{lemma}

\begin{proof}
For each $t \in [0,T]$ we have
\begin{align*}
|Y_t'| \leq |Y_0'| + |Y_t' - Y_0'| \leq |Y_0'| + \| Y' \|_{\alpha} t^{\alpha},
\end{align*}
and hence
\begin{align*}
\| Y' \|_{\infty} &\leq |Y_0'| + \| Y' \|_{\alpha} T^{\alpha} \leq \max \{ 1,T^{\alpha} \} \, | Y,Y' |_{X,2\alpha},
\end{align*}
proving (\ref{norm-Y-1}). Now, let $s,t \in [0,T]$ be arbitrary. Recalling that $Y_{s,t} = Y_s' X_{s,t} + R_{s,t}^Y$, we have
\begin{align*}
| Y_{s,t} | &\leq \| Y' \|_{\infty} |X_{s,t}| + |R_{s,t}^Y|
\\ &\leq \| Y' \|_{\infty} \| X \|_{\alpha} |t-s|^{\alpha} + \| R_{s,t}^Y \|_{2\alpha} |t-s|^{2\alpha}
\\ &\leq \| Y' \|_{\infty} \| X \|_{\alpha} |t-s|^{\alpha} + \| R_{s,t}^Y \|_{2\alpha} T^{\alpha} |t-s|^{\alpha}.
\end{align*}
Therefore, we obtain
\begin{align*}
\| Y \|_{\alpha} &\leq \| Y' \|_{\infty} \| X \|_{\alpha} + \| R^Y \|_{2\alpha} T^{\alpha}
\\ &\leq \max \{ 1,T^{\alpha} \} \, | Y,Y' |_{X,2\alpha} \| X \|_{\alpha} + \| R^Y \|_{2\alpha} T^{\alpha}
\\ &\leq  \max \{ 1,T^{\alpha} \} | Y,Y' |_{X,2\alpha} ( \interleave \mathbf{X} \interleave_{\alpha} + T^{\alpha} ),
\end{align*}
showing (\ref{norm-Y-2}). Using this estimate we deduce that
\begin{align*}
\| Y \|_{\infty} &\leq |Y_0| + \| Y \|_{\alpha} T^{\alpha}
\\ &\leq \interleave Y,Y' \interleave_{X,2\alpha} + \max \{ 1,T^{\alpha} \} T^{\alpha} \interleave Y,Y' \interleave_{X,2\alpha} ( \interleave \mathbf{X} \interleave_{\alpha} + T^{\alpha} )
\\ &\leq \interleave Y,Y' \interleave_{X,2\alpha} \big( \max \{ 1,T^{\alpha} \} T^{\alpha} \interleave \mathbf{X} \interleave_{\alpha} + 1 + \max \{ 1,T^{\alpha} \} T^{2\alpha} \big),
\end{align*}
proving (\ref{norm-Y-3}).
\end{proof}

\begin{lemma}\label{lemma-pairs}
Let $E$ be another Banach space. Then for all $(Y,Y') \in \scrd_X^{2\alpha}([0,T],W)$ and $(Z,Z') \in \scrd_X^{2\alpha}([0,T],E)$ we have
\begin{align}\label{vector-in-D}
((Y,Z),(Y,Z)') \in \scrd_X^{2\alpha}([0,T],W \times E),
\end{align}
where the Gubinelli derivative is given by $(Y,Z)' = (Y',Z')$, and we have the estimates
\begin{align}\label{pairs-1}
\| (Y,Z), (Y,Z)' \|_{X,2\alpha} &\leq \| Y,Y' \|_{X,2\alpha} + \| Z,Z' \|_{X,2\alpha},
\\ \label{pairs-2} | (Y,Z), (Y,Z)' |_{X,2\alpha} &\leq | Y,Y' |_{X,2\alpha} + | Z,Z' |_{X,2\alpha},
\\ \label{pairs-3} \interleave (Y,Z), (Y,Z)' \interleave_{X,2\alpha} &\leq \interleave Y,Y' \interleave_{X,2\alpha} + \interleave Z,Z' \interleave_{X,2\alpha}.
\end{align}
\end{lemma}

\begin{proof}
Let $s,t \in [0,T]$ be arbitrary. Then we have
\begin{align*}
|(Y_t,Z_t)-(Y_s,Z_s)| &= |(Y_t-Y_s,Z_t-Z_s)| = |Y_t-Y_s| + |Z_t-Z_s|
\\ &\leq \big( \| Y \|_{\alpha} + \| Z \|_{\alpha} \big) |t-s|^{\alpha},
\end{align*}
showing that $(Y,Z) \in \calc^{\alpha}([0,T],W_1 \times W_2)$.
Similarly, we have
\begin{align*}
|(Y_t',Z_t')-(Y_s',Z_s')| &= |(Y_t'-Y_s',Z_t'-Z_s')| = |Y_t'-Y_s'| + |Z_t'-Z_s'|
\\ &\leq \big( \| Y' \|_{\alpha} + \| Z' \|_{\alpha} \big) |t-s|^{\alpha},
\end{align*}
showing that $(Y',Z') \in \calc^{\alpha}([0,T],L(V,W_1 \times W_2))$. Furthermore, we have
\begin{align*}
R_{s,t}^{(Y,Z)} = (Y_{s,t},Z_{s,t}) - (Y_s',Z_s') X_{s,t} = (R_{s,t}^Y,R_{s,t}^Z),
\end{align*}
and hence
\begin{align*}
|R_{s,t}^{(Y,Z)}| = |R_{s,t}^Y| + |R_{s,t}^Z| \leq \big( \| R^Y \|_{2\alpha} + \| R^Z \|_{2\alpha} \big) |t-s|^{2\alpha},
\end{align*}
proving (\ref{vector-in-D}). Moreover, the previous calculations reveal that
\begin{align*}
\| (Y',Z') \|_{\alpha} &\leq \| Y' \|_{\alpha} + \| Z' \|_{\alpha},
\\ \| R^{(Y,Z)} \|_{2\alpha} &\leq \| R^Y \|_{2\alpha} + \| R^Z \|_{2\alpha}.
\end{align*}
Therefore, the estimates (\ref{pairs-1})--(\ref{pairs-3}) are an immediate consequence.
\end{proof}

Now, let $(Y,Y') \in \scrd_X^{2 \alpha}([0,T],L(V,W))$ be a controlled rough path. Noting that $L(V,L(V,W)) \hookrightarrow L(V \otimes V,W)$, we consider the candidate
\begin{align}\label{candidate}
\int_0^1 Y_s \, d \bfx_s := \lim_{|\Pi| \to 0} \sum_{[s,t] \in \Pi} ( Y_s X_{s,t} + Y_s' \bbx_{s,t} )
\end{align}
for the rough integral of $Y$ against $\mathbf{X}$. For the following result we refer to \cite[Thm. 4.10]{Friz-Hairer-2020}.

\begin{theorem}[Gubinelli]\label{thm-Gubinelli}
For each controlled rough path $(Y,Y') \in \scrd_X^{2 \alpha}([0,T],L(V,W))$ the following statements are true:
\begin{enumerate}
\item[(a)] The $W$-valued rough integral (\ref{candidate}) exists.

\item[(b)] For all $s,t \in [0,T]$ we have
\begin{align*}
\bigg| \int_s^t Y_r \, d\bfx_r - Y_s X_{s,t} - Y_s' \bbx_{s,t} \bigg| \leq C \big( \| X \|_{\alpha} \| R^Y \|_{2 \alpha} + \| \bbx \|_{2 \alpha} \| Y' \|_{\alpha} \big) |t-s|^{3 \alpha}
\end{align*}
with a constant $C > 0$ depending on $\alpha$.

\item[(c)] The map
\begin{align*}
&\scrd_X^{2 \alpha}([0,T],L(V,W)) \to \scrd_X^{2 \alpha}([0,T],W),
\\ &(Y,Y') \mapsto (Z,Z') := \bigg( \int_0^{\cdot} Y_s \, d \bfx_s, Y \bigg)
\end{align*}
is a continuous linear operator between Banach spaces, and we have
\begin{align*}
\| Z,Z' \|_{X,2 \alpha} \leq \| Y \|_{\alpha} + \| Y' \|_{\infty} \| \bbx \|_{2 \alpha} + C T^{\alpha} \big( \| X \|_{\alpha} \| R^Y \|_{2 \alpha} + \| \bbx \|_{2 \alpha} \| Y' \|_{\alpha} \big)
\end{align*}
with a constant $C > 0$ depending on $\alpha$.
\end{enumerate}
\end{theorem}

\begin{lemma}\label{lemma-int-cont-embedded spaces}
Let $E$ be another Banach space such that $E \subset W$ as sets, and there is a constant $C > 0$ such that $|y|_W \leq C |y|_E$ for all $y \in E$. Then for each controlled rough path $(Y,Y') \in \scrd_X^{2 \alpha}([0,T],L(V,E))$ we also have $(Y,Y') \in \scrd_X^{2 \alpha}([0,T],L(V,W))$ and the identity
\begin{align*}
\text{{\rm ($E$-)}} \int_0^1 Y_s \, d \bfx_s = \text{{\rm ($W$-)}} \int_0^1 Y_s \, d \bfx_s,
\end{align*}
where the integrals above denote the rough integrals (\ref{candidate}) in the Banach spaces $(E,|\cdot|_E)$ and $(W,|\cdot|_W)$.
\end{lemma}

\begin{proof}
The result is easily verified, noting that for any sequence $(y_n)_{n \in \bbn} \subset E$ the convergence $|y_n - y|_E \to 0$ implies $|y_n - y|_W \to 0$.
\end{proof}

\begin{proposition}\label{prop-Gubinelli-triangle}
Let $(Y,Y') \in \scrd_X^{2 \alpha}([0,T],L(V,W))$ be a controlled rough path. Then for all $s,t \in [0,T]$ we have
\begin{align*}
\bigg| \int_s^t Y_r \, d\bfx_r \bigg| &\leq C \big( \| X \|_{\alpha} \| R^Y \|_{2 \alpha} + \| \bbx \|_{2 \alpha} \| Y' \|_{\alpha} \big) |t-s|^{3 \alpha}
\\ &\quad + \| Y' \|_{\infty} \| \bbx \|_{2 \alpha} |t-s|^{2 \alpha} + \| Y \|_{\infty} \| X \|_{\alpha} |t-s|^{\alpha}
\end{align*}
with a constant $C > 0$ depending on $\alpha$.
\end{proposition}

\begin{proof}
Let $s,t \in [0,T]$ be arbitrary. Then we have
\begin{align*}
\bigg| \int_s^t Y_r \, d\bfx_r \bigg| \leq \bigg| \int_s^t Y_r \, d\bfx_r - Y_s X_{s,t} - Y_s' \bbx_{s,t} \bigg| + | Y_s X_{s,t} | + | Y_s' \bbx_{s,t} |.
\end{align*}
Noting that
\begin{align*}
| Y_s X_{s,t} | &\leq |Y_s| \, |X_{s,t}| \leq \| Y \|_{\infty} \| X \|_{\alpha} |t-s|^{\alpha},
\\ | Y_s' \bbx_{s,t} | &\leq |Y_s'| \, |\bbx_{s,t}| \leq \| Y' \|_{\infty} \| \bbx \|_{2\alpha} |t-s|^{2\alpha},
\end{align*}
the statement is a consequence of Theorem \ref{thm-Gubinelli}.
\end{proof}

For the following rough Fubini theorem see \cite[Exercise 4.10]{Friz-Hairer-2020}.

\begin{proposition}[Rough Fubini]\label{prop-rough-Fubini}
Let $(\Omega,\calf,\mu)$ be a measure space. Furthermore, let $(Y,Y') : \Omega \to \scrd_X^{2\alpha}([0,T],L(V,W))$ be a measurable mapping such that
\begin{align*}
\Omega \to \bbr_+, \quad \omega \mapsto |Y(\omega)_0| + \| Y(\omega),Y(\omega)' \|_{X,2\alpha}
\end{align*}
belongs to $\call^1(\Omega,\calf,\mu)$. Then the mapping
\begin{align*}
\Omega \to W, \quad \omega \mapsto \int_0^T Y(\omega)_s \, d \mathbf{X}_s
\end{align*}
belongs to $\call^1(\Omega,\calf,\mu;W)$, the mapping
\begin{align*}
s \mapsto \int_{\Omega} Y(\omega)_s \, \mu(d\omega)
\end{align*}
belongs to $\scrd_X^{2\alpha}([0,T],L(V,W))$, and we have the identity
\begin{align*}
\int_{\Omega} \bigg( \int_0^T Y(\omega)_s \, d \mathbf{X}_s \bigg) \mu(d\omega) = \int_0^T \bigg( \int_{\Omega} Y(\omega)_s \, \mu(d\omega) \bigg) d \mathbf{X}_s.
\end{align*}
\end{proposition}

\section{Strongly continuous semigroups}\label{sec-semigroups}

In this section we provide the required results about strongly continuous semigroups. Further details can be found, for example, in \cite{Engel-Nagel} or \cite{Pazy}. Let $W$ be a Banach space. A family $(S_t)_{t \geq 0}$ of continuous linear operators $S_t \in L(W)$ is called a \emph{strongly continuous semigroup} (or \emph{$C_0$-semigroup}) on $W$ if the following conditions are fulfilled:
\begin{enumerate}
\item $S_0 = \Id$.

\item $S_{s+t} = S_s S_t$ for all $s,t \geq 0$.

\item $\lim_{t \to 0} S_t y = y$ for all $y \in W$.
\end{enumerate}
Let $(S_t)_{t \geq 0}$ be a $C_0$-semigroup on $W$. It is well-known that there are constants $M \geq 1$ and $\omega \in \bbr$ such that
\begin{align}\label{est-semigroup}
|S_t| \leq M e^{\omega t} \quad \text{for all $t \geq 0$.}
\end{align}
If we can choose $M=1$ and $\omega = 0$ in (\ref{est-semigroup}), that is $|S_t| \leq 1$ for all $t \geq 0$, then we call the $C_0$-semigroup $(S_t)_{t \geq 0}$ a \emph{semigroup of contractions}. Recall that the \emph{infinitesimal generator} $A : W \supset D(A) \to W$ is the  operator
\begin{align*}
Ay := \lim_{h \to 0} \frac{S_h y - y}{h}
\end{align*}
defined on the \emph{domain}
\begin{align*}
D(A) := \bigg\{ y \in W : \lim_{h \to 0} \frac{S_h y - y}{h} \text{ exists} \bigg\}.
\end{align*}
The following result is well-known.

\begin{lemma}\label{lemma-hg-rules}
Let $(S_t)_{t \geq 0}$ be a $C_0$-semigroup with infinitesimal generator $A$. Then the following statements are true:
\begin{enumerate}
\item[(a)] For every $y \in D(A)$ the mapping
\begin{align*}
\mathbb{R}_+ \rightarrow W, \quad t \mapsto S_t y
\end{align*}
belongs to class $C^1(\mathbb{R}_+;W)$.

\item[(b)] For all $t \geq 0$ and $y \in D(A)$ we have $S_t y \in D(A)$ and
\begin{align*}
\frac{d}{dt} S_t y = A S_t y = S_t A y.
\end{align*}

\item[(c)] For all $t \geq 0$ and $y \in W$ we have $\int_0^t S_s y \, ds \in D(A)$ and
\begin{align*}
A \bigg( \int_0^t S_s y \, ds \bigg) = S_t y - y.
\end{align*}

\item[(d)] For all $t \geq 0$ and $y \in D(A)$ we have
\begin{align*}
\int_0^t S_s A y \, ds = S_t y - y.
\end{align*}
\end{enumerate}
\end{lemma}

The domain $D(A)$ equipped with the graph norm
\begin{align*}
|y|_{D(A)} := |y| + |Ay| \quad \text{for all $y \in D(A)$}
\end{align*}
is a Banach space. Inductively, for each $n \geq 2$ we define the \emph{higher-order domain}
\begin{align*}
D(A^n) := \{ y \in D(A^{n-1}) :A^{n-1} y \in D(A) \}.
\end{align*}
Furthermore, we agree on the notation $D(A^0) := W$. Then for each $n \in \bbn_0$ the space $D(A^n)$ equipped with the graph norm
\begin{align*}
|y|_{D(A^n)} := |y| + \sum_{j=1}^n |A^j y| \quad \text{for all $y \in D(A^n)$}
\end{align*}
is a Banach space. Using Lemma \ref{lemma-hg-rules}, the upcoming auxiliary result is easy to verify.

\begin{lemma}\label{lemma-restricted-semigroup}
For every $n \in \bbn$ the restriction $(S_t|_{D(A^n)})_{t \geq 0}$ is a $C_0$-semigroup on $(D(A^n),|\cdot|_{D(A^n)})$ with generator $A$ on the domain $D(A^{n+1})$.
\end{lemma}

The following auxiliary result is also straightforward to check.

\begin{lemma}\label{lemma-graph-norms-eq}
For all $m,n \in \bbn_0$ the graph norms $|\cdot|_{D(A^{m+n})}$ and
\begin{align*}
|y|_{D(A^m)} + \sum_{j=1}^n |A^j y|_{D(A^m)}, \quad y \in D(A^{m+n})
\end{align*}
on are equivalent norms on the space $D(A^{m+n})$.
\end{lemma}

For what follows, we agree on the notation
\begin{align*}
S_{s,t} := S_t - S_s \quad \text{for all $s,t \in \bbr_+$.}
\end{align*}
Furthermore, we fix some $T \in \bbr_+$.

\begin{proposition}\label{prop-orbit-map}
For all $y \in D(A)$ and all $s,t \in [0,T]$ we have
\begin{align*}
| S_{s,t} y | \leq M e^{\omega T} | y |_{D(A)} \, |t-s|.
\end{align*}
\end{proposition}

\begin{proof}
By Lemma \ref{lemma-hg-rules} and the growth estimate (\ref{est-semigroup}) we obtain
\begin{align*}
| S_{s,t} y | &= | S_{t} y - S_{s} y | = \bigg| \int_{s}^{t} S_u Ay \, du \bigg| \leq \int_{s}^{t} | S_u Ay | \, du
\\ &\leq \int_{s}^{t} | S_u | \, | Ay | \, du \leq Me^{\omega T} | y |_{D(A)} |t - s|,
\end{align*}
completing the proof.
\end{proof}

\begin{corollary}\label{cor-orbit-map-1}
For all $s,t \in [0,T]$ we have
\begin{align*}
| S_{s,t} |_{L(D(A),W)} \leq M e^{\omega T} |t-s|.
\end{align*}
\end{corollary}

\begin{corollary}\label{cor-orbit-map-2}
Let $V$ be another Banach space. Then for all $Y \in L(V,D(A))$ and all $s,t \in [0,T]$ we have
\begin{align*}
| S_{s,t} Y |_{L(V,W)} \leq M e^{\omega T} | Y |_{L(V,D(A))} \, |t-s|.
\end{align*}
\end{corollary}

The following result will be crucial for the analysis of rough convolutions later on.

\begin{proposition}\label{prop-est-quad}
Let $y \in D(A^2)$ be arbitrary. Then we have
\begin{align}\label{est-quad}
| S_{s-r,t-r}y - S_{s-q,t-q}y | \leq M e^{2 \omega T} |y|_{D(A^2)} |t-s| \, |r-q|
\end{align}
for all $s,t \in [0,T]$ with $s \leq t$ and all $q,r \in [0,s]$.
\end{proposition}

Before we provide the proof of Proposition \ref{prop-est-quad}, let us state some consequences.

\begin{corollary}\label{cor-orbit-map-quad-1}
For all $s,t \in [0,T]$ with $s \leq t$ and all $q,r \in [0,s]$ we have
\begin{align*}
| S_{s-r,t-r} - S_{s-q,t-q} |_{L(D(A^2),W)} \leq M e^{\omega T} |t-s| \, |r-q|.
\end{align*}
\end{corollary}

\begin{corollary}\label{cor-orbit-map-quad-2}
Let $V$ be another Banach space, and let $Y \in L(V,D(A^2))$ be arbitrary. Then for all $s,t \in [0,T]$ with $s \leq t$ and all $q,r \in [0,s]$ we have
\begin{align*}
| S_{s-r,t-r}Y - S_{s-q,t-q}Y |_{L(V,W)} \leq M e^{\omega T} | Y |_{L(V,D(A^2))} |t-s| \, |r-q|.
\end{align*}
\end{corollary}

Now, we prepare some auxiliary results for the proof of Proposition \ref{prop-est-quad}.

\begin{lemma}\label{lemma-est-group}
Suppose that $(S_t)_{t \in \bbr}$ is a $C_0$-group satisfying
\begin{align*}
| S_t | \leq e^{\omega |t|} \quad \text{for all $t \in \bbr$.}
\end{align*}
Let $y \in D(A^2)$ be arbitrary. Then we have (\ref{est-quad}) for all $s,t \in [0,T]$ with $s \leq t$ and all $q,r \in [0,s]$.
\end{lemma}

\begin{proof}
Using Proposition \ref{prop-orbit-map} and Lemma \ref{lemma-restricted-semigroup} we obtain
\begin{align*}
| S_{s-r,t-r}y - S_{s-q,t-q}y | &= | (S_{t-r} - S_{s-r})y - (S_{t-q} - S_{s-q})y |
\\ &= | (S_t - S_s)(S_{-r} - S_{-q}) y |
\\ &\leq e^{\omega T} | (S_{-r} - S_{-q}) y |_{D(A)} |t-s|
\\ &\leq e^{2 \omega T} |y|_{D(A^2)} |t-s| \, |r-q|,
\end{align*}
completing the proof.
\end{proof}

The following auxiliary result is easily verified.

\begin{lemma}\label{lemma-change-omega}
Let $(S_t)_{t \geq 0}$ be a $C_0$-semigroup with generator $A$ satisfying (\ref{est-semigroup}). Then the family $(T_t)_{t \geq 0}$ given by $T_t := e^{-\omega t} S_t$ for all $t \geq 0$ is a $C_0$-semigroup with generator $B := A - \omega$ and
\begin{align}\label{semigroup-bounded}
| T_t | \leq M \quad \text{for all $t \geq 0$.}
\end{align}
\end{lemma}

The next auxiliary result is a consequence of Lemma I.5.1 and the proof of Thm. I.5.2 in \cite{Pazy}.

\begin{lemma}\label{lemma-semigroup-eq-norm}
Let $(T_t)_{t \geq 0}$ be a $C_0$-semigroup satisfying (\ref{semigroup-bounded}). Then there exists an equivalent norm $\| \cdot \|$ on $W$ such that
\begin{align}\label{equiv-norm}
|y| \leq \| y \| \leq M |y| \quad \text{for all $y \in W$.}
\end{align}
and $(T_t)_{t \geq 0}$ is a semigroup of contractions under the norm $\| \cdot \|$.
\end{lemma}

Let $(T_t)_{t \geq 0}$ be a semigroup of contractions with generator $B$. By the Hille-Yosida theorem we have $(0,\infty) \subset \rho(B)$, where $\rho(B)$ denotes the resolvent set of $B$. For $\lambda > 0$ we define the \emph{Yosida approximation} $B_{\lambda} \in L(W)$ as
\begin{align*}
B_{\lambda} := \lambda A R(\lambda,B) = \lambda ^2 R(\lambda,B) - \lambda,
\end{align*}
where $R(\lambda,B) \in L(W)$ denotes the resolvent given by
\begin{align*}
R(\lambda,B) := (\lambda - B)^{-1}.
\end{align*}
Using Lemma I.3.3, Lemma I.3.4 and equations (3.12) and (3.14) in \cite{Pazy}, we obtain the following auxiliary result.

\begin{lemma}\label{lemma-Yosida-approx}
Let $(T_t)_{t \geq 0}$ be a $C_0$-semigroup of contractions with generator $B$. Then the following statements are true:
\begin{enumerate}
\item For each $\lambda > 0$ the family $(e^{t B_{\lambda}})_{t \in \bbr}$ is a uniformly continuous group of contractions.

\item For all $t \geq 0$ and $y \in W$ we have
\begin{align*}
\lim_{\lambda \to \infty} e^{t B_{\lambda}} y = T_t y.
\end{align*}
\item For all $t \geq 0$ and $y \in D(B)$ we have
\begin{align*}
\lim_{\lambda \to \infty} B_{\lambda} y = B y.
\end{align*}
\end{enumerate}
\end{lemma}

\begin{proof}
Lemma I.3.3, Lemma I.3.4 and eqn. (3.14) in \cite{Pazy}.
\end{proof}

Now, we are ready to provide the proof of Proposition \ref{prop-est-quad}.

\begin{proof}[Proof of Proposition \ref{prop-est-quad}]
By Lemma \ref{lemma-change-omega} the family $(T_t)_{t \geq 0}$ given by
$T_t := e^{-\omega t} S_t$ for all $t \geq 0$ is a $C_0$-semigroup with generator $B := A - \omega$ and we have (\ref{semigroup-bounded}). According to Lemma \ref{lemma-semigroup-eq-norm} there exists an equivalent norm $\| \cdot \|$ on $W$ such that we have (\ref{equiv-norm}) and $(T_t)_{t \geq 0}$ is a semigroup of contractions under the norm $\| \cdot \|$. For $\lambda > 0$ we define the family $(T_t^{\lambda})_{t \in \bbr}$ as $T_t^{\lambda} := e^{t B_{\lambda}}$ for all $t \in \bbr$. According to Lemma \ref{lemma-Yosida-approx} the family $(T_t^{\lambda})_{t \in \bbr}$ is a uniformly continuous group of contractions and we have
\begin{align*}
\lim_{\lambda \to \infty} T_t^{\lambda} y = T_t y \quad \text{for all $t \geq 0$ and $y \in W$.}
\end{align*}
Now, for $\lambda > 0$ define the family $(S_t^{\lambda})_{t \in \bbr}$ as $S_t^{\lambda} := e^{\omega t} T_t^{\lambda}$ for all $t \in \bbr$. Then $(S_t^{\lambda})_{t \in \bbr}$ is a uniformly continuous $C_0$-group satisfying
\begin{align*}
\| S_t^{\lambda} \| \leq e^{\omega |t|} \quad \text{for all $t \in \bbr$,}
\end{align*}
and we have
\begin{align}\label{conv-group}
\lim_{\lambda \to \infty} S_t^{\lambda} y = S_t y \quad \text{for all $t \geq 0$ and $y \in W$.}
\end{align}
Now, let $s,t \in [0,T]$ with $s \leq t$ and $q,r \in [0,s]$ be arbitrary. Using (\ref{equiv-norm}), (\ref{conv-group}) and Lemma \ref{lemma-est-group} we obtain
\begin{align*}
&| (S_{t-r} - S_{s-r})y - (S_{t-q} - S_{s-q})y | \leq \| (S_{t-r} - S_{s-r})y - (S_{t-q} - S_{s-q})y \|
\\ &= \lim_{\lambda \to \infty} \| (S_{t-r}^{\lambda} - S_{s-r}^{\lambda})y - (S_{t-q}^{\lambda} - S_{s-q}^{\lambda})y \|
\\ &\leq \lim_{\lambda \to \infty} e^{2 \omega T} \| y \|_{D(A_{\lambda}^2)} |t-s| \cdot |r-q|.
\end{align*}
Furthermore, by Lemma \ref{lemma-Yosida-approx}, Lemma \ref{lemma-restricted-semigroup} and (\ref{equiv-norm}) we have
\begin{align*}
\lim_{\lambda \to \infty} \| y \|_{D(A_{\lambda}^2)} &= \lim_{\lambda \to \infty} \big( \| y \| + \| A_{\lambda} y \| + \| A_{\lambda}^2 y \| \big)
\\ &= \| y \| + \| A y \| + \| A^2 y \|
\\ &\leq M \big( | y | + | A y | + | A^2 y | \big) = M |y|_{D(A)},
\end{align*}
completing the proof.
\end{proof}

\section{Compositions of regular paths with functions}\label{sec-comp-reg}

In this section we briefly deal with compositions of regular paths with functions. Let us fix a time horizon $T \in \bbr_+$, and let $W,\bar{W}$ be Banach spaces.

\begin{definition}
We denote by $\Lip([0,T] \times W, \bar{W})$ the space of all continuous functions $f_0 : [0,T] \times W \to \bar{W}$ such that for some constant $L > 0$ we have
\begin{align*}
| f_0(t,y) - f_0(t,z) | \leq L |y-z| \quad \text{for all $t \in [0,T]$ and $y,z \in W$.}
\end{align*}
\end{definition}

Note that for each $f_0 \in \Lip([0,T] \times W, \bar{W})$ the seminorm
\begin{align*}
\| f_0 \|_{\Lip} := \sup_{t \in [0,T]} |f(t,0)| + \sup_{t \in [0,T]} | f_0(t,\cdot) |_{\Lip}
\end{align*}
is finite.

\begin{remark}
Let $g_0 \in \Lip(W,\bar{W})$ be arbitrary, and define the function
\begin{align*}
f_0 : [0,T] \times W \to \bar{W}, \quad f_0(t,y) := g_0(y).
\end{align*}
Then we have $f_0 \in \Lip([0,T] \times W, \bar{W})$ with $\| f_0 \|_{\Lip} = |g_0(0)| + | g_0 |_{\Lip}$.
\end{remark}

\begin{lemma}\label{lemma-lin-growth}
Let $f_0 \in \Lip([0,T] \times W, \bar{W})$ be arbitrary. Then we have
\begin{align}\label{Lip-1}
| f_0(t,y) - f_0(t,z) | &\leq \| f_0 \|_{\Lip} |y-z|, \quad \text{$y,z \in W$ and $t \in [0,T]$,}
\\ \label{Lip-2} |f_0(t,y)| &\leq \| f_0 \|_{\Lip} (1 + |y|), \quad \text{$y \in W$ and $t \in [0,T]$.}
\end{align}
\end{lemma}

\begin{proof}
Setting $L := \sup_{t \in [0,T]} | f_0(t,\cdot) |_{\Lip}$, we have
\begin{align*}
| f_0(t,y) - f_0(t,z) | &\leq L |y-z|, \quad \text{$y,z \in W$ and $t \in [0,T]$,}
\end{align*}
showing (\ref{Lip-1}). Now, for all $y \in W$ and $t \in [0,T]$ we obtain
\begin{align*}
|f_0(t,y)| &\leq |f_0(t,y) - f_0(t,0)| + |f_0(t,0)|
\\ &\leq L |y| + \sup_{s \in [0,T]} |f_0(s,0)| = \| f_0 \|_{\Lip} (1 + |y|),
\end{align*}
proving (\ref{Lip-2}).
\end{proof}

For a path $Y : [0,T] \to W$ and a function $f_0 : [0,T] \times W \to \bar{W}$ we denote by $f_0(Y) : [0,T] \to \bar{W}$ the new path
\begin{align*}
f_0(Y)_t := f_0(t,Y_t), \quad t \in [0,T].
\end{align*}
If the path $Y$ is continuous and $f_0 \in \Lip([0,T] \times W, \bar{W})$, then the path $f_0(Y)$ is continuous, and hence bounded.

\section{Compositions of controlled rough paths with functions}\label{sec-comp-rough}

In this section we deal with compositions of controlled rough paths with functions. Let $V,W,\bar{W}$ be Banach spaces. We fix an index $\alpha \in (\frac{1}{3},\frac{1}{2}]$ and a time horizon $T \in \bbr_+$. Furthermore, let $\mathbf{X} = (X,\bbx) \in \scrc^{\alpha}([0,T],V)$ be a rough path.

\subsection{Compositions with time-dependent smooth functions}

In this subsection we consider compositions of controlled rough paths with time-dependent smooth functions.

\begin{definition}
Let $\gamma \in (0,1]$ and $n \in \bbn$ be arbitrary. We denote by $C_b^{\gamma,n}([0,T] \times W, \bar{W})$ the space of all functions $f : [0,T] \times W \to \bar{W}$ such that:
\begin{enumerate}
\item We have $f(t,\cdot) \in C_b^n(W,\bar{W})$ for all $t \in [0,T]$.

\item We have $f(\cdot,y),\ldots,D_y^{n-1}f(\cdot,y) \in \calc^{\gamma}([0,T],\bar{W})$ for each $y \in W$.

\item We have
\begin{align*}
\| f \|_{C_b^{\gamma,n}} := \sup_{t \in [0,T]} \| f(t,\cdot) \|_{C_b^n} + \sum_{k=0}^{n-1} \sup_{y \in W} \| D_y^k f(\cdot,y) \|_{\gamma} < \infty.
\end{align*}
\end{enumerate}
\end{definition}

\begin{remark}
Let $g \in C_b^n(W,\bar{W})$ be arbitrary, and define the function
\begin{align*}
f : [0,T] \times W \to \bar{W}, \quad f(t,y) := g(y).
\end{align*}
Then we have $f \in C_b^{\gamma,n}([0,T] \times W, \bar{W})$ with $\| f \|_{C_b^{\gamma,n}} = \| g \|_{C_b^n}$.
\end{remark}

\begin{proposition}\label{prop-comp-Hoelder-beta}
Let $\beta, \gamma \in (0,1]$ with $\beta \leq \gamma$ be arbitrary. Furthermore, let $Y \in \calc^{\beta}([0,T],W)$ and $f \in C_b^{\gamma,1}([0,T] \times W, \bar{W})$ be arbitrary. Then we have $f(Y) \in \calc^{\beta}([0,T],\bar{W})$.
\end{proposition}

\begin{proof}
Let $s,t \in [0,T]$ be arbitrary. Then we have
\begin{align*}
| f(t,Y_t) - f(s,Y_s) | &\leq | f(t,Y_t) - f(t,Y_s) | + | f(t,Y_s) - f(s,Y_s) |
\\ &\leq \| f(t,\cdot) \|_{C_b^1} | Y_t - Y_s | + \| f(\cdot,Y_s) \|_{\gamma} |t-s|^{\gamma}
\\ &\leq \| f \|_{C_b^{\gamma,1}} \| Y \|_{\beta} |t-s|^{\beta} + \| f \|_{C_b^{\gamma,1}} |t-s|^{\gamma}
\\ &\leq \| f \|_{C_b^{\gamma,1}} (\| Y \|_{\beta} + T^{\gamma-\beta}) |t-s|^{\beta},
\end{align*}
showing that $f(Y) \in \calc^{\beta}([0,T],\bar{W})$.
\end{proof}

For a controlled rough path $(Y,Y') \in \scrd_X^{2\alpha}([0,T],W)$ and a function $f \in C_b^{2\alpha,2}([0,T] \times W, \bar{W})$ we denote by $f(Y) : [0,T] \to \bar{W}$ the path
\begin{align*}
f(Y)_t := f(t,Y_t), \quad t \in [0,T],
\end{align*}
and we denote by $f(Y)' : [0,T] \to L(V,\bar{W})$ the path
\begin{align*}
f(Y)_t' := D_y f(t,Y_t) Y_t', \quad t \in [0,T].
\end{align*}

\begin{proposition}\label{prop-comp-f}
Let $(Y,Y') \in \scrd_X^{2\alpha}([0,T],W)$ and $f \in C_b^{2\alpha,2}([0,T] \times W, \bar{W})$ be arbitrary. Then we have
\begin{align}\label{comp-f-in-D}
(f(Y),f(Y)') \in \scrd_X^{2\alpha}([0,T],\bar{W}).
\end{align}
Furthermore, we have
\begin{align}\label{est-comp-Hoelder-1}
\| f(Y) \|_{\alpha} &\leq \| f \|_{C_b^{2\alpha,2}} ( \| Y \|_{\alpha} + T^{\alpha} ),
\\ \label{est-comp-Hoelder-2} \| f(Y)' \|_{\alpha} &\leq \| f \|_{C_b^{2\alpha,2}} \big( \| Y' \|_{\alpha} + \| Y \|_{\alpha} \| Y' \|_{\infty} + T^{\alpha} \| Y' \|_{\infty} \big),
\\ \label{est-comp-rem} \| R^{f(Y)} \|_{2\alpha} &\leq \| f \|_{C_b^{2\alpha,2}} \bigg( 1 + \frac{1}{2} \| Y \|_{\alpha}^2 + \| R^Y \|_{2\alpha} \bigg),
\end{align}
and for $T \leq 1$ we have the estimates
\begin{align}\label{comp-f-in-D-est}
| f(Y),f(Y)' |_{X,2\alpha} &\leq C \big( 1 + | Y,Y' |_{X,2\alpha}^2 \big),
\\ \label{comp-f-in-D-est-2} \interleave f(Y),f(Y)' \interleave_{X,2\alpha} &\leq C \big( 1 + \interleave Y,Y' \interleave_{X,2\alpha}^2 \big),
\end{align}
where the constant $C > 0$ depends on $\alpha$, $T$, $\mathbf{X}$ and $\| f \|_{C_b^{2\alpha,2}}$. Moreover, for $T \leq 1$ the constant $C$ does not depend on $T$.
\end{proposition}

\begin{proof}
Let $s,t \in [0,T]$ be arbitrary. Then we have
\begin{align*}
| f(t,Y_t) - f(s,Y_s) | &\leq | f(t,Y_t) - f(t,Y_s) | + | f(t,Y_s) - f(s,Y_s) |
\\ &\leq \| f(t,\cdot) \|_{C_b^2} | Y_t - Y_s | + \| f(\cdot,Y_s) \|_{2 \alpha} |t-s|^{2\alpha}
\\ &\leq \| f \|_{C_b^{2\alpha,2}} \| Y \|_{\alpha} |t-s|^{\alpha} + \| f \|_{C_b^{2\alpha,2}} |t-s|^{2\alpha}
\\ &\leq \| f \|_{C_b^{2\alpha,2}} (\| Y \|_{\alpha} + T^{\alpha}) |t-s|^{\alpha},
\end{align*}
showing $f(Y) \in \calc^{\alpha}([0,T],\bar{W})$ and the estimate (\ref{est-comp-Hoelder-1}). Furthermore, we have
\begin{align*}
&|D_y f(t,Y_t) Y_t' - D_y f(s,Y_s) Y_s'|
\\ &\leq |D_y f(t,Y_t) Y_t' - D_y f(t,Y_t) Y_s'| + |D_y f(t,Y_t) Y_s' - D_y f(t,Y_s) Y_s'|
\\ &\quad + |D_y f(t,Y_s) Y_s' - D_y f(s,Y_s) Y_s'|
\\ &\leq | D_y f(t,Y_s) | \, |Y_t' - Y_s'| + |D_y f(t,Y_t) - D_y f(t,Y_s)| \, |Y_s'|
\\ &\quad + |D_y f(t,Y_s) - D_y f(s,Y_s)| \, |Y_s'|
\\ &\leq \| f(t,\cdot) \|_{C_b^2} \| Y' \|_{\alpha} |t-s|^{\alpha} + \| f(t,\cdot) \|_{C_b^2} |Y_t-Y_s| \, \|Y'\|_{\infty}
\\ &\quad + \| D_y f(\cdot,Y_s) \|_{2 \alpha} |t-s|^{2\alpha} \|Y'\|_{\infty}
\\ &\leq \| f \|_{C_b^{2\alpha,2}} \| Y' \|_{\alpha} |t-s|^{\alpha} + \| f \|_{C_b^{2\alpha,2}} \| Y \|_{\alpha} |t-s|^{\alpha} \|Y'\|_{\infty} + \| f \|_{C_b^{2\alpha,2}} \|Y'\|_{\infty} |t-s|^{2\alpha}
\\ &\leq \| f \|_{C_b^{2\alpha,2}} \big( \| Y' \|_{\alpha} + \| Y \|_{\alpha} \| Y' \|_{\infty} + T^{\alpha} \| Y' \|_{\infty} \big) |t-s|^{\alpha},
\end{align*}
showing $f(Y)' \in \calc^{\alpha}([0,T],L(V,\bar{W}))$ and the estimate (\ref{est-comp-Hoelder-2}). Moreover, by Taylor's theorem we obtain
\begin{align*}
|R_{s,t}^{f(Y)}| &= | f(Y)_{s,t} - f(Y)_s' X_{s,t} |
\\ &= | f(t,Y_t) - f(s,Y_s) - D_y f(s,Y_s) Y_s' X_{s,t} |
\\ &\leq | f(t,Y_t) - f(s,Y_t) | + | f(s,Y_t) - f(s,Y_s) - D_y f(s,Y_s) Y_{s,t} |
\\ &\quad + | D_y f(s,Y_s) (Y_{s,t} - Y_s' X_{s,t}) |
\\ &\leq \| f(\bullet,Y_t) \|_{2 \alpha} |t-s|^{2\alpha} + \frac{1}{2} \| f(s,\bullet) \|_{C_b^2} |Y_t-Y_s|^2 + \| f(s,\bullet) \|_{C_b^2} |R_{s,t}^Y|
\\ &\leq \| f \|_{C_b^{2\alpha,2}} |t-s|^{2 \alpha} + \frac{1}{2} \| f \|_{C_b^{2\alpha,2}} \| Y \|_{\alpha}^2 |t-s|^{2\alpha} + \| f \|_{C_b^{2\alpha,2}} \| R^Y \|_{2\alpha} |t-s|^{2\alpha}
\\ &\leq \| f \|_{C_b^{2\alpha,2}} \bigg( 1 + \frac{1}{2} \| Y \|_{\alpha}^2 + \| R^Y \|_{2\alpha} \bigg) |t-s|^{2\alpha},
\end{align*}
proving (\ref{comp-f-in-D}) and (\ref{est-comp-rem}). Moreover, note that
\begin{align*}
|f(Y)_0| &= |f(0,Y_0)| \leq \| f \|_{C_b^{2\alpha,2}},
\\ |f(Y)_0'| &= |D_y f(0,Y_0) Y_0'| \leq |D_y f(0,Y_0)| \, |Y_0'| \leq \| f \|_{C_b^{2\alpha,2}} | Y,Y' |_{X,2\alpha}.
\end{align*}
Together with Lemma \ref{lemma-norm-of-Y} we obtain (\ref{comp-f-in-D-est}) and (\ref{comp-f-in-D-est-2}).
\end{proof}

\subsection{Compositions with bilinear operators}

In this subsection we consider compositions of controlled rough paths with bilinear operators. First of all, let us consider compositions with linear operators. For a controlled rough path $(Y,Y') \in \scrd_X^{2\alpha}([0,T],W)$ and a continuous linear operator $\varphi \in L(W,\bar{W})$ we denote by $\varphi(Y) : [0,T] \to \bar{W}$ the path
\begin{align*}
\varphi(Y)_t := \varphi(Y_t), \quad t \in [0,T],
\end{align*}
and we denote by $\varphi(Y)' : [0,T] \to L(V,\bar{W})$ the path
\begin{align*}
\varphi(Y)_t' := \varphi(Y_t'), \quad t \in [0,T].
\end{align*}

\begin{proposition}\label{prop-comp-linear-0}
Let $(Y,Y') \in \scrd_X^{2\alpha}([0,T],W)$ and $\varphi \in L(W,\bar{W})$ be arbitrary. Then we have
\begin{align}\label{comp-linear-0-1}
(\varphi(Y),\varphi(Y)') \in \scrd_X^{2\alpha}([0,T],\bar{W})
\end{align}
and the estimates
\begin{align}\label{comp-linear-0-2}
\| \varphi(Y), \varphi(Y)' \|_{X,2\alpha} &\leq | \varphi | \, \| Y,Y' \|_{X,2\alpha},
\\ \label{comp-linear-0-3} | \varphi(Y), \varphi(Y)' |_{X,2\alpha} &\leq | \varphi | \, | Y,Y' |_{X,2\alpha},
\\ \label{comp-linear-0-4} \interleave \varphi(Y), \varphi(Y)' \interleave_{X,2\alpha} &\leq | \varphi | \, \interleave Y,Y' \interleave_{X,2\alpha}.
\end{align}
\end{proposition}

\begin{proof}
Let $s,t \in [0,T]$ be arbitrary. Then we have
\begin{align*}
| \varphi(Y_t') - \varphi(Y_s') | = |\varphi(Y_t' - Y_s')| \leq | \varphi | \, |Y_t' - Y_s'| \leq | \varphi | \, \| Y' \|_{\alpha} |t-s|^{\alpha},
\end{align*}
showing $\| \varphi(Y') \|_{\alpha} \leq | \varphi | \, \| Y' \|_{\alpha}$. Furthermore, we have
\begin{align*}
R_{s,t}^{\varphi(Y)} = \varphi(Y_{s,t}) - \varphi(Y_s') X_{s,t} = \varphi(Y_{s,t} - Y_s' X_{s,t}) = \varphi(R_{s,t}^Y),
\end{align*}
and hence, we obtain
\begin{align*}
| R_{s,t}^{\varphi(Y)} | \leq | \varphi | \, |R_{s,t}^Y| \leq | \varphi | \, \| R^Y \|_{2\alpha} |t-s|^{2 \alpha}.
\end{align*}
This shows $\| R^{\varphi(Y)} \|_{2 \alpha} \leq | \varphi | \, \| R^Y \|_{2 \alpha}$, and we deduce (\ref{comp-linear-0-1}). Noting that $|\varphi(Y_0)| \leq | \varphi | \, |Y_0|$ and $|\varphi(Y_0')| \leq | \varphi | \, |Y_0'|$, the desired estimates (\ref{comp-linear-0-2})--(\ref{comp-linear-0-4}) follow.
\end{proof}

Now, let $E$ be another Banach space. For controlled rough paths $(Y,Y') \in \scrd_X^{2\alpha}([0,T],W)$, $(Z,Z') \in \scrd_X^{2\alpha}([0,T],E)$ and a continuous bilinear operator $B \in L^{(2)}(W \times E,\bar{W})$ we denote by $B(Y,Z) : [0,T] \to \bar{W}$ the path
\begin{align*}
B(Y,Z)_t := B(Y_t,Z_t), \quad t \in [0,T],
\end{align*}
and we denote by $B(Y,Z)' : [0,T] \to L(V,\bar{W})$ the path
\begin{align*}
B(Y,Z)_t' := B(Y_t,Z_t') + B(Y_t',Z_t), \quad t \in [0,T].
\end{align*}

\begin{proposition}\label{prop-bilinear}
Let $(Y,Y') \in \scrd_X^{2\alpha}([0,T],W)$, $(Z,Z') \in \scrd_X^{2\alpha}([0,T],E)$ and $B \in L^{(2)}(W \times E,\bar{W})$ be arbitrary. Then we have
\begin{align}\label{bilinear-controlled}
(B(Y,Z),B(Y,Z)') \in \scrd_X^{2\alpha}([0,T],\bar{W}),
\end{align}
and the estimate
\begin{align}\label{est-bilinear}
\interleave B(Y,Z),B(Y,Z)' \interleave_{X,2\alpha} \leq C \interleave Y,Y' \interleave_{X,2\alpha} \interleave Z,Z' \interleave_{X,2\alpha},
\end{align}
where the constant $C > 0$ depends on $\alpha$, $T$, $\mathbf{X}$ and $| B |$. Moreover, for $T \leq 1$ the constant $C$ does not depend on $T$.
\end{proposition}

\begin{proof}
Let $s,t \in [0,T]$ be arbitrary. Then we have
\begin{align*}
| B(Y_t,Z_t) - B(Y_s,Z_s) | &\leq |B(Y_t,Z_t) - B(Y_t,Z_s)| + |B(Y_t,Z_s) - B(Y_s,Z_s)|
\\ &= |B(Y_t,Z_{s,t})| + |B(Y_{s,t},Z_s)|
\\ &\leq | B | \big( |Y_t| \, |Z_{s,t}| + |Y_{s,t}| \, |Z_s| \big)
\\ &\leq | B | \big( \| Y \|_{\infty} \| Z \|_{\alpha} + \| Y \|_{\alpha} \| Z \|_{\infty} \big) |t-s|^{\alpha},
\end{align*}
showing that $B(Y,Z) \in C^{\alpha}([0,T],\bar{W})$. Furthermore, we have
\begin{align*}
| B(Y_t,Z_t') - B(Y_s,Z_s') | &\leq | B(Y_t,Z_t') - B(Y_t,Z_s') | + | B(Y_t,Z_s') - B(Y_s,Z_s') |
\\ &= | B(Y_t,Z_{s,t}') | + | B(Y_{s,t},Z_s') |
\\ &\leq | B | \big( |Y_t| \, |Z_{s,t}'| + |Y_{s,t}| \, |Z_s'| \big)
\\ &\leq | B \| \big( \| Y \|_{\infty} \| Z' \|_{\alpha} + \| Y \|_{\alpha} \| Z' \|_{\infty} \big) |t-s|^{\alpha}
\end{align*}
as well as
\begin{align*}
| B(Y_t',Z_t) - B(Y_s',Z_s) | &\leq | B(Y_t',Z_t) - B(Y_t',Z_s) | + | B(Y_t',Z_s) - B(Y_s',Z_s) |
\\ &= | B(Y_t',Z_{s,t}) | + | B(Y_{s,t}',Z_s) |
\\ &\leq | B | \big( |Y_t'| \, |Z_{s,t}| + |Y_{s,t}'| \, |Z_s| \big)
\\ &\leq | B | \big( \| Y' \|_{\infty} \| Z \|_{\alpha} + \| Y' \|_{\alpha} \| Z \|_{\infty} \big) |t-s|^{\alpha},
\end{align*}
showing that $B(Y,Z)' \in C^{\alpha}([0,T],L(V,\bar{W})$. Moreover, we have
\begin{align*}
R_{s,t}^{B(Y,Z)} &= B(Y_t,Z_t) - B(Y_s,Z_s) - B(Y_s,Z_s)' X_{s,t}
\\ &= B(Y_t,Z_{s,t}) + B(Y_{s,t},Z_s) - \big( B(Y_s,Z_s' X_{s,t}) + B(Y_s' X_{s,t}, Z_s) \big)
\\ &= B(Y_{s,t}-Y_s' X_{s,t},Z_s) + B(Y_s,Z_{s,t}- Z_s' X_{s,t}) + B(Y_{s,t},Z_{s,t})
\\ &= B(R_{s,t}^Y,Z_s) + B(Y_s,R_{s,t}^Z) + B(Y_{s,t},Z_{s,t}),
\end{align*}
and hence
\begin{align*}
| R_{s,t}^{B(Y,Z)} | &\leq | B | \big( |R_{s,t}^Y| \, |Z_s| + |Y_s| \, | R_{s,t}^Z | + |Y_{s,t}| \, |Z_{s,t}| \big)
\\ &\leq | B | \big( \| R^Y \|_{2 \alpha} \| Z \|_{\infty} + \| Y \|_{\infty} \| R^Y \|_{2 \alpha} + \| Y \|_{\alpha} \| Z \|_{\alpha} \big) |t-s|^{2 \alpha},
\end{align*}
which shows (\ref{bilinear-controlled}). Finally, noting that
\begin{align*}
|B(Y_0,Z_0)| &\leq | B | \, |Y_0| \, |Z_0|,
\\ |B(Y_0,Z_0')| &\leq | B | \, |Y_0| \, |Z_0'|,
\\ |B(Y_0',Z_0)| &\leq | B | \, |Y_0'| \, |Z_0|,
\end{align*}
applying Lemma \ref{lemma-norm-of-Y} proves (\ref{est-bilinear}).
\end{proof}

\begin{lemma}\label{lemma-chain-rule}
Let $E,F$ be two other Banach spaces. Furthermore, let $h : W \to L(W,F)$ be of class $C^1$, and let $B \in L(E,W)$ be arbitrary. Then the mapping
\begin{align*}
h_B : W \to L(E,F), \quad h_B(y) := h(y) B
\end{align*}
is of class $C^1$, and we have
\begin{align*}
D h_B(y)v = (D h(y)v) B \quad \text{for all $y,v \in W$.}
\end{align*}
\end{lemma}

\begin{proof}
We have $h_B = \ell \circ h$, where $\ell : L(W,F) \to L(E,F)$ is given by $\ell(z) = zB$. Note that $\ell$ is a linear operator. Moreover, we have
\begin{align*}
| \ell(z) | \leq | z | \, |B| \quad \text{for all $z \in L(W,F)$,}
\end{align*}
showing that $\ell$ is continuous. Therefore, the mapping $h_B$ is of class $C^1$, and by the chain rule we obtain
\begin{align*}
D h_B(y)v &= D (\ell \circ h)(y)v = D \ell(h(y)) \circ D h(y) v
\\ &= \ell \circ D h(y) v = (D h(y)v) B,
\end{align*}
completing the proof.
\end{proof}

\begin{lemma}\label{lemma-bilinear-fg-time}
Let $f \in C_b^{2 \alpha,3}([0,T] \times W,\bar{W})$ be arbitrary. Then there exists $g \in C_b^{2 \alpha,2}([0,T] \times (W \times W), L(W,\bar{W}))$ with $\| g \|_{C_b^{2 \alpha,2}} \leq \| f \|_{C_b^{2 \alpha,3}}$ such that
\begin{align}\label{identity-Taylor}
f(t,y_1) - f(t,y_2) = g(t,y) (y_1-y_2)
\end{align}
for all $t \in [0,T]$ and all $y = (y_1,y_2) \in W \times W$.
\end{lemma}

\begin{proof}
We define $g : [0,T] \times W \times W \to L(W,\bar{W})$ as
\begin{align*}
g(t,y) := \int_0^1 D_2 f(t,\theta y_1 + (1-\theta)y_2) \, d\theta.
\end{align*}
Then by Taylor's theorem identity (\ref{identity-Taylor}) is satisfied for all $t \in [0,T]$ and all $y = (y_1,y_2) \in W \times W$. For $\theta \in [0,1]$ we consider the linear operator $B_{\theta} : W \times W \to W$ given by
\begin{align*}
B_{\theta}(y) := \theta y_1 + (1-\theta) y_2, \quad y = (y_1,y_2) \in W \times W.
\end{align*}
Then for all $y = (y_1,y_2) \in W \times W$ we have
\begin{align*}
|B_{\theta}(y)| = |\theta y_1 + (1-\theta)y_2| \leq \theta |y_1| + (1-\theta) |y_2| \leq |y_1| + |y_2| = |y|.
\end{align*}
Therefore, for each $\theta \in [0,1]$ we have $B_{\theta} \in L(W \times W,W)$ with
\begin{align}\label{B-norm}
|B_{\theta}| \leq 1 \quad \text{for all $\theta \in [0,1]$.}
\end{align}
Now, we fix an arbitrary $t \in [0,T]$. Then we have
\begin{align*}
g(t,y) = \int_0^1 D_2 f(t,B_{\theta}(y)) \, d\theta, \quad y \in W \times W.
\end{align*}
Therefore, we have $g(t,\cdot) \in C^2(W \times W,L(W,\bar{W}))$. Let $y \in W \times W$ be arbitrary. By the chain rule we have
\begin{align*}
D_y g(t,y) &= \int_0^1 D_y \big( D_2 f(t,B_{\theta}(y)) \big) \, d\theta = \int_0^1 D_2^2 f(t,B_{\theta}(y)) B_{\theta} \, d\theta,
\end{align*}
Moreover, by Lemma \ref{lemma-chain-rule} and the chain rule for all $v \in W \times W$ we obtain
\begin{align*}
D_y^2 g(t,y)v &= \int_0^1 D_y \big( D_2^2 f(t,B_{\theta}(y)) B_{\theta} \big) v \, d\theta
\\ &= \int_0^1 \big( D_2^3 f(t,B_{\theta}(y)) B_{\theta}(v) \big) B_{\theta} \, d\theta.
\end{align*}
Therefore, noting (\ref{B-norm}) we have $g(t,\cdot) \in C_b^2(W \times W,L(W,\bar{W}))$ with
\begin{align}\label{g-fct-1}
\| D_y^k g(t,\cdot) \|_{\infty} \leq \| D_y^{k+1} f(t,\cdot) \|_{\infty}, \quad k=0,1,2.
\end{align}
Now, let $y \in W \times W$ be arbitrary. Furthermore, let $s,t \in [0,T]$ be arbitrary. Then we have
\begin{align*}
|g(t,y)-g(s,y)| &\leq \int_0^1 |D_2 f(t,B_{\theta}(y)) - D_2 f(s,B_{\theta}(y)) | \, d\theta
\\ &\leq \int_0^1 \| D_2 f(\cdot,B_{\theta}(y)) \|_{2\alpha} |t-s|^{2\alpha} \, d\theta.
\end{align*}
Moreover, by (\ref{B-norm}) we obtain
\begin{align*}
|D_y g(t,y) - D_y g(s,y)| &\leq \int_0^1 |D_2^2 f(t,B_{\theta}(y)) B_{\theta} - D_2^2 f(s,B_{\theta}(y)) B_{\theta} | \, d\theta
\\ &\leq \int_0^1 |D_2^2 f(t,B_{\theta}(y)) - D_2^2 f(s,B_{\theta}(y)) | \, d\theta
\\ &\leq \int_0^1 \| D_2^2 f(\cdot,B_{\theta}(y)) \|_{2\alpha} |t-s|^{2\alpha} \, d\theta.
\end{align*}
Consequently, we have $g(\cdot,y), D_y g(\cdot,y) \in \calc^{2\alpha}([0,T],W)$ with
\begin{align}\label{g-fct-2}
\sup_{y \in W} \| D_y^k g(\cdot,y) \|_{2\alpha} \leq \sup_{y \in W} \| D_y^{k+1} f(\cdot,y) \|_{2\alpha}, \quad k=0,1.
\end{align}
From (\ref{g-fct-1}) and (\ref{g-fct-2}) we obtain $\| g \|_{C_b^{2\alpha,2}} \leq \| f \|_{C_b^{2 \alpha,3}} < \infty$, which completes the proof.
\end{proof}

\begin{proposition}\label{prop-diff-f}
Let $(Y,Y'), (Z,Z') \in \scrd_X^{2\alpha}([0,T],W)$ and $f \in C_b^{2\alpha,3}([0,T] \times W, \bar{W})$ be arbitrary. Then we have
\begin{align*}
\interleave (f(Y),f(Y)')-(f(Z),f(Z)') \interleave_{X,2\alpha} &\leq C \big( 1 + | Y,Y' |_{X,2\alpha}^2 + | Z,Z' |_{X,2\alpha}^2 \big)
\\ &\quad \interleave (Y,Y') - (Z,Z') \interleave_{X,2\alpha},
\end{align*}
where the constant $C > 0$ depends on $\alpha$, $T$, $\mathbf{X}$ and $\| f \|_{C_b^{2 \alpha,3}}$. Moreover, for $T \leq 1$ the constant $C$ does not depend on $T$.
\end{proposition}

\begin{proof}
According to Lemma \ref{lemma-bilinear-fg-time} there exists a mapping $g \in C_b^{2 \alpha,2}([0,T] \times (W \times W), L(W,\bar{W}))$ with
\begin{align}\label{g-f-norms}
\| g \|_{C_b^{2 \alpha,2}} \leq \| f \|_{C_b^{2 \alpha,3}}
\end{align}
such that
\begin{align*}
f(t,y) - f(t,z) = g(t,y,z)(y-z), \quad t \in [0,T] \text{ and } y,z \in W.
\end{align*}
We define the bilinear operator $B : L(W,\bar{W}) \times W \to \bar{W}$ as
\begin{align*}
B(T,y) := Ty.
\end{align*}
Then for all $(T,y) \in L(W,\bar{W}) \times W$ we have
\begin{align*}
|B(T,y)| \leq |T| \, |y|.
\end{align*}
Therefore, we have $B \in L^{(2)}(L(W,\bar{W}) \times W, \bar{W})$ with $|B| \leq 1$. Hence, using Proposition \ref{prop-bilinear} we obtain
\begin{align*}
&\interleave (f(Y),f(Y)')-(f(Z),f(Z)') \interleave_{X,2\alpha}
\\ &= \interleave g(Y,Z)(Y-Z), (g(Y,Z)(Y-Z))' \interleave_{X,2\alpha}
\\ &= \interleave B(g(Y,Z),Y-Z), (B(g(Y,Z),Y-Z))' \interleave_{X,2\alpha}
\\ &\leq C \interleave g(Y,Z), g(Y,Z)' \interleave_{X,2\alpha} \interleave (Y,Y') - (Z,Z') \interleave_{X,2\alpha},
\end{align*}
where the constant $C > 0$ depends on $\alpha$, $T$ and $\mathbf{X}$, and does not depend on $T \leq 1$. Furthermore, by Proposition \ref{prop-comp-f} and Lemma \ref{lemma-pairs} we have
\begin{align*}
\interleave g(Y,Z), g(Y,Z)' \interleave_{X,2\alpha} &= |g(0,Y_0,Z_0)| + |g(Y,Z), g(Y,Z)'|_{X,2\alpha}
\\ &\leq C ( 1 + | (Y,Z), (Y,Z)' |_{X,2\alpha}^2 )
\\ &\leq C ( 1 + |Y,Y'|_{X,2\alpha}^2 + |Z,Z'|_{X,2\alpha}^2 ),
\end{align*}
where the constant $C > 0$ depends on $\alpha$, $T$, $\mathbf{X}$ and $\| g \|_{C_b^{2 \alpha,2}}$, and does not depend on $T \leq 1$. Taking into account (\ref{g-f-norms}), this completes the proof.
\end{proof}

\subsection{Compositions with time-dependent linear operators}

In this subsection we consider compositions of controlled rough paths with time-dependent linear operators. For a controlled rough path $(Y,Y') \in \scrd_X^{2 \alpha}([0,T],W)$ and a mapping $\varphi : [0,T] \times W \to \bar{W}$ such that $\varphi(t,\cdot) \in L(W,\bar{W})$ for each $t \in [0,T]$ we denote by $\varphi(Y) : [0,T] \to \bar{W}$ the path
\begin{align*}
\varphi(Y)_t := \varphi(t,Y_t), \quad t \in [0,T],
\end{align*}
and we denote by $\varphi(Y)' : [0,T] \to L(V,\bar{W})$ the path
\begin{align*}
\varphi(Y)_t' := \varphi(t,Y_t'), \quad t \in [0,T].
\end{align*}

\begin{proposition}\label{prop-comp-linear}
Let $(Y,Y') \in \scrd_X^{2 \alpha}([0,T],W)$ be arbitrary, and let $\varphi : [0,T] \times W \to \bar{W}$ be a function such that $\varphi(t,\cdot) \in L(W,\bar{W})$ for each $t \in [0,T]$. We assume there are constants $K,L > 0$ such that
\begin{align*}
&| \varphi(t,\cdot) | \leq K, \quad t \in [0,T],
\\ &| \varphi(t,y) - \varphi(s,y) | \leq L |y| \, |t-s|^{2\alpha}, \quad s,t \in [0,T] \text{ and } y \in W.
\end{align*}
Then we have
\begin{align}\label{comp-linear-1}
(\varphi(Y),\varphi(Y)') \in \scrd_X^{2 \alpha}([0,T],\bar{W}).
\end{align}
Furthermore, we have
\begin{align}\label{comp-linear-2}
\| \varphi(Y)' \|_{\alpha} &\leq L T^{\alpha} \| Y' \|_{\infty} + K \| Y' \|_{\alpha},
\\ \label{comp-linear-3} \| R^{\varphi(Y)} \|_{2 \alpha} &\leq L \| Y \|_{\infty} + K \| R^Y \|_{2 \alpha},
\end{align}
and we have the estimate
\begin{align}\label{comp-linear-4}
\interleave \varphi(Y),\varphi(Y)' \interleave_{X,2\alpha} \leq C \interleave Y,Y' \interleave_{X,2 \alpha},
\end{align}
where the constant $C > 0$ depends on $\alpha$, $T$, $\mathbf{X}$ and $K$, $L$. Moreover, for $T \leq 1$ the constant $C$ does not depend on $T$.
\end{proposition}

\begin{proof}
Let $s,t \in [0,T]$ be arbitrary. Then we have
\begin{align*}
| \varphi(t,Y_t) - \varphi(s,Y_s) | &\leq | \varphi(t,Y_t) - \varphi(t,Y_s) | + | \varphi(t,Y_s) - \varphi(s,Y_s) |
\\ &\leq | \varphi(t,\cdot) | \, | Y_t - Y_s | + L |Y_s| \, |t-s|^{2 \alpha}
\\ &\leq M \| Y \|_{\alpha} |t-s|^{\alpha} + L \| Y \|_{\infty} |t-s|^{2 \alpha}
\\ &\leq (M \| Y \|_{\alpha} + L T^{\alpha} \| Y \|_{\infty}) |t-s|^{\alpha},
\end{align*}
and hence $\varphi(Y) \in C^{\alpha}([0,T],\bar{W})$. Furthermore, we have
\begin{align*}
| \varphi(t,Y_t') - \varphi(s,Y_s') | &\leq | \varphi(t,Y_t') - \varphi(t,Y_s') | + | \varphi(t,Y_s') - \varphi(s,Y_s') |
\\ &\leq | \varphi(t,\cdot) | \, | Y_t' - Y_s' | + L |Y_s'| \, |t-s|^{2\alpha}
\\ &\leq M \| Y' \|_{\alpha} |t-s|^{\alpha} + L \| Y' \|_{\infty} |t-s|^{2 \alpha}
\\ &\leq (M \| Y' \|_{\alpha} + L T^{\alpha} \| Y' \|_{\infty}) |t-s|^{\alpha},
\end{align*}
showing $\varphi(Y)' \in C^{\alpha}([0,T],L(V,\bar{W}))$ and the estimate (\ref{comp-linear-2}). Moreover, we have
\begin{align*}
R_{s,t}^{\varphi(Y)} &= \varphi(Y_{s,t}) - \varphi(Y_s') X_{s,t}
\\ &= \varphi(t,Y_t) - \varphi(s,Y_s) - \varphi(s,Y_s' X_{s,t})
\\ &= \varphi(t,Y_t) - \varphi(s,Y_t) + \varphi(s,Y_t) - \varphi(s,Y_s) - \varphi(s,Y_s' X_{s,t})
\\ &= \varphi(t,Y_t) - \varphi(s,Y_t) + \varphi(s,Y_{s,t} - Y_s' X_{s,t}).
\end{align*}
Therefore, we obtain
\begin{align*}
| R_{s,t}^{\varphi(Y)} | &\leq | \varphi(t,Y_t) - \varphi(s,Y_t) | + | \varphi(s,Y_{s,t} - Y_s' X_{s,t}) |
\\ &\leq L \| Y \|_{\infty} |t-s|^{2\alpha} + \| \varphi(s,\cdot) \| \, | R_{s,t}^Y |
\\ &\leq L \| Y \|_{\infty} |t-s|^{2\alpha} + M \| R^Y \|_{2 \alpha} |s-t|^{2 \alpha}
\\ &= \big( L \| Y \|_{\infty} + M \| R^Y \|_{2 \alpha} \big) |s-t|^{2 \alpha},
\end{align*}
proving (\ref{comp-linear-1}) and (\ref{comp-linear-3}). Furthermore, note that
\begin{align*}
|\varphi(Y)_0| &= |\varphi(0,Y_0)| \leq |\varphi(0,\cdot)| \, |Y_0| \leq M |Y_0|,
\\ |\varphi(Y)_0'| &= |\varphi(0,Y_0')| \leq |\varphi(0,\cdot)| \, |Y_0'| \leq M |Y_0'|.
\end{align*}
Hence, the estimate (\ref{comp-linear-4}) is a consequence of Lemma \ref{lemma-norm-of-Y}.
\end{proof}

For what follows, the upcoming auxiliary result will be useful.

\begin{lemma}\label{lemma-embedding-operators}
Let $E$ be another Banach space, and let $\varphi \in L(E,W)$ be arbitrary. Setting
\begin{align*}
L(V,E) \to L(V,W), \quad S \mapsto \varphi S,
\end{align*}
we may regard $\varphi$ as an element from $L(L(V,E),L(V,W))$. Furthermore, we have
\begin{align*}
| \varphi |_{L(L(V,E),L(V,W))} \leq | \varphi |_{L(E,W)}.
\end{align*}
\end{lemma}

\begin{proof}
We have
\begin{align*}
| \varphi |_{L(L(V,E),L(V,W))} &= \sup_{| S | \leq 1} | \varphi S |_{L(V,W)}
\\ &\leq \sup_{| S | \leq 1} | \varphi |_{L(E,W)} | S |_{L(V,E)} \leq | \varphi |_{L(E,W)},
\end{align*}
completing the proof.
\end{proof}

Now, let $A$ be the generator of a $C_0$-semigroup $(S_t)_{t \geq 0}$ on $W$. Then there are constants $M \geq 1$ and $\omega \in \bbr$ such that the estimate (\ref{est-semigroup}) is satisfied.

\begin{proposition}\label{prop-conv-well-defined}
Let $(Y,Y') \in \scrd_X^{2 \alpha}([0,T],L(V,D(A)))$ be arbitrary, and let $t \in [0,T]$ be arbitrary. We define the paths $Z : [0,t] \to L(V,W)$ and $Z' : [0,t] \to L(V,L(V,W))$ as
\begin{align*}
Z_s &:= S_{t-s} Y_s, \quad s \in [0,t],
\\ Z_s' &:= S_{t-s} Y_s', \quad s \in [0,t].
\end{align*}
Then we have
\begin{align*}
(Z,Z') \in \scrd_X^{2 \alpha}([0,t],L(V,W)).
\end{align*}
Furthermore, we have
\begin{align*}
\interleave Z,Z' \interleave_{X,2\alpha} \leq C \interleave Y,Y' \interleave_{X,2\alpha},
\end{align*}
where the constant $C > 0$ depends on $\alpha$, $T$, $\mathbf{X}$ and $M$, $\omega$. Moreover, for $T \leq 1$ the constant $C$ does not depend on $T$.
\end{proposition}

\begin{proof}
In view of Lemma \ref{lemma-embedding-operators}, we can define
\begin{align*}
\varphi : [0,t] \times L(V,D(A)) \to L(V,W), \quad \varphi(s,y) := S_{t-s} y.
\end{align*}
Let $s \in [0,t]$ be arbitrary. By Lemma \ref{lemma-embedding-operators} we have
\begin{align*}
\varphi(s,\cdot) \in L(L(V,D(A)), L(V,W)).
\end{align*}
Furthermore, by (\ref{est-semigroup}) for each $y \in L(V,D(A))$ we have
\begin{align*}
|\varphi(s,y)|_{L(V,W)} = |S_{t-s} y|_{L(V,W)} &\leq |S_{t-s}|_{L(D(A),W)} |y|_{L(V,D(A))}
\\ &\leq M e^{\omega T} |y|_{L(V,D(A))},
\end{align*}
showing that
\begin{align*}
|\varphi(s,\cdot)| \leq M e^{\omega T}.
\end{align*}
Now, let $s,t \in [0,T]$ and $y \in L(V,D(A))$ be arbitrary. Then by Corollary \ref{cor-orbit-map-2} we have
\begin{align*}
| \varphi(s,y) - \varphi(r,y) |_{L(V,W)} &= |S_{t-s} y - S_{t-r} y|_{L(V,W)} = |S_{t-r,t-s} y|_{L(V,W)}
\\ &\leq M e^{\omega T} |y|_{L(V,D(A))} |s-r|.
\end{align*}
Hence the statement follows from Proposition \ref{prop-comp-linear}.
\end{proof}

Let $(Y,Y') \in \scrd_X^{2 \alpha}([0,T],L(V,D(A)))$ be a controlled rough path. Then Proposition \ref{prop-conv-well-defined} allows us to define the \emph{rough convolution} $N : [0,T] \to W$ as follows. For each $t \in [0,T]$ let $N_t$ be the Gubinelli integral
\begin{align*}
N_t := \int_0^t S_{t-s} Y_s \, d\mathbf{X}_s
\end{align*}
according to Theorem \ref{thm-Gubinelli}. We will investigate this rough convolution further in Section \ref{sec-rough-conv}.

\begin{lemma}\label{lemma-conv-1-pre}
Let $(Y,Y') \in \scrd_X^{2 \alpha}([0,T],L(V,D(A^2)))$ be arbitrary, and let $s,t \in [0,T]$ with $s \leq t$ be arbitrary. We define the paths $Z : [s,t] \to L(V,W)$ and $Z' : [s,t] \to L(V,L(V,W))$ as
\begin{align*}
Z_r &:= (S_{t-r} - \Id) Y_r, \quad r \in [s,t],
\\ Z_r' &:= (S_{t-r} - \Id) Y_r', \quad t \in [s,t].
\end{align*}
Then we have
\begin{align}\label{conv-1}
(Z,Z') \in \scrd_X^{2 \alpha}([s,t],L(V,W)).
\end{align}
Furthermore, we have
\begin{align}\label{conv-1a}
\| Z' \|_{\alpha} &\leq M e^{\omega T} \big( T^{1-\alpha} \| Y' \|_{\infty} + \| Y' \|_{\alpha} |t-s| \big)
\\ \label{conv-1b} \| R^Z \|_{2\alpha} &\leq M e^{\omega T} \big( T^{1 - 2\alpha} \| Y  \|_{\infty} + \| R^Y \|_{2 \alpha} |t-s| \big),
\\ \label{conv-1c} \| Z \|_{\infty} &\leq M e^{\omega T} \| Y \|_{\infty} |t-s|,
\\ \label{conv-1d} \| Z' \|_{\infty} &\leq M e^{\omega T} \| Y' \|_{\infty} |t-s|.
\end{align}
\end{lemma}

\begin{proof}
We define $\varphi : [s,t] \times L(V,D(A^2)) \to L(V,W)$ as
\begin{align*}
\varphi(r,y) := (S_{t-r} - \Id) y.
\end{align*}
Let $r \in [s,t]$ be arbitrary. Then we have
\begin{align*}
\varphi(r,\cdot) \in L(L(V,D(A^2)),L(V,W)).
\end{align*}
Furthermore, by Corollary \ref{cor-orbit-map-1} for each $y \in L(V,D(A^2))$ we have
\begin{align*}
|\varphi(r,y)|_{L(V,W)} &= |(S_{t-r} - \Id) y|_{L(V,W)}
\\ &\leq |S_{t-r} - \Id|_{L(D(A^2),W)} |y|_{L(V,D(A^2))}
\\ &\leq |S_{0,t-r}|_{L(D(A),W)} |y|_{L(V,D(A^2))}
\\ &\leq M e^{\omega T} |t-s| \, |y|_{L(V,D(A^2))},
\end{align*}
showing that
\begin{align*}
| \varphi(r,\cdot) | \leq M e^{\omega T} |t-s|.
\end{align*}
Furthermore, we obtain
\begin{align*}
| Z_r |_{L(V,W)} = |\varphi(r,Y_r)|_{L(V,W)} \leq M e^{\omega T} |t-s| \, |Y_r|_{L(V,D(A^2))},
\end{align*}
showing (\ref{conv-1c}). Now, let $y \in L(V,D(A^2))$ and $r,q \in [s,t]$ be arbitrary. Then by Corollary \ref{cor-orbit-map-2} we have
\begin{align*}
| \varphi(r,y) - \varphi(q,y) |_{L(V,W)} &= |S_{t-r} y - S_{t-q} y|_{L(V,W)}
\\ &= | S_{t-q,t-r} y |_{L(V,W)}
\\ &\leq M e^{\omega T} |y|_{L(V,D(A^2))} |r-q|
\\ &\leq M e^{\omega T} |y|_{L(V,D(A^2))} T^{1-2\alpha} |r-q|^{2\alpha}.
\end{align*}
Hence, from Proposition \ref{prop-comp-linear} we obtain (\ref{conv-1}) and (\ref{conv-1a}), (\ref{conv-1b}). Now, we define $\Phi : [s,t] \times L(V,L(V,D(A^2))) \to L(V,L(V,W))$ as
\begin{align*}
\Phi(r,y') := (S_{t-r} - \Id) y'.
\end{align*}
Let $r \in [s,t]$ be arbitrary. Then by Lemma \ref{lemma-embedding-operators} and Corollary \ref{cor-orbit-map-1} for each $y' \in L(V,L(V,D(A^2)))$ we have
\begin{align*}
|\Phi(r,y')|_{L(V,L(V,W))} &= |(S_{t-r} - \Id) y'|_{L(V,L(V,W))}
\\ &\leq |S_{t-r} - \Id|_{L(L(V,D(A^2)),L(V,W))} |y'|_{L(V,L(V,D(A^2)))}
\\ &\leq |S_{t-r} - \Id|_{L(D(A^2),W)} |y'|_{L(V,L(V,D(A^2)))}
\\ &\leq |S_{0,t-r}|_{L(D(A),W)} |y'|_{L(V,L(V,D(A^2)))}
\\ &\leq M e^{\omega T} |t-s| \, |y'|_{L(V,L(V,D(A^2)))}.
\end{align*}
Therefore, we obtain
\begin{align*}
| Z_r' |_{L(V,L(V,W))} = |\Phi(r,Y_r')|_{L(V,L(V,W))} \leq M e^{\omega T} |t-s| \, |Y_r'|_{L(V,L(V,D(A^2)))},
\end{align*}
showing (\ref{conv-1d}).
\end{proof}

\begin{lemma}\label{lemma-conv-2-pre}
Let $(Y,Y') \in \scrd_X^{2 \alpha}([0,T],L(V,D(A^2)))$ be arbitrary, and let $s,t \in [0,T]$ with $s \leq t$ be arbitrary. We define the paths $Z : [0,s] \to L(V,W)$ and $Z' : [0,s] \to L(V,L(V,W))$ as
\begin{align*}
Z_r &:= (S_{t-r} - S_{s-r}) Y_r, \quad r \in [0,s],
\\ Z_r' &:= (S_{t-r} - S_{s-r}) Y_r', \quad r \in [0,s].
\end{align*}
Then we have
\begin{align}\label{conv-2}
(Z,Z') \in \scrd_X^{2 \alpha}([0,s],L(V,W)).
\end{align}
Furthermore, we have
\begin{align}\label{conv-2a}
\| Z' \|_{\alpha} &\leq M e^{\omega T} \big( T^{1-\alpha} e^{\omega T} \| Y' \|_{\infty} + \| Y' \|_{\alpha} \big) |t-s|,
\\ \label{conv-2b} \| R^Z \|_{2\alpha} &\leq M e^{\omega T} \big( T^{1-2\alpha} e^{\omega T} \| Y  \|_{\infty} + \| R^Y \|_{2 \alpha} \big) |t-s|,
\\ \label{conv-2c} \| Z \|_{\infty} &\leq M e^{\omega T} \| Y \|_{\infty} |t-s|,
\\ \label{conv-2d} \| Z' \|_{\infty} &\leq M e^{\omega T} \| Y' \|_{\infty} |t-s|.
\end{align}
\end{lemma}

\begin{proof}
We define $\varphi : [0,s] \times L(V,D(A^2)) \to L(V,W)$ as
\begin{align*}
\varphi(r,y) := S_{s-r,t-r} y.
\end{align*}
Let $r \in [0,s]$ be arbitrary. Then we have
\begin{align*}
\varphi(r,\cdot) \in L(L(V,D(A^2)),L(V,W)).
\end{align*}
Furthermore, by Corollary \ref{cor-orbit-map-1} for each $y \in L(V,D(A^2))$ we have
\begin{align*}
|\varphi(r,y)|_{L(V,W)} &= |(S_{t-r} - S_{s-r}) y|_{L(V,W)}
\\ &\leq |S_{t-r} - S_{s-r}|_{L(D(A^2),W)} |y|_{L(V,D(A^2))}
\\ &\leq |S_{s-r,t-r}|_{L(D(A),W)} |y|_{L(V,D(A^2))}
\\ &\leq M e^{\omega T} |t-s| \, |y|_{L(V,D(A^2))},
\end{align*}
showing that
\begin{align*}
| \varphi(r,\cdot) | \leq M e^{\omega T} |t-s|.
\end{align*}
Furthermore, we obtain
\begin{align*}
| Z_r |_{L(V,W)} = |\varphi(r,Y_r)|_{L(V,W)} \leq M e^{\omega T} |t-s| \, |Y_r|_{L(V,D(A^2))},
\end{align*}
showing (\ref{conv-2c}). Now, let $y \in L(V,D(A^2))$ and $q,r \in [0,s]$ be arbitrary. Then by Corollary \ref{cor-orbit-map-quad-2} we have
\begin{align*}
| \varphi(r,y) - \varphi(q,y) |_{L(V,W)} &= |S_{s-r,t-r} y - S_{s-q,t-q} y|_{L(V,W)}
\\ &\leq M e^{2 \omega T} |y|_{L(V,D(A^2))} |t-s| \, |r-q|
\\ &\leq M e^{2 \omega T} |y|_{L(V,D(A^2))} T^{1-2\alpha} |t-s| \, |r-q|^{2\alpha}.
\end{align*}
Hence, from Proposition \ref{prop-comp-linear} we obtain (\ref{conv-2}) and (\ref{conv-2a}), (\ref{conv-2b}). Now, we define $\Phi : [0,s] \times L(V,L(V,D(A^2))) \to L(V,L(V,W))$ as
\begin{align*}
\Phi(r,y') := S_{s-r,t-r} y'.
\end{align*}
Let $r \in [0,s]$ be arbitrary. Then by Lemma \ref{lemma-embedding-operators} and Corollary \ref{cor-orbit-map-1} for each $y' \in L(V,L(V,D(A^2)))$ we have
\begin{align*}
|\Phi(r,y')|_{L(V,L(V,W))} &= |S_{s-r,t-r} y'|_{L(V,L(V,W))}
\\ &\leq |S_{s-r,t-r}|_{L(L(V,D(A^2)),L(V,W))} |y'|_{L(V,L(V,D(A^2)))}
\\ &\leq |S_{s-r,t-r}|_{L(D(A^2),W)} |y'|_{L(V,L(V,D(A^2)))}
\\ &\leq |S_{s-r,t-r}|_{L(D(A),W)} |y'|_{L(V,L(V,D(A^2)))}
\\ &\leq M e^{\omega T} |t-s| \, |y'|_{L(V,L(V,D(A^2)))}.
\end{align*}
Therefore, we obtain
\begin{align*}
| Z_r' |_{L(V,L(V,W))} = |\Phi(r,Y_r')|_{L(V,L(V,W))} \leq M e^{\omega T} |t-s| \, |Y_r'|_{L(V,L(V,D(A^2)))},
\end{align*}
showing (\ref{conv-2d}).
\end{proof}

\section{Regular convolution integrals}\label{sec-reg-conv}

In this section we consider regular convolution integrals. We fix $\alpha \in (0,\frac{1}{2}]$ and a time horizon $T \in \bbr_+$. Let $\mathbf{X} = (X,\bbx) \in \scrc^{\alpha}([0,T],V)$ be a rough path with values in a Banach space $V$. Furthermore, let $A$ be the generator of a $C_0$-semigroup $(S_t)_{t \geq 0}$ on a Banach space $W$. Then there are constants $M \geq 1$ and $\omega \in \bbr$ such that the estimate (\ref{est-semigroup}) is satisfied.

\begin{proposition}\label{prop-orbit-Hoelder}
Let $\xi \in D(A)$ be arbitrary. We define the path $\Xi : [0,T] \to W$ as $\Xi_t := S_t \xi$ for all $t \in [0,T]$. Then we have
\begin{align*}
(\Xi,0) \in \scrd_X^{2\alpha}([0,T],W).
\end{align*}
Furthermore, we have
\begin{align*}
\| \Xi,0 \|_{X,2\alpha} \leq M e^{\omega T} |\xi|_{D(A)} T^{1-2\alpha}.
\end{align*}
\end{proposition}

\begin{proof}
Let $s,t \in [0,T]$ be arbitrary. Then by Proposition \ref{prop-orbit-map} we have
\begin{align*}
| \Xi_{s,t} | = | S_t \xi - S_s \xi | \leq M e^{\omega T} |\xi|_{D(A)} |t-s| \leq M e^{\omega T} |\xi|_{D(A)} T^{1-2\alpha} |t-s|^{2\alpha},
\end{align*}
showing that $\Xi \in \calc^{2\alpha}([0,T],W)$. Together with Lemma \ref{lemma-reg-derivative-zero}, this completes the proof.
\end{proof}

\begin{lemma}\label{lemma-reg-conv-Hoelder}
Let $Y : [0,T] \to D(A)$ be measurable and bounded. We define the path $N : [0,T] \to W$ as
\begin{align*}
N_t := \int_0^t S_{t-s} Y_s \, ds, \quad t \in [0,T].
\end{align*}
Then we have
\begin{align*}
|N_{s,t}| \leq (1+T) M e^{\omega T} \| Y \|_{\infty} |t-s| \quad \text{for all $s,t \in [0,T]$.}
\end{align*}
\end{lemma}

\begin{proof}
Let $s,t \in [0,T]$ with $s \leq t$ be arbitrary. Then we have
\begin{align*}
N_{s,t} &= \int_0^t S_{t-r} Y_r \, dr - \int_0^s S_{s-r} Y_r \, dr
\\ &= \int_s^t S_{t-r} Y_r \, dr + \int_0^s (S_{t-r} - S_{s-r}) Y_r \, dr.
\end{align*}
Therefore, using the estimate (\ref{est-semigroup}) and Proposition \ref{prop-orbit-map} we obtain
\begin{align*}
|N_{s,t}| &\leq \bigg| \int_s^t S_{t-r} Y_r \, dr \bigg| + \bigg| \int_0^s (S_{t-r} - S_{s-r}) Y_r \, dr \bigg|
\\ &\leq \int_s^t |S_{t-s} Y_r| \, dr + \int_0^s |S_{t-r} Y_r - S_{s-r} Y_r| \, dr
\\ &\leq M e^{\omega T} \int_s^t |Y_r| \, dr + M e^{\omega T} \int_0^s |Y_r|_{D(A)} |t-s| \, dr.
\\ &\leq M e^{\omega T} \|Y\|_{\infty} |t-s| + M e^{\omega T} T \|Y\|_{\infty} |t-s|,
\end{align*}
completing the proof.
\end{proof}

\begin{proposition}\label{prop-reg-conv-Hoelder}
Let $Y : [0,T] \to D(A)$ be measurable and bounded. We define the path $N : [0,T] \to W$ as
\begin{align*}
N_t := \int_0^t S_{t-s} Y_s \, ds, \quad t \in [0,T].
\end{align*}
Then $N$ is $D(A)$-valued, we have $(N,0) \in \scrd_X^{2\alpha}([0,T],W)$ and
\begin{align*}
\| N,0 \|_{X,2\alpha} \leq (1+T) M e^{\omega T} \| Y \|_{\infty} T^{1-2 \alpha}.
\end{align*}
\end{proposition}

\begin{proof}
This is a consequence of Lemma \ref{lemma-reg-conv-Hoelder} and Lemma \ref{lemma-reg-derivative-zero}.
\end{proof}

\begin{corollary}\label{cor-reg-conv-Hoelder}
Let $Y,Z : [0,T] \to D(A)$ be measurable and bounded. We define the paths $N,P : [0,T] \to W$ as
\begin{align*}
N_t &:= \int_0^t S_{t-s} Y_s \, ds, \quad t \in [0,T],
\\ P_t &:= \int_0^t S_{t-s} Z_s \, ds, \quad t \in [0,T].
\end{align*}
Then we have $(N,0),(P,0) \in \scrd_X^{2\alpha}([0,T],W)$ and
\begin{align*}
\| (N,0) - (P,0) \|_{X,2\alpha} \leq (1+T) M e^{\omega T} \| Y-Z \|_{\infty} T^{1-2 \alpha}.
\end{align*}
\end{corollary}

\begin{proof}
Noting that
\begin{align*}
N_t - P_t = \int_0^t S_{t-s} (Y_s - Z_s) \, ds, \quad t \in [0,T],
\end{align*}
this is an immediate consequence of Proposition \ref{prop-reg-conv-Hoelder}.
\end{proof}

\section{Rough convolution integrals}\label{sec-rough-conv}

In this section we investigate rough convolution integrals. We fix $\alpha \in (\frac{1}{3},\frac{1}{2}]$ and a time horizon $T \in \bbr_+$. Let $\mathbf{X} = (X,\bbx) \in \scrc^{\alpha}([0,T],V)$ be a rough path with values in a Banach space $V$. Furthermore, let $A$ be the generator of a $C_0$-semigroup $(S_t)_{t \geq 0}$ on a Banach space $W$. Then there are constants $M \geq 1$ and $\omega \in \bbr$ such that the estimate (\ref{est-semigroup}) is satisfied. For what follows, let a controlled rough path $(Y,Y') \in \scrd_X^{2 \alpha}([0,T],L(V,D(A^2)))$ be given. According to Proposition \ref{prop-conv-well-defined}, we can define the rough convolution $M : [0,T] \to W$ as
\begin{align*}
N_t := \int_0^t S_{t-s} Y_s \, d \mathbf{X}_s, \quad t \in [0,T].
\end{align*}
Note that the path $N$ is actually $D(A)$-valued. We also define the Gubinelli integral $I : [0,T] \to D(A^2)$ as
\begin{align*}
I_t := \int_0^t Y_s \, d \mathbf{X}_s, \quad t \in [0,T].
\end{align*}
By Theorem \ref{thm-Gubinelli} we have $(I,I') \in \scrd_X^{2\alpha}([0,T],D(A^2))$ with $I' = Y$. Furthermore, for all $s,t \in [0,T]$ with $s \leq t$ we have
\begin{equation}\label{eqn-diff-Z-I}
\begin{aligned}
N_{s,t} - I_{s,t} &= \int_0^t S_{t-r} Y_r \, d \mathbf{X}_r - \int_0^s S_{s-r} Y_r \, d \mathbf{X}_r - \int_s^t Y_r \, d \mathbf{X}_r
\\ &= \int_s^t S_{t-r} Y_r \, d \mathbf{X}_r + \int_0^s (S_{t-r} - S_{s-r}) Y_r \, d \mathbf{X}_r - \int_s^t Y_r \, d \mathbf{X}_r
\\ &= \int_s^t ( S_{t-r} - \Id ) Y_r \, d \mathbf{X}_r + \int_0^s (S_{t-r} - S_{s-r}) Y_r \, d \mathbf{X}_r.
\end{aligned}
\end{equation}

\begin{lemma}\label{lemma-comp-semigroup-1}
For all $s,t \in [0,T]$ with $s \leq t$ we have
\begin{align*}
&\bigg| \int_s^t (S_{t-r} - \Id) Y_r \, d \mathbf{X}_r \bigg|
\\ &\leq C \Big( \| X \|_{\alpha} \big( \| Y \|_{\infty} + \| R^Y \|_{2\alpha} \big) + \| \bbx \|_{2 \alpha} \big( \| Y' \|_{\infty} + \| Y' \|_{\alpha} \big) \Big) |t-s|^{3\alpha},
\end{align*}
where the constant $C > 0$ depends on $\alpha$, $T$ and $M$, $\omega$. Moreover, for $T \leq 1$ the constant $C$ does not depend on $T$.
\end{lemma}

\begin{proof}
We define the paths $Z : [s,t] \to L(V,W)$ and $Z' : [s,t] \to L(V,L(V,W))$ as
\begin{align*}
Z_r &:= (S_{t-r} - \Id) Y_r, \quad r \in [s,t],
\\ Z_r' &:= (S_{t-r} - \Id) Y_r', \quad t \in [s,t].
\end{align*}
By Lemma \ref{lemma-conv-1-pre} we have
\begin{align*}
(Z,Z') \in \scrd_X^{2 \alpha}([s,t],L(V,W))
\end{align*}
as well as
\begin{align*}
\| Z' \|_{\alpha} &\leq M e^{\omega T} \big( T^{1-\alpha} \| Y' \|_{\infty} + \| Y' \|_{\alpha} |t-s| \big),
\\ \| R^Z \|_{2\alpha} &\leq M e^{\omega T} \big( T^{1-2\alpha} \| Y  \|_{\infty} + \| R^Y \|_{2 \alpha} |t-s| \big),
\\ \| Z \|_{\infty} &\leq M e^{\omega T} \| Y \|_{\infty} |t-s|,
\\ \| Z' \|_{\infty} &\leq M e^{\omega T} \| Y' \|_{\infty} |t-s|.
\end{align*}
Furthermore, by Proposition \ref{prop-Gubinelli-triangle} we have
\begin{align*}
\bigg| \int_s^t Z_r \, d \mathbf{X}_r \bigg| &\leq C \big( \| X \|_{\alpha} \| R^{Z} \|_{2 \alpha} + \| \bbx \|_{2 \alpha} \| Z' \|_{\alpha} \big) |t-s|^{3 \alpha}
\\ &\quad + \| Z' \|_{\infty} \| \bbx \|_{2 \alpha} |t-s|^{2 \alpha} + \| Z \|_{\infty} \| X \|_{\alpha} |t-s|^{\alpha},
\end{align*}
where the constant $C > 0$ depends on $\alpha$. Moreover, note that $\alpha \leq \frac{1}{2}$ implies $3 \alpha \leq 1 + \alpha$. This completes the proof.
\end{proof}

\begin{corollary}\label{cor-comp-semigroup-1}
For all $s,t \in [0,T]$ with $s \leq t$ we have
\begin{align*}
\bigg| \int_s^t (S_{t-r} - \Id) Y_r \, d \mathbf{X}_r \bigg| \leq C \interleave \mathbf{X} \interleave_{\alpha} \interleave Y,Y' \interleave_{X,2\alpha} |t-s|^{3\alpha},
\end{align*}
where the constant $C > 0$ depends on $\alpha$, $T$, $\mathbf{X}$ and $M$, $\omega$. Moreover, for $T \leq 1$ the constant $C$ does not depend on $T$.
\end{corollary}

\begin{proof}
This is a consequence of Lemma \ref{lemma-comp-semigroup-1} and Lemma \ref{lemma-norm-of-Y}
\end{proof}

\begin{lemma}\label{lemma-comp-semigroup-2}
For all $s,t \in [0,T]$ with $s \leq t$ we have
\begin{align*}
&\bigg| \int_0^s (S_{t-r} - S_{s-r}) Y_r \, d \mathbf{X}_r \bigg|
\\ &\leq C \Big( \| X \|_{\alpha} \big( \| Y \|_{\infty} + \| R^Y \|_{2\alpha} \big) + \| \bbx \|_{2 \alpha} \big( \| Y' \|_{\infty} + \| Y' \|_{\alpha} \big) \Big) |t-s|,
\end{align*}
where the constant $C > 0$ depends on $\alpha$, $T$ and $M$, $\omega$. Moreover, for $T \leq 1$ the constant $C$ does not depend on $T$.
\end{lemma}

\begin{proof}
We define the paths $Z : [0,s] \to L(V,W)$ and $Z' : [0,s] \to L(V,L(V,W))$ as
\begin{align*}
Z_r &:= (S_{t-r} - S_{s-r}) Y_r, \quad r \in [0,s],
\\ Z_r' &:= (S_{t-r} - S_{s-r}) Y_r', \quad r \in [0,s].
\end{align*}
By Lemma \ref{lemma-conv-2-pre} we have
\begin{align*}
(Z,Z') \in \scrd_X^{2 \alpha}([0,s],L(V,W))
\end{align*}
and
\begin{align*}
\| Z' \|_{\alpha} &\leq M e^{\omega T} \big( T^{1-\alpha} e^{\omega T} \| Y' \|_{\infty} + \| Y' \|_{\alpha} \big) |t-s|,
\\ \| R^Z \|_{2\alpha} &\leq M e^{\omega T} \big( T^{1-2\alpha} e^{\omega T} \| Y  \|_{\infty} + \| R^Y \|_{2 \alpha} \big) |t-s|,
\\ \| Z \|_{\infty} &\leq M e^{\omega T} \| Y \|_{\infty} |t-s|,
\\ \| Z' \|_{\infty} &\leq M e^{\omega T} \| Y' \|_{\infty} |t-s|.
\end{align*}
Furthermore, by Proposition \ref{prop-Gubinelli-triangle} we have
\begin{align*}
\bigg| \int_0^s Z_r \, d \mathbf{X}_r \bigg| &\leq C \big( \| X \|_{\alpha} \| R^{Z} \|_{2 \alpha} + \| \bbx \|_{2 \alpha} \| Z' \|_{\alpha} \big) |s|^{3 \alpha}
\\ &\quad + \| Z' \|_{\infty} \| \bbx \|_{2 \alpha} |s|^{2 \alpha} + \| Z \|_{\infty} \| X \|_{\alpha} |s|^{\alpha},
\end{align*}
where the constant $C > 0$ depends on $\alpha$. This completes the proof.
\end{proof}

\begin{corollary}\label{cor-comp-semigroup-2}
For all $s,t \in [0,T]$ with $s \leq t$ we have
\begin{align*}
\bigg| \int_0^s (S_{t-r} - S_{s-r}) Y_r \, d \mathbf{X}_r \bigg| \leq C \interleave \mathbf{X} \interleave_{\alpha} \interleave Y,Y' \interleave_{X,2\alpha} |t-s|,
\end{align*}
where the constant $C > 0$ depends on $\alpha$, $T$, $\mathbf{X}$ and $M$, $\omega$. Moreover, for $T \leq 1$ the constant $C$ does not depend on $T$.
\end{corollary}

\begin{proof}
This is a consequence of Lemma \ref{lemma-comp-semigroup-2} and Lemma \ref{lemma-norm-of-Y}
\end{proof}

\begin{corollary}\label{cor-Z-I}
We have
\begin{align*}
| N_{s,t} - I_{s,t} | \leq C \interleave \mathbf{X} \interleave_{\alpha} \interleave Y,Y' \interleave_{X,2\alpha} |t-s|^{3 \alpha}, \quad s,t \in [0,T],
\end{align*}
where the constant $C > 0$ depends on $\alpha$, $T$, $\mathbf{X}$ and $M$, $\omega$. Moreover, for $T \leq 1$ the constant $C$ does not depend on $T$.
\end{corollary}

\begin{proof}
Note that $\alpha > \frac{1}{3}$ implies $1 < 3 \alpha$. Therefore, taking into account equation (\ref{eqn-diff-Z-I}), the assertion an immediate consequence Corollaries \ref{cor-comp-semigroup-1}, \ref{cor-comp-semigroup-2}.
\end{proof}

\begin{proposition}\label{prop-conv-rough}
Let $(Y,Y') \in \scrd_X^{2 \alpha}([0,T],L(V,D(A^2)))$ be arbitrary. We define the paths $N : [0,T] \to W$ and $N' : [0,T] \to L(V,W)$ as
\begin{align*}
N_t &:= \int_0^t S_{t-s} Y_s \, d \mathbf{X}_s, \quad t \in [0,T],
\\ N_t' &:= Y_t, \quad t \in [0,T].
\end{align*}
Then $N$ is $D(A)$-valued and we have
\begin{align}\label{M-conv-in-space}
(N,N') \in \scrd_X^{2\alpha}([0,T],W).
\end{align}
Furthermore, we have
\begin{align}\label{M-conv-est}
\| N,N' \|_{X,2\alpha} \leq C \interleave Y,Y' \interleave_{X,2\alpha} (  \interleave \mathbf{X} \interleave_{\alpha} + T^{\alpha} ),
\end{align}
where the constant $C > 0$ depends on $\alpha$, $T$, $\mathbf{X}$ and $M$, $\omega$. Moreover, for $T \leq 1$ the constant $C$ does not depend on $T$.
\end{proposition}

\begin{proof}
By Proposition \ref{prop-Gubinelli-triangle} we have $I \in \calc^{\alpha}([0,T],W)$, and by Corollary \ref{cor-Z-I} we have $N-I \in \calc^{3\alpha}([0,T],W)$. Therefore, it follows that $N \in \calc^{\alpha}([0,T],W)$. Furthermore, we have $N' = Y \in \calc^{\alpha}([0,T],L(V,W))$. Now, let $s,t \in [0,T]$ be arbitrary. Then we have
\begin{align*}
R_{s,t}^N &= N_{s,t} - N_s' X_{s,t} = N_{s,t} - Y_s X_{s,t}
\\ &= Y_s' \bbx_{s,t} + (N_{s,t} - I_{s,t}) + (I_{s,t} - Y_s X_{s,t} - Y_s' \bbx_{s,t}).
\end{align*}
Furthermore, we have
\begin{align*}
|Y_s' \bbx_{s,t}| \leq \| Y' \|_{\infty} |\bbx_{s,t}| \leq \| Y' \|_{\infty} \| \bbx \|_{2\alpha} |t-s|^{2\alpha}.
\end{align*}
Moreover, by Theorem \ref{thm-Gubinelli} we have
\begin{align*}
| I_{s,t} - Y_s X_{s,t} - Y_s' \bbx_{s,t} | &\leq C \big( \| X \|_{\alpha} \| R^Y \|_{2 \alpha} + \| \bbx \|_{2 \alpha} \| Y' \|_{\alpha} \big) |t-s|^{3 \alpha},
\end{align*}
where the constant $C > 0$ depends on $\alpha$. Therefore, together with Corollary \ref{cor-Z-I} we obtain $\| R^N \|_{2 \alpha} < \infty$, proving (\ref{M-conv-in-space}). Moreover, by the previous estimates and Corollary \ref{cor-Z-I} we have
\begin{align*}
\| N,N' \|_{X,2\alpha} &= \| N' \|_{\alpha} + \| R^N \|_{2 \alpha} = \| Y \|_{\alpha} + \| R^N \|_{2 \alpha}
\\ &\leq \| Y \|_{\alpha} + \| Y' \|_{\infty} \| \bbx \|_{2\alpha} + C \interleave \mathbf{X} \interleave_{\alpha} \interleave Y,Y' \interleave_{X,2\alpha}
\\ &\quad + C \big( \| X \|_{\alpha} \| R^Y \|_{2 \alpha} + \| \bbx \|_{2 \alpha} \| Y' \|_{\alpha} \big),
\end{align*}
where the constant $C > 0$ depends on $\alpha$, $T$, $\mathbf{X}$ and $M$, $\omega$, and does not depend on $T \leq 1$. Hence, using Lemma \ref{lemma-norm-of-Y} provides (\ref{M-conv-est}).
\end{proof}

\begin{corollary}\label{cor-conv-rough}
Let $(Y,Y'),(Z,Z') \in \scrd_X^{2\alpha}([0,T],L(V,D(A^2)))$ be arbitrary. We define the paths $N,P : [0,T] \to W$ as
\begin{align*}
N_t &:= \int_0^t S_{t-s} Y_s \, d\mathbf{X}_s, \quad t \in [0,T],
\\ P_t &:= \int_0^t S_{t-s} Z_s \, d\mathbf{X}_s, \quad t \in [0,T].
\end{align*}
Then we have
\begin{align*}
(N,Y), (P,Z) \in \scrd_X^{2\alpha}([0,T],W).
\end{align*}
Furthermore, for $T \leq 1$ we have
\begin{align*}
\| (N,Y) - (P,Z) \|_{X,2\alpha} \leq C \interleave (Y,Y') - (Z,Z') \interleave_{X,2\alpha} ( \interleave \mathbf{X} \interleave_{\alpha} + T^{\alpha} ),
\end{align*}
where the constant $C > 0$ depends on $\alpha$, $\mathbf{X}$ and $M$, $\omega$.
\end{corollary}

\begin{proof}
Noting that
\begin{align*}
N_t - P_t = \int_0^t S_{t-s} (Y_s - Z_s) \, d\mathbf{X}_s, \quad t \in [0,T],
\end{align*}
the result follows from Proposition \ref{prop-conv-rough}.
\end{proof}

\section{Rough partial differential equations}\label{sec-RPDEs}

In this section we deal with rough partial differential equations. Let $V,W$ be Banach spaces, and let $A$ be the generator of a $C_0$-semigroup $(S_t)_{t \geq 0}$ on $W$. Then there are constants $M \geq 1$ and $\omega \in \bbr$ such that the estimate (\ref{est-semigroup}) is satisfied. We fix a time horizon $T \in \bbr_+$. Consider the rough partial differential equation (RPDE)
\begin{align}\label{RPDE}
\left\{
\begin{array}{rcl}
dY_t & = & (A Y_t + f_0(t,Y_t)) dt + f(t,Y_t) d \bfx_t
\\ Y_0 & = & \xi
\end{array}
\right.
\end{align}
with a rough path $\bfx = (X,\bbx) \in \scrc^{\alpha}([0,T],V)$ for some index $\alpha \in (\frac{1}{3},\frac{1}{2}]$, and appropriate mappings $f_0 : [0,T] \times W \to W$ and $f : [0,T] \times W \to L(V,W))$.

\subsection{Solution concepts}

In this subsection we introduce the required solution concepts. Let $\bfx = (X,\bbx) \in \scrc^{\alpha}([0,T],V)$ be a rough path for some index $\alpha \in (\frac{1}{3},\frac{1}{2}]$.

\begin{definition}
Suppose that $f_0 \in \Lip([0,T] \times W,W)$ and $f \in C_b^{2\alpha,2}([0,T] \times W, L(V,W))$. Furthermore, let $\xi \in D(A)$ be arbitrary. A path $(Y,Y') \in \scrd_X^{2 \alpha}([0,T_0],W)$ for some $T_0 \in (0,T]$ is called a \emph{local strong solution} to the RPDE (\ref{RPDE}) with $Y_0 = \xi$ if $Y \in D(A)$, the path $AY : [0,T] \to W$ is bounded, and we have $Y' = f(Y)$ as well as
\begin{align}\label{strong-solution}
Y_t = \xi + \int_0^t \big( A Y_s + f_0(s,Y_s) \big) \, ds + \int_0^t f(s,Y_s) \, d \bfx_s, \quad t \in [0,T_0].
\end{align}
If we can choose $T_0 = T$, then we also call $(Y,Y')$ a \emph{(global) strong solution} to the RPDE (\ref{RPDE}) with $Y_0 = \xi$.
\end{definition}

\begin{remark}
Note that the right-hand side of (\ref{strong-solution}) is well-defined. Indeed, let $(Y,Y') \in \scrd_X^{2 \alpha}([0,T_0],W)$ for some $T_0 \in (0,T]$ be arbitrary. We define the paths $B(Y), \Phi(Y) : [0,T_0] \to W$ as
\begin{align}
B(Y)_t &:= \int_0^t \big( A Y_s + f_0(s,Y_s) \big) \, ds,
\\ \Phi(Y)_t &:= \int_0^t f(s,Y_s) \, d \bfx_s.
\end{align}
By assumption, the path $A Y : [0,T_0] \to W$ is bounded. Furthermore, the path $f_0(Y) : [0,T_0] \to W$ is continuous, and hence bounded.
Therefore, by Proposition \ref{prop-reg-conv-Hoelder} (applied with $S_t = \Id$ for all $t \geq 0$) it follows that
\begin{align*}
(B(Y),0) \in \scrd_X^{2\alpha}([0,T_0],W).
\end{align*}
Moreover, by Proposition \ref{prop-comp-f} we have
\begin{align*}
(f(Y),f(Y)') \in \scrd_X^{2\alpha}([0,T_0],L(V,W)),
\end{align*}
and hence by Proposition \ref{prop-conv-rough} (applied with $S_t = \Id$ for all $t \geq 0$) it follows that
\begin{align*}
(\Phi(Y),f(Y)) \in \scrd_X^{2\alpha}([0,T_0],W).
\end{align*}
Consequently, the right-hand side of (\ref{strong-solution}) is an element of $\calc^{\alpha}([0,T_0],W)$.
\end{remark}

\begin{definition}
Suppose that $f_0 \in \Lip([0,T] \times W,D(A))$ and $f \in C_b^{2\alpha,2}([0,T] \times W, L(V,D(A^2)))$. Furthermore, let $\xi \in D(A)$ be arbitrary. A path $(Y,Y') \in \scrd_X^{2 \alpha}([0,T_0],W)$ for some $T_0 \in (0,T]$ is called a \emph{local mild solution} to the RPDE (\ref{RPDE}) with $Y_0 = \xi$ if $Y' = f(Y)$ and
\begin{align}\label{mild-solution}
Y_t = S_t \xi + \int_0^t S_{t-s} f_0(s,Y_s) \, ds + \int_0^t S_{t-s} f(s,Y_s) \, d \bfx_s, \quad t \in [0,T_0].
\end{align}
If we can choose $T_0 = T$, then we also call $(Y,Y')$ a \emph{(global) mild solution} to the RPDE (\ref{RPDE}) with $Y_0 = \xi$.
\end{definition}

\begin{remark}\label{rem-well-defined}
Note that the right-hand side of (\ref{mild-solution}) is well-defined. Indeed, let $(Y,Y') \in \scrd_X^{2 \alpha}([0,T_0],W)$ for some $T_0 \in (0,T]$ be arbitrary. We define the paths $\Xi, \Gamma(Y), \Psi(Y) : [0,T_0] \to W$ as
\begin{align}\label{map-part-1}
\Xi_t &:= S_t \xi,
\\ \label{map-part-2} \Gamma(Y)_t &:= \int_0^t S_{t-s} f_0(s,Y_s) \, ds,
\\ \label{map-part-3} \Psi(Y)_t &:= \int_0^t S_{t-s} f(s,Y_s) \, d \bfx_s.
\end{align}
The path $\Xi$ is $D(A)$-valued, and by Proposition \ref{prop-orbit-Hoelder} we have
\begin{align*}
(\Xi,0) \in \scrd_X^{2\alpha}([0,T_0],W).
\end{align*}
Furthermore, the path $f_0(Y) : [0,T_0] \to D(A)$ is continuous, and hence bounded. Therefore, by Proposition \ref{prop-reg-conv-Hoelder} it follows that $\Gamma(Y)$ is $D(A)$-valued and
\begin{align*}
(\Gamma(Y),0) \in \scrd_X^{2\alpha}([0,T_0],W).
\end{align*}
Moreover, by Proposition \ref{prop-comp-f} we have
\begin{align*}
(f(Y),f(Y)') \in \scrd_X^{2\alpha}([0,T_0],L(V,D(A^2))),
\end{align*}
and hence by Proposition \ref{prop-conv-rough} it follows that $\Psi(Y)$ is $D(A)$-valued and
\begin{align*}
(\Psi(Y),f(Y)) \in \scrd_X^{2\alpha}([0,T_0],W).
\end{align*}
Consequently, the right-hand side of (\ref{mild-solution}) is an element of $\calc^{\alpha}([0,T_0],W)$, and we obtain that the path $Y$ is $D(A)$-valued.
\end{remark}

\begin{remark}\label{rem-der-solutions}
Suppose that $f_0 \in \Lip([0,T] \times W,D(A))$ and $f \in C_b^{2\alpha,2}([0,T] \times W, L(V,D(A^2)))$. Furthermore, let $\xi \in D(A)$ be arbitrary, and let $(Y,Y') \in \scrd_X^{2 \alpha}([0,T_0],W)$ be a local mild solution to the RPDE (\ref{RPDE}) with $Y_0 = \xi$ for some $T_0 \in (0,T]$. Then we have $Y' = f(Y)$, and hence
\begin{align*}
f(Y)' = Df(Y)Y' = Df(Y)f(Y),
\end{align*}
where $Df(Y)f(Y) : [0,T_0] \to L(V,L(V,D(A^2)))$ denotes the path
\begin{align*}
Df(Y)f(Y)_t := D_y f(t,Y_t) f(t,Y_t), \quad t \in [0,T_0].
\end{align*}
\end{remark}

In view of the upcoming result, recall that according to Remark \ref{rem-well-defined} for every local mild solution $(Y,Y')$ to the RPDE (\ref{RPDE}) we have $Y \in D(A)$.

\begin{proposition}\label{prop-mild-strong}
Suppose that $f_0 \in \Lip([0,T] \times W,D(A))$ and $f \in C_b^{2\alpha,2}([0,T] \times W, L(V,D(A^2)))$. Let $\xi \in D(A)$ be arbitrary, and let $(Y,Y') \in \scrd_X^{2\alpha}([0,T_0],W)$ be a local mild solution to the RPDE (\ref{RPDE}) with $Y_0 = \xi$ for some $T_0 \in (0,T]$ such that the path $AY : [0,T_0] \to W$ is bounded. Then $(Y,Y')$ is also a local strong solution to the RPDE (\ref{RPDE}) with $Y_0 = \xi$.
\end{proposition}

\begin{proof}
Let $t \in [0,T_0]$ be arbitrary. By Lemma~\ref{lemma-hg-rules} we have
\begin{align*}
S_t \xi - \xi = \int_0^t A S_s \xi \, ds.
\end{align*}
Furthermore, by Lemma~\ref{lemma-hg-rules} and Fubini's theorem we have
\begin{align*}
&\int_0^t \big( S_{t-s} f_0(s, Y_s) - f_0(s, Y_s) \big) ds = \int_0^t \bigg( \int_0^{t-s} A S_u f_0(s, Y_s) du \bigg) ds
\\ &= \int_0^t \bigg( \int_u^{t} A S_{s-u} f_0(u, Y_u) ds \bigg) du = \int_0^t \bigg( \int_0^s A S_{s-u} f_0(u, Y_u) du \bigg) ds
\\ &= \int_0^t A \bigg( \int_0^s S_{s-u} f_0(u, Y_u) du \bigg) ds.
\end{align*}
For the upcoming calculation we would like to use the Rough Fubini (Proposition \ref{prop-rough-Fubini}). In order to verify that the required conditions are fulfilled, let us define the mapping $Z : [0,t] \to \scrd_X^{2\alpha}([0,t],W)$ as follows. We fix $s \in [0,t]$ and define $Z(s) : [0,t] \to W$ as
\begin{align*}
Z(s)_u := A S_{s-u} f(u,Y_u) \mathbf{1}_{[0,s]}(u), \quad u \in [0,t].
\end{align*}
Restricting the controlled rough path to the interval $[0,s]$ we have
\begin{align*}
(Y,Y') \in \scrd_X^{2\alpha}([0,s],W).
\end{align*}
By Proposition \ref{prop-comp-f} we have
\begin{align*}
(f(Y),f(Y)') \in \scrd_X^{2\alpha}([0,s],L(V,D(A^2)))
\end{align*}
and the estimate
\begin{align*}
\interleave f(Y),f(Y)' \interleave_{X,2\alpha} \leq C \big( 1 + \interleave Y,Y' \interleave_{X,2\alpha}^2 \big),
\end{align*}
where the constant $C > 0$ depends on $\alpha$, $T$, $\mathbf{X}$ and $\| f \|_{C_b^{2\alpha,2}}$. Consider the mapping $\varphi^s : [0,s] \times L(V,D(A^2)) \to L(V,D(A))$ given by
\begin{align*}
\varphi^s(u,y) = S_{s-u} y.
\end{align*}
Taking into account Lemmas \ref{lemma-restricted-semigroup} and \ref{lemma-graph-norms-eq}, by Proposition \ref{prop-conv-well-defined} we have
\begin{align*}
(\varphi^s(f(Y)),\varphi^s(f(Y))') \in \scrd_X^{2\alpha}([0,s],L(V,D(A)))
\end{align*}
and the estimate
\begin{align*}
\interleave \varphi^s(f(Y)),\varphi^s(f(Y))' \interleave_{X,2\alpha} \leq C \interleave f(Y),f(Y)' \interleave_{X,2\alpha},
\end{align*}
where the constant $C > 0$ depends on $\alpha$, $T$, $\mathbf{X}$ and $M$, $\omega$. Moreover, taking into account Lemma \ref{lemma-embedding-operators}, by Proposition \ref{prop-comp-linear-0} we have
\begin{align*}
(Z,Z') \in \scrd_X^{2\alpha}([0,s],L(V,W))
\end{align*}
and the estimate
\begin{align*}
\interleave Z,Z' \interleave_{X,2\alpha} \leq | A | \, \interleave \varphi^s(f(Y)),\varphi^s(f(Y))' \interleave_{X,2\alpha},
\end{align*}
where $A$ is regarded as a continuous linear operator $A \in L(D(A),W)$. Consequently, the mapping
\begin{align*}
[0,t] \to \bbr_+, \quad u \mapsto |Z(s)_0| + \| Z(s),Z(s)' \|_{X,2\alpha}
\end{align*}
is bounded. Therefore, we may apply the Rough Fubini (Proposition \ref{prop-rough-Fubini}), which, together with Lemma \ref{lemma-hg-rules} gives us
\begin{align*}
&\int_0^t \big( S_{t-s} f(s, Y_s) - f(s, Y_s) \big) d \mathbf{X}_s = \int_0^t \bigg( \int_0^{t-s} A S_u f(s, Y_s) du \bigg) d \mathbf{X}_s
\\ &= \int_0^t \bigg( \int_u^{t} A S_{s-u} f(u, Y_u) ds \bigg) d \mathbf{X}_u = \int_0^t \bigg( \int_0^s A S_{s-u} f(u, X_u) d \mathbf{X}_u \bigg) ds
\\ &= \int_0^t A \bigg( \int_0^s S_{s-u} f(u, X_u) d \mathbf{X}_u \bigg) ds.
\end{align*}
Since $Y$ is a local mild solution to the RPDE (\ref{RPDE}) with $Y_0 = \xi$, we have
\begin{align*}
Y_t &= S_t \xi + \int_0^{t} S_{t-s} f_0(s, Y_s) ds
+ \int_0^{t} S_{t-s} f(s, Y_s) d \mathbf{X}_s
\\ &= \xi + \int_0^t f_0(s, Y_s) ds + \int_0^t f(s, Y_s) d \mathbf{X}_s
\\ &\quad + (S_t \xi - \xi) + \int_0^t (S_{t-s} f_0(s, Y_s) - f_0(s, Y_s)) ds
\\ &\quad + \int_0^t (S_{t-s} f(s, Y_s) - f(s, Y_s)) d \mathbf{X}_s,
\end{align*}
and hence, combining the latter identities, we obtain
\begin{align*}
Y_t &= \xi + \int_0^t f_0(s, Y_s) ds + \int_0^t f(s, Y_s) d \mathbf{X}_s
\\ &\quad + \int_0^t A S_s \xi \, ds + \int_0^t A \bigg( \int_0^s S_{s-u} f_0(u, Y_u) du \bigg) ds
\\ &\quad + \int_0^t A \bigg( \int_0^s S_{s-u} f(u, Y_u) d \mathbf{X}_u \bigg) ds,
\end{align*}
which implies
\begin{align*}
Y_t &= \xi + \int_0^t f_0(s, Y_s) ds + \int_0^t f(s, Y_s) d \mathbf{X}_s
\\ &\quad + \int_0^t A \underbrace{\bigg( S_s \xi + \int_0^s S_{s-u} f_0(u, Y_u) du + \int_0^s S_{s-u} f(u, Y_u) d \mathbf{X}_u \bigg)}_{= Y_s} ds
\\ &= \xi + \int_{0}^{t} \big( A Y_s + f_0(s, Y_s) \big) ds + \int_{0}^{t} f(s, Y_s) d \mathbf{X}_s.
\end{align*}
This proves that $Y$ is also a local strong solution to the RPDE (\ref{RPDE}) with $Y_0 = \xi$.
\end{proof}

\subsection{The space for the fixed point problem}

Note that equation (\ref{mild-solution}) may be regarded as a fixed point problem. In this subsection we will analyze the space for this fixed point problem. Let $\bfx = (X,\bbx) \in \scrc^{\alpha}([0,T],V)$ be a rough path for some index $\alpha \in (\frac{1}{3},\frac{1}{2}]$. Furthermore, let $f_0 \in \Lip([0,T] \times W,D(A))$ and $f \in C_b^{2\alpha,3}([0,T] \times W, L(V,D(A^2)))$ be arbitrary. We also fix an initial condition $\xi \in D(A)$. For every $t \in [0,T]$ we define the subset $\bbb_t \subset \scrd_X^{2\alpha}([0,t],W)$ as
\begin{align*}
\bbb_t := \{ (Y,Y') \in \scrd_X^{2\alpha}([0,t],W) : Y_0 = \xi, Y_0' = f(0,\xi), \| Y,Y' \|_{X,2\alpha;[0,t]} \leq 1 \}.
\end{align*}

\begin{lemma}
For each $t \in [0,T]$ the set $\bbb_t$ equipped with the metric
\begin{align*}
d \big( (Y,Y'), (Z,Z') \big) := \| (Y,Y') - (Z,Z') \|_{X,2\alpha;[0,t]}, \quad (Y,Y'), (Z,Z') \in \bbb_t
\end{align*}
is a complete metric space.
\end{lemma}

\begin{proof}
The space $\scrd_X^{2\alpha}([0,t],W)$ equipped with the norm $\interleave \cdot \interleave_{X,2\alpha;[0,t]}$ is a Banach space, and hence a complete metric space. Furthermore, for all $(Y,Y'), (Z,Z') \in \bbb_t$ we have
\begin{align*}
d \big( (Y,Y'), (Z,Z') \big) &= \| Y-Z,Y'-Z' \|_{X,2\alpha;[0,t]}
\\ &= |Y_0-Z_0| + |Y_0' - Z_0'| +  \| Y-Z,Y'-Z' \|_{X,2\alpha;[0,t]}
\\ &= \interleave Y-Z,Y'-Z' \interleave_{X,2\alpha;[0,t]}.
\end{align*}
Moreover, the set $\bbb_t$ is a closed subset of $\scrd_X^{2\alpha}([0,t],W)$, completing the proof.
\end{proof}

For $t \in [0,T]$ we define the mapping $\Phi_t : \scrd_X^{2\alpha}([0,t],W) \to \scrd_X^{2\alpha}([0,t],W)$ as
\begin{align}\label{def-Phi}
\Phi_t(Y,Y') := \Xi + \Gamma(Y) + \Psi(Y),
\end{align}
where the paths $\Xi, \Gamma(Y), \Psi(Y) : [0,t] \to W$ are defined according to (\ref{map-part-1})--(\ref{map-part-3}). Note that the mapping $\Phi_t$ is well-defined due to Remark \ref{rem-well-defined}.

\begin{proposition}\label{prop-fixed-point-1}
For all $t \in [0,T]$ with $t \leq 1$ and all $(Y,Y') \in \bbb_t$ we have
\begin{align*}
\| \Phi_t(Y,Y') \|_{X,2\alpha;[0,t]} \leq C ( 1 + |\xi|_{D(A)} ) t^{1-2\alpha} + C ( \interleave \mathbf{X} \interleave_{\alpha;[0,t]} + t^{\alpha} ),
\end{align*}
where the constant $C > 0$ depends on $\alpha$, $\mathbf{X}$, $\| f_0 \|_{\Lip}$, $\| f \|_{C_b^{2\alpha,2}}$ and $M$, $\omega$.
\end{proposition}

\begin{proof}
For convenience of notation, we will skip the subscript $[0,t]$ in the following calculations. Note that
\begin{align*}
\| \Phi_t(Y,Y') \|_{X,2\alpha} \leq \| \Xi,0 \|_{X,2\alpha} + \| \Gamma(Y),0 \|_{X,2\alpha} + \| \Psi(Y),f(Y) \|_{X,2\alpha}.
\end{align*}
By Proposition \ref{prop-orbit-Hoelder} we have
\begin{align*}
\| \Xi,0 \|_{X,2\alpha} \leq M e^{\omega} |\xi|_{D(A)} t^{1-2\alpha}.
\end{align*}
Moreover, by Proposition \ref{prop-reg-conv-Hoelder} we have
\begin{align*}
\| \Gamma(Y),0 \|_{X,2\alpha} &\leq 2 M e^{\omega} \| f_0(Y) \|_{\infty} t^{1-2\alpha}.
\end{align*}
Noting that $Y_0 = \xi$, $Y_0' = f(0,\xi)$ and $\| Y,Y' \|_{X,2\alpha} \leq 1$, by Lemma \ref{lemma-lin-growth} and Lemma \ref{lemma-norm-of-Y} we obtain
\begin{align*}
\| f_0(Y) \|_{\infty} &\leq \| f_0 \|_{\Lip} ( 1 + \| Y \|_{\infty} )
\\ &\leq \| f_0 \|_{\Lip} \big( 1 + \interleave Y,Y' \interleave_{X,2\alpha} ( \interleave \mathbf{X} \interleave_{\alpha} + 2 ) \big)
\\ &\leq \| f_0 \|_{\Lip} \big( 1 + ( |\xi| + |f(0,\xi)| + 1) ( \interleave \mathbf{X} \interleave_{\alpha} + 2 ) \big)
\\ &\leq \| f_0 \|_{\Lip} \big( 1 + ( |\xi| + \| f \|_{C_b^{2\alpha,2}} + 1) ( \interleave \mathbf{X} \interleave_{\alpha} + 2 ) \big).
\end{align*}
Moreover, noting that $Y_0 = \xi$, $Y_0' = f(0,\xi)$ and $\| Y,Y' \|_{X,2\alpha} \leq 1$, by Proposition \ref{prop-conv-rough} and Proposition \ref{prop-comp-f} we have
\begin{align*}
\| \Psi(Y),f(Y) \|_{X,2\alpha} &\leq C \interleave f(Y), f(Y)' \interleave_{X,2\alpha} ( \interleave \mathbf{X} \interleave_{\alpha} + t^{\alpha} )
\\ &= C ( |f(0,\xi)| + | f(Y), f(Y)' |_{X,2\alpha} ) ( \interleave \mathbf{X} \interleave_{\alpha} + t^{\alpha} )
\\ &\lesssim C \big( 1 + | Y,Y' |_{X,2\alpha}^2 \big) ( \interleave \mathbf{X} \interleave_{\alpha} + t^{\alpha} )
\\ &\lesssim C \big( 1 + (|f(0,\xi)| + 1)^2 \big) ( \interleave \mathbf{X} \interleave_{\alpha;[0,t]} + t^{\alpha} )
\\ &\lesssim C ( \interleave \mathbf{X} \interleave_{\alpha} + t^{\alpha} ),
\end{align*}
where the constant $C > 0$, which changes from line to line, depends on $\alpha$, $\mathbf{X}$, $\| f \|_{C_b^{2\alpha,2}}$ and $M$, $\omega$. This completes the proof.
\end{proof}

\begin{proposition}\label{prop-fixed-point-2}
For all $t \in [0,T]$ with $t \leq 1$ and all $(Y,Y'), (Z,Z') \in \bbb_t$ we have
\begin{align*}
\| \Phi_t(Y,Y') - \Phi_t(Z,Z') \|_{X,2\alpha;[0,t]} &\leq C \| (Y,Y') - (Z,Z') \|_{X,2\alpha;[0,t]}
\\ &\qquad \times \big( \interleave \mathbf{X} \interleave_{\alpha;[0,t]} + t^{\alpha} + t^{1-2\alpha} \big),
\end{align*}
where the constant $C > 0$ depends on $\alpha$, $\mathbf{X}$, $\| f_0 \|_{\Lip}$, $\| f \|_{C_b^{2\alpha,3}}$ and $M$, $\omega$.
\end{proposition}

\begin{proof}
For convenience of notation, we will skip the subscript $[0,t]$ in the following calculations. Note that
\begin{align*}
\| \Phi(Y,Y') - \Phi(Z,Z') \|_{X,2\alpha} &\leq \| (\Gamma(Y),0) - (\Gamma(Z),0) \|_{X,2\alpha}
\\ &\quad + \| (\Psi(Y),f(Y)) - (\Psi(Z),f(Z)) \|_{X,2\alpha}.
\end{align*}
By Corollary \ref{cor-reg-conv-Hoelder} we have
\begin{align*}
\| (\Gamma(Y),0) - (\Gamma(Z),0) \|_{X,2\alpha} \leq 2 M e^{\omega} \| f_0(Y) - f_0(Z) \|_{\infty} t^{1-2\alpha}.
\end{align*}
Furthermore, noting that $Y_0 = Z_0$ and $Y_0' = Z_0'$, by Lemma \ref{lemma-lin-growth} and Lemma \ref{lemma-norm-of-Y} we obtain
\begin{align*}
\| f_0(Y) - f_0(Z) \|_{\infty} &\leq \| f_0 \|_{\Lip} \| Y-Z \|_{\infty}
\\ &\leq \| f_0 \|_{\Lip} \| (Y,Y') - (Z,Z') \|_{X,2\alpha} ( \interleave \mathbf{X} \interleave_{\alpha} + 2).
\end{align*}
Moreover, by Corollary \ref{cor-conv-rough} and Proposition \ref{prop-diff-f} we have
\begin{align*}
&\| (\Psi(Y),f(Y)) - (\Psi(Z),f(Z)) \|_{X,2\alpha}
\\ &\leq C \interleave (f(Y),f(Y)') - (f(Z),f(Z)') \interleave_{X,2\alpha} ( \interleave \mathbf{X} \interleave_{\alpha} + t^{\alpha} )
\\ &\lesssim C ( 1 + | Y,Y' |_{X,2\alpha}^2 + | Z,Z' |_{X,2\alpha}^2 )
\\ &\qquad \times \interleave (Y,Y') - (Z,Z') \interleave_{X,2\alpha} ( \interleave \mathbf{X} \interleave_{\alpha} + t^{\alpha} ),
\end{align*}
where the constant $C>0$, which changes from line to line, depends on $\alpha$, $\mathbf{X}$, $\| f \|_{C_b^{2\alpha,3}}$ and $M$, $\omega$. Now, noting that $Y_0 = Z_0 = \xi$, $Y_0' = Z_0' = f(0,\xi)$ and $\| Y,Y' \|_{X,2\alpha}, \| Z,Z' \|_{X,2\alpha} \leq 1$, we obtain
\begin{align*}
| Y,Y' |_{X,2\alpha} &= |f(0,\xi)| + \| Y,Y' \|_{X,2\alpha} \leq \| f \|_{C_b^{2\alpha,3}} + 1,
\\ | Z,Z' |_{X,2\alpha} &= |f(0,\xi)| + \| Z,Z' \|_{X,2\alpha} \leq \| f \|_{C_b^{2\alpha,3}} + 1.
\end{align*}
Consequently, noting again that $Y_0 = Z_0$ and $Y_0' = Z_0'$ we have
\begin{align*}
\interleave (Y,Y') - (Z,Z') \interleave_{X,2\alpha} = \| (Y,Y') - (Z,Z') \|_{X,2\alpha},
\end{align*}
concluding the proof.
\end{proof}

\subsection{Auxiliary results}

In this subsection we derive auxiliary results which are required in order to establish existence and uniqueness of local mild solutions for the RPDE (\ref{RPDE}). Let $f_0 \in \Lip([0,T] \times W,D(A))$ and $f \in C_b^{2\alpha,2}([0,T] \times W, L(V,D(A^2)))$ be arbitrary.

\begin{lemma}\label{lemma-Hoelder-order-2}
Let $\bfx = (X,\bbx) \in \scrc^{\alpha}([0,T],V)$ be a rough path for some index $\alpha \in (\frac{1}{3},\frac{1}{2}]$. Furthermore, let $\xi \in D(A)$ be arbitrary, and let $(Y,Y') \in \scrd_X^{2 \alpha}([0,T_0],W)$ be a local mild solution to the RPDE (\ref{RPDE}) with $Y_0 = \xi$ for some $T_0 \in (0,T]$. Then we have
\begin{align*}
\| f(Y)' \|_{\alpha} \leq C ( 1 + \| Y \|_{\alpha} ),
\end{align*}
where the constant $C > 0$ depends on $\alpha$, $T$ and $\| f \|_{C_b^{2\alpha,2}}$.
\end{lemma}

\begin{proof}
By Proposition \ref{prop-comp-f} we have
\begin{align*}
\| f(Y) \|_{\alpha} \leq \| f \|_{C_b^{2\alpha,2}} ( \| Y \|_{\alpha} + T^{\alpha} ).
\end{align*}
Therefore, by Proposition \ref{prop-comp-f} and Remark \ref{rem-der-solutions} we obtain
\begin{align*}
\| f(Y)' \|_{\alpha} &\leq \| f \|_{C_b^{2\alpha,2}} \big( \| f(Y) \|_{\alpha} + \| Y \|_{\alpha} \| f(Y) \|_{\infty} + T^{\alpha} \| f(Y) \|_{\infty} \big)
\\ &\leq \| f \|_{C_b^{2\alpha,2}} \big( \| f \|_{C_b^{2\alpha,2}} ( \| Y \|_{\alpha} + T^{\alpha} ) + \| Y \|_{\alpha} \| f \|_{C_b^{2\alpha,2}} + T^{\alpha} \| f \|_{C_b^{2\alpha,2}} \big),
\end{align*}
completing the proof.
\end{proof}

\begin{lemma}\label{lemma-remainder-mild-sol}
Let $\bfx = (X,\bbx) \in \scrc^{\alpha}([0,T],V)$ be a rough path for some index $\alpha \in (\frac{1}{3},\frac{1}{2}]$. Furthermore, let $\xi \in D(A)$ be arbitrary, and let $(Y,Y') \in \scrd_X^{2 \alpha}([0,T_0],W)$ be a local mild solution to the RPDE (\ref{RPDE}) with $Y_0 = \xi$ for some $T_0 \in (0,T]$. Let $s,t \in [0,T_0]$ with $s \leq t$ be arbitrary, and define the interval $I := [s,t]$. Then we have
\begin{align*}
| R_{s,t}^Y | &\leq C \Big( 1 + \| Y \|_{\alpha;I} + \| X \|_{\alpha;I} ( 1 + \| R^{f(Y)} \|_{2 \alpha;I} ) + \| \bbx \|_{2\alpha;I} ( 1 + \| Y \|_{\alpha;I} ) \Big) |t-s|
\\ &\quad + C | \bbx_{s,t} |,
\end{align*}
where the constant $C > 0$ depends on $\xi$, $\alpha$, $T$, $\| f_0 \|_{\Lip}$, $\| f \|_{C_b^{2\alpha,2}}$ and $M$, $\omega$.
\end{lemma}

\begin{proof}
For convenience of notation, we will skip the subscript $I$ in the following calculations. By equation (\ref{eqn-diff-Z-I}) we have
\begin{align*}
Y_{s,t} &= S_{s,t} \xi + \Gamma(Y)_{s,t} + \int_s^t (S_{t-r} - \Id)f(Y_r) \, d\mathbf{X}_r
\\ &\quad + \int_0^s (S_{t-r} - S_{s-r})f(Y_r) \, d\mathbf{X}_r + \int_s^t f(Y_r) \, d\mathbf{X}_r,
\end{align*}
where the path $\Gamma(Y) : [0,T_0] \to W$ is defined according to (\ref{map-part-2}). Therefore, by Remark \ref{rem-der-solutions} we have
\begin{align*}
|R_{s,t}^Y| &= | Y_{s,t} - Y_s' X_{s,t} | = | Y_{s,t} - f(Y_s) X_{s,t} |
\\ &\leq |S_{s,t} \xi| + | \Gamma(Y)_{s,t} | + \bigg| \int_s^t (S_{t-r} - \Id)f(Y_r) \, d\mathbf{X}_r \bigg|
\\ &\quad + \bigg| \int_0^s (S_{t-r} - S_{s-r})f(Y_r) \, d\mathbf{X}_r \bigg|
\\ &\quad + \bigg| \int_s^t f(Y_r) \, d \mathbf{X}_r - f(Y_s) X_{s,t} - Df(Y_s)f(Y_s) \bbx_{s,t} \bigg|
\\ &\quad + | Df(Y_s)f(Y_s) \bbx_{s,t} |.
\end{align*}
By Proposition \ref{prop-orbit-map} we have
\begin{align*}
|S_{s,t} \xi| \leq M e^{\omega T} |\xi|_{D(A)} |t-s|,
\end{align*}
and by Lemma \ref{lemma-reg-conv-Hoelder} and Lemma \ref{lemma-lin-growth} we have
\begin{align*}
| \Gamma(Y)_{s,t} | &\leq (1+T) M e^{\omega T} \| f_0(Y) \|_{\infty} |t-s|
\\ &\leq (1+T) M e^{\omega T} \| f_0 \|_{\Lip} (1 + \| Y \|_{\infty}) |t-s|
\\ &\leq (1+T) M e^{\omega T} \| f_0 \|_{\Lip} (1 + |Y_0| + \| Y \|_{\alpha} T^{\alpha}) |t-s|
\\ &\leq (1+T) M e^{\omega T} \| f_0 \|_{\Lip} (1 + |\xi| + \| Y \|_{\alpha} T^{\alpha}) |t-s|.
\end{align*}
Furthermore, by Lemma \ref{lemma-comp-semigroup-1} we have
\begin{align*}
&\bigg| \int_s^t (S_{t-r} - \Id)f(Y_r) \, d\mathbf{X}_r \bigg|
\\ &\leq C \Big( \| X \|_{\alpha} \big( \| f(Y) \|_{\infty} + \| R^{f(Y)} \|_{2 \alpha} \big) + \| \bbx \|_{2\alpha} \big( \| f(Y)' \|_{\infty} + \| f(Y)' \|_{\alpha} \big) \Big) |t-s|^{3\alpha},
\end{align*}
where the constant $C > 0$ depends on $\alpha$, $T$ and $M$, $\omega$. Similarly, by Lemma \ref{lemma-comp-semigroup-2} we have
\begin{align*}
&\bigg| \int_0^s (S_{t-r} - S_{s-r})f(Y_r) \, d\mathbf{X}_r \bigg|
\\ &\leq C \Big( \| X \|_{\alpha} \big( \| f(Y) \|_{\infty} + \| R^{f(Y)} \|_{2 \alpha} \big) + \| \bbx \|_{2\alpha} \big( \| f(Y)' \|_{\infty} + \| f(Y)' \|_{\alpha} \big) \Big) |t-s|,
\end{align*}
where the constant $C > 0$ depends on $\alpha$, $T$ and $M$, $\omega$. Furthermore, by Theorem \ref{thm-Gubinelli} and Remark \ref{rem-der-solutions} we have
\begin{align*}
&\bigg| \int_s^t f(Y_r) \, d \mathbf{X}_r - f(Y_s) X_{s,t} - Df(Y_s)f(Y_s) \bbx_{s,t} \bigg|
\\ &\leq C \big( \| X \|_{\alpha} \| R^{f(Y)} \|_{2\alpha} + \| \bbx \|_{2\alpha} \| f(Y)' \|_{\alpha} \big) |t-s|^{3\alpha},
\end{align*}
where the constant $C > 0$ depends on $\alpha$. Moreover, we have
\begin{align*}
| Df(Y_s)f(Y_s) \bbx_{s,t} | \leq |Df(Y_s)| \, |f(Y_s)| \, | \bbx_{s,t} | \leq \| f \|_{C_b^{2\alpha,2}}^2 | \bbx_{s,t} |.
\end{align*}
Noting that $\alpha > \frac{1}{3}$, which implies $3 \alpha > 1$, using Lemma \ref{lemma-Hoelder-order-2} the proof is complete.
\end{proof}

\begin{lemma}\label{lemma-spaces-alpha-beta}
Let $\beta \in (\frac{1}{3},\frac{1}{2}]$ be arbitrary. We choose $\alpha \in (\frac{1}{3}, \beta)$ such that $\beta \leq \frac{3}{2} \alpha$, and let
\begin{align*}
\mathbf{X} = (X,\bbx) \in \scrc^{\beta}([0,T],V) \subset \scrc^{\alpha}([0,T],V)
\end{align*}
be a rough path. Furthermore, let $\xi \in D(A)$ be arbitrary, and let $(Y,Y') \in \scrd_X^{2 \alpha}([0,T_0],W)$ be a local mild solution to the RPDE (\ref{RPDE}) with $Y_0 = \xi$ for some $T_0 \in (0,T]$. Then we even have $(Y,Y') \in \scrd_X^{2 \beta}([0,T_0],W)$.
\end{lemma}

\begin{proof}
Let $s,t \in [0,T_0]$ be arbitrary. Since $\beta \leq \frac{3}{2} \alpha < 2\alpha$, we have
\begin{align*}
|Y_{s,t}| &= |Y_s' X_{s,t} + R_{s,t}^Y| \leq \| Y' \|_{\infty} |X_{s,t}| + |R_{s,t}^Y|
\\ &\leq \| Y' \|_{\infty} \| X \|_{\beta} |t-s|^{\beta} + \| R^Y \|_{2\alpha} |t-s|^{2\alpha}
\\ &\leq \| Y' \|_{\infty} \| X \|_{\beta} |t-s|^{\beta} + \| R^Y \|_{2\alpha} |t-s|^{\beta},
\end{align*}
showing that $Y \in \calc^{\beta}([0,T_0],W)$. Furthermore, from Remark \ref{rem-der-solutions} and Proposition \ref{prop-comp-Hoelder-beta} we obtain
\begin{align*}
Y' = f(Y) \in \calc^{\beta}([0,T_0],L(V,D(A^2))) \subset \calc^{\beta}([0,T_0],L(V,W)).
\end{align*}
Moreover, noting that $2 \beta \leq 1$, by Lemma \ref{lemma-remainder-mild-sol} we have
\begin{align*}
| R_{s,t}^Y | &\leq C \Big( 1 + \| Y \|_{\alpha;I} + \| X \|_{\alpha;I} ( 1 + \| R^{f(Y)} \|_{2 \alpha;I} ) + \| \bbx \|_{2\alpha;I} ( 1 + \| Y \|_{\alpha;I} ) \Big) |t-s|
\\ &\quad + C | \bbx_{s,t} |
\\ &\leq C \Big( 1 + \| Y \|_{\alpha;I} + \| X \|_{\alpha;I} ( 1 + \| R^{f(Y)} \|_{2 \alpha;I} ) + \| \bbx \|_{2\alpha;I} ( 1 + \| Y \|_{\alpha;I} ) \Big) T^{1-2\beta} |t-s|^{2\beta}
\\ &\quad + C \| \bbx \|_{2 \beta} |t-s|^{2 \beta},
\end{align*}
where the constant $C > 0$ depends on $\xi$, $\alpha$, $T$, $\| f_0 \|$, $\| f \|_{C_b^{2\alpha,2}}$ and $M$, $\omega$. This shows $\| R^Y \|_{2\beta} < \infty$, completing the proof.
\end{proof}

\subsection{Existence and uniqueness of local mild solutions}

In this subsection we present a result concerning existence and uniqueness of local mild solutions to the RPDE (\ref{RPDE}).

\begin{theorem}\label{thm-RDGL-main-local}
Let $\bfx = (X,\bbx) \in \scrc^{\beta}([0,T],V)$ be a rough path for some index $\beta \in (\frac{1}{3},\frac{1}{2}]$, and let $f_0 \in \Lip([0,T] \times W,D(A))$ and $f \in C_b^{2\beta,3}([0,T] \times W, L(V,D(A^2)))$ be mappings. Then for every $\xi \in D(A)$ there exist $T_0 \in (0,T]$ and a unique local mild solution $(Y,Y') \in \scrd_X^{2 \beta}([0,T_0],W)$ to the RPDE (\ref{RPDE}) with $Y_0 = \xi$.
\end{theorem}

\begin{proof}
We choose $\alpha \in (\frac{1}{3}, \beta)$ such that $\beta \leq \frac{3}{2} \alpha$. Since $\alpha < \beta$, we also have $\bfx \in \scrc^{\alpha}([0,T],V)$. For $t \in [0,T]$ we define the mapping $\Phi_t : \scrd_X^{2\alpha}([0,t],W) \to \scrd_X^{2\alpha}([0,t],W)$ according to (\ref{def-Phi}). By Propositions \ref{prop-fixed-point-1}, \ref{prop-fixed-point-2} and Lemma \ref{lemma-alpha-beta}, for every $t \in [0,T]$ with $t \leq 1$ and all $(Y,Y'), (Z,Z') \in \bbb_t$ we have
\begin{align*}
&\| \Phi_t(Y,Y') \|_{X,2\alpha;[0,t]} \leq C ( 1 + |\xi|_{D(A)} ) t^{1-2\alpha}
\\ &\quad + C \big( \| X \|_{\beta,[0,t]} t^{\beta-\alpha} + \| \bbx \|_{2\beta,[0,t]} t^{2(\beta-\alpha)} + t^{\alpha} \big),
\\ &\| \Phi(Y,Y') - \Phi(Z,Z') \|_{X,2\alpha;[0,t]} \leq C \interleave (Y,Y') - (Z,Z') \interleave_{X,2\alpha,[0,t]}
\\ &\quad \big( \| X \|_{\beta,[0,t]} t^{\beta-\alpha} + \| \bbx \|_{2\beta,[0,t]} t^{2(\beta-\alpha)} + t^{\alpha} + t^{1-2\alpha} \big),
\end{align*}
where the constant $C > 0$ depends on $\alpha$, $\mathbf{X}$, $\| f_0 \|_{\Lip}$, $\| f \|_{C_b^{2\alpha,3}}$ and $M$, $\omega$. We choose $T_0 \in (0,T]$ with $T_0 \leq 1$ small enough such that
\begin{align*}
C ( 1 + |\xi|_{D(A)} ) T_0^{1-2\alpha} + C \big( \| X \|_{\beta,[0,T_0]} T_0^{\beta-\alpha} + \| \bbx \|_{2\beta,[0,T_0]} T_0^{2(\beta-\alpha)} + T_0^{\alpha} \big) \leq 1,
\\ C \big( \| X \|_{\beta,[0,T_0]} T_0^{\beta-\alpha} + \| \bbx \|_{2\beta,[0,t]} T_0^{2(\beta-\alpha)} + T_0^{\alpha} + T_0^{1-2\alpha} \big) \leq \frac{1}{2}.
\end{align*}
Then $\Phi_{T_0}$ is a contraction, and by the Banach fixed point theorem there is a unique mild solution $(Y,Y') \in \scrd_X^{2\alpha}([0,T_0],W)$ to the RPDE (\ref{RPDE}) with $Y_0 = \xi$. Now, applying Lemma \ref{lemma-spaces-alpha-beta} completes the proof.
\end{proof}

\subsection{Further auxiliary results}

In this subsection we derive further auxiliary results which are required in order to establish existence and uniqueness of global mild solutions for the RPDE (\ref{RPDE}).

\begin{proposition}\label{prop-estimate}
Let $\bfx = (X,\bbx) \in \scrc^{\alpha}([0,T],V)$ be a rough path for some index $\alpha \in (\frac{1}{3},\frac{1}{2})$, let $f_0 \in \Lip([0,T] \times W,D(A))$ and $f \in C_b^{2\alpha,2}([0,T] \times W, L(V,D(A^2)))$ be mappings, and let $\xi \in D(A)$ be arbitrary. Then there exists a constant $K > 0$, depending on $\xi$, $\alpha$, $T$, $\mathbf{X}$, $\| f_0 \|_{\Lip}$, $\| f \|_{C_b^{2\alpha,2}}$ and $M$, $\omega$, such that for every mild solution $(Y,Y') \in \scrd_X^{2\alpha}([0,T],W)$ to the RPDE (\ref{RPDE}) with $Y_0 = \xi$ we have
\begin{align*}
| Y_t | \leq K \quad \text{for all $t \in [0,T]$.}
\end{align*}
\end{proposition}

\begin{proof}
Let $s,t \in [0,T]$ with $s \leq t$ be arbitrary, and define the interval $I := [s,t]$. By Lemma \ref{lemma-remainder-mild-sol} we have
\begin{align*}
| R_{s,t}^Y | &\leq C \Big( 1 + \| Y \|_{\alpha;I} + \| X \|_{\alpha;I} ( 1 + \| R^{f(Y)} \|_{2 \alpha;I} ) + \| \bbx \|_{2\alpha;I} ( 1 + \| Y \|_{\alpha;I} ) \Big) |t-s|
\\ &\quad + C \| X \|_{\alpha;I} |t-s|^{2\alpha},
\end{align*}
where the constant $C > 0$ depends on $\xi$, $\alpha$, $T$, $\| f_0 \|_{\Lip}$, $\| f \|_{C_b^{2\alpha,2}}$ and $M$, $\omega$. We set $\beta := 1-2\alpha$. Since $\alpha < \frac{1}{2}$, we have $\beta > 0$. Furthermore, let $h > 0$ be arbitrary. Then we have
\begin{equation}\label{R-1}
\begin{aligned}
\| R^Y \|_{2\alpha;h} &\leq C \| X \|_{\alpha;h}
\\ &\quad + C \Big( 1 + \| Y \|_{\alpha;h} + \| X \|_{\alpha;h} ( 1 + \| R^{f(Y)} \|_{2 \alpha;h} ) + \| \bbx \|_{2\alpha;h} ( 1 + \| Y \|_{\alpha;h} ) \Big) h^{\beta}.
\end{aligned}
\end{equation}
By Proposition \ref{prop-comp-f} we have
\begin{align*}
\| R^{f(Y)} \|_{2\alpha;I} \leq \| f \|_{C_b^{2\alpha,2}} \bigg( 1 + \frac{1}{2} \| Y \|_{\alpha;I}^2 + \| R^Y \|_{2\alpha;I} \bigg),
\end{align*}
and hence
\begin{align}\label{R-2}
\| R^{f(Y)} \|_{2\alpha;h} \leq \| f \|_{C_b^{2\alpha,2}} \big( 1 + \| Y \|_{\alpha;h}^2 + \| R^Y \|_{2\alpha;h} \big).
\end{align}
Therefore, combining (\ref{R-1}) and (\ref{R-2}), there is a constant $c_1 > 0$, only depending on $\alpha$, $T$ and $\| f \|_{C_b^{2\alpha,2}}$, such that
\begin{align*}
\| R^Y \|_{2\alpha;h} &\leq c_1 \| X \|_{\alpha;h}
\\ &\quad + c_1 \Big( 1 + \| Y \|_{\alpha;h} + \| X \|_{\alpha;h} ( 1 + \| Y \|_{\alpha;h}^2 + \| R^Y \|_{2\alpha;h} ) + \| \bbx \|_{2\alpha;h} ( 1 + \| Y \|_{\alpha;h} ) \Big) h^{\beta}.
\end{align*}
Thus, there is a constant $c_2 > 0$, only depending on $\alpha$, $T$ and $\| f \|_{C_b^{2\alpha,2}}$, such that for all $h \in (0,1]$ we obtain
\begin{equation}\label{R-3}
 \begin{aligned}
\| R^Y \|_{2\alpha;h} &\leq c_2 \big( \| X \|_{\alpha;h} + \| \bbx \|_{2 \alpha; h} + 1 \big) + c_2 \| X \|_{\alpha;h} h^{\beta} \| Y \|_{\alpha;h}^2
\\ &\quad + c_2 \| X \|_{\alpha;h} h^{\beta} \| R^Y \|_{2\alpha;h} + c_2 \big( \| \bbx \|_{2\alpha;h} + 1 \big) h^{\beta} \| Y \|_{\alpha;h}.
\end{aligned}
\end{equation}
Now we consider $h \in (0,1]$ so small such that
\begin{align}\label{choose-h}
c_2 \| X \|_{\alpha} h^{\beta} \leq \frac{1}{2} \quad \text{and} \quad c_2 \big( \| X \|_{\alpha} + \| \bbx \|_{2 \alpha} + 1 \big)^{1/2} h^{\beta} \leq \frac{1}{2}.
\end{align}
Then by (\ref{R-3}) and (\ref{choose-h}) we have
\begin{align*}
\| R^Y \|_{2\alpha;h} &\leq c_2 \big( \| X \|_{\alpha;h} + \| \bbx \|_{2 \alpha; h} + 1 \big) + \frac{1}{2} \| Y \|_{\alpha;h}^2 + \frac{1}{2} \| R^Y \|_{2\alpha;h}
\\ &\quad + \frac{1}{2} \big( \| \bbx \|_{2\alpha;h} + 1 \big)^{1/2} \| Y \|_{\alpha;h}.
\end{align*}
Therefore, using the elementary inequality $xy \leq \frac{x^2}{2} + \frac{y^2}{2}$, $x,y \in \bbr$ we obtain
\begin{equation}\label{R-4}
\begin{aligned}
\| R^Y \|_{2\alpha;h} &\leq 2 c_2 \big( \| X \|_{\alpha;h} + \| \bbx \|_{2 \alpha; h} + 1 \big) + \| Y \|_{\alpha;h}^2 + \big( \| \bbx \|_{2\alpha;h} + 1 \big)^{1/2} \| Y \|_{\alpha;h}
\\ &\leq c_3 \big( \| X \|_{\alpha;h} + \| \bbx \|_{2 \alpha; h} + 1 \big) + \frac{3}{2} \| Y \|_{\alpha;h}^2
\end{aligned}
\end{equation}
with a constant $c_3 > 0$, only depending on $\alpha$, $T$ and $\| f \|_{C_b^{2\alpha,2}}$. On the other hand, since $Y_{s,t} = f(Y_s)X_{s,t} + R_{s,t}^Y$, we have
\begin{align*}
|Y_t - Y_s| \leq |f(Y_s)X_{s,t}| + |R_{s,t}^Y| \leq \|f\|_{C_b^{2\alpha,2}} \| X \|_{\alpha} |t-s|^{\alpha} + \| R^Y \|_{2 \alpha} |t-s|^{2\alpha},
\end{align*}
and hence
\begin{align*}
\| Y \|_{\alpha;h} \leq C \big( \| X \|_{\alpha;h} + \| R^Y \|_{2\alpha;h} h^{\alpha} \big),
\end{align*}
where the constant $C > 0$ depends on $\| f \|_{C_b^{2\alpha,2}}$. Note that $\beta < \alpha$. Indeed, since $3\alpha > 1$, we have $\beta = 1-2\alpha < 3\alpha - 2\alpha = \alpha$. Therefore, we have
\begin{align*}
\| Y \|_{\alpha;h} \leq C \big( \| X \|_{\alpha;h} + \| R^Y \|_{2\alpha;h} h^{\beta} \big).
\end{align*}
Thus, using (\ref{R-4}) and (\ref{choose-h}) we obtain
\begin{align*}
\| Y \|_{\alpha;h} &\leq c_4 \| X \|_{\alpha;h} + c_4 \big( \| X \|_{\alpha;h} + \| \bbx \|_{2 \alpha; h} + 1 \big) h^{\beta} + c_4 \| Y \|_{\alpha,h}^2 h^{\beta}
\\ &\leq c_4 \| X \|_{\alpha;h} + c_5 \big( \| X \|_{\alpha;h} + \| \bbx \|_{2 \alpha; h} + 1 \big)^{1/2} + c_4 \| Y \|_{\alpha,h}^2 h^{\beta}
\end{align*}
with constants $c_4,c_5 > 0$, only depending on $\alpha$, $T$ and $\| f \|_{C_b^{2\alpha,2}}$. Hence, multiplying this inequality with $c_4 h^{\beta}$ we have
\begin{align*}
c_4 \| Y \|_{\alpha;h} h^{\beta} \leq c_6 ( \interleave \mathbf{X} \interleave_{\alpha} + 1 ) h^{\beta} + ( c_4 \| Y \|_{\alpha,h} h^{\beta} )^2,
\end{align*}
where $c_6 := c_4 + c_5$. Now, we set
\begin{align*}
\psi_h := c_4 \| Y \|_{\alpha;h} h^{\beta} \quad \text{and} \quad \lambda_h := c_6 (\interleave \mathbf{X} \interleave_{\alpha} + 1) h^{\beta}.
\end{align*}
Then we have
\begin{align*}
\psi_h \leq \lambda_h + \psi_h^2.
\end{align*}
Now, literally the same argumentation as in the proof of \cite[Prop. 8.2]{Friz-Hairer-2020} shows the existence of a constant $N > 0$, depending on $\alpha$, $T$, $\mathbf{X}$, $\| f_0 \|_{\Lip}$, $\| f \|_{C_b^{2\alpha,2}}$ and $M$, $\omega$, such that $\| Y \|_{\alpha} \leq N$. Consequently, setting $K := |\xi| + N T^{\alpha}$, we obtain $\| Y \|_{\infty} \leq |Y_0| + \| Y \|_{\alpha} T^{\alpha} \leq K$, completing the proof.
\end{proof}

The following auxiliary result shows how two local mild solutions of the RPDE (\ref{RPDE}) can be concatenated.

\begin{lemma}\label{lemma-flow-property}
Let $\bfx = (X,\bbx) \in \scrc^{\alpha}([0,T],V)$ be a rough path for some index $\alpha \in (\frac{1}{3},\frac{1}{2}]$, and let $f_0 \in \Lip([0,T] \times W,D(A))$ and $f \in C_b^{2\alpha,2}([0,T] \times W, L(V,D(A^2)))$ be mappings. Moreover, let $0 \leq q \leq r \leq s \leq T$ and $\xi \in D(A)$ be arbitrary. Let $(Y(q,r,\xi),Y(q,r,\xi)') \in \scrd_X^{2\alpha}([q,r],W)$ be a solution to the equation $Y(q,r,\xi)' = f(Y(q,r,\xi))$ and
\begin{equation}\label{flow-1}
\begin{aligned}
Y(q,r,\xi)_t &= S_{t-q} \xi + \int_q^t S_{t-u} f_0(u,Y(q,r,\xi)_u) \, du
\\ &\quad + \int_q^t S_{t-u} f(u,Y(q,r,\xi)_u) \, d \mathbf{X}_u \quad t \in [q,r].
\end{aligned}
\end{equation}
Furthermore, we set $\eta := Y(q,r,\xi)_r$ and let $(Y(r,s,\eta),Y(r,s,\eta)') \in \scrd_X^{2\alpha}([r,s],W)$ be a solution to the equation $Y(r,s,\eta)' = f(Y(r,s,\eta))$ and
\begin{equation}
\begin{aligned}
Y(r,s,\eta)_t &= S_{t-r} \eta + \int_r^t S_{t-u} f_0(u,Y(r,s,\eta)_u) \, du
\\ &\quad + \int_r^t S_{t-u} f(u,Y(r,s,\eta)_u) \, d \mathbf{X}_u, \quad t \in [r,s].
\end{aligned}
\end{equation}
We define the concatenated path $Y(q,s,\xi) : [q,s] \to W$ as
\begin{align*}
Y(q,s,\xi)_t :=
\begin{cases}
Y(q,r,\xi)_t, & t \in [q,r],
\\ Y(r,s,\eta)_t, & t \in [r,s],
\end{cases}
\end{align*}
and we define the concatenated path $Y(q,s,\xi)' : [q,s] \to L(V,W)$ as
\begin{align*}
Y(q,s,\xi)_t' :=
\begin{cases}
Y(q,r,\xi)_t', & t \in [q,r],
\\ Y(r,s,\eta)_t', & t \in [r,s].
\end{cases}
\end{align*}
Then we have $(Y(q,s,\xi),Y(q,s,\xi)') \in \scrd_X^{2\alpha}([q,r],W)$, and this path is a solution to $Y(q,s,\xi)' = f(Y(q,s,\xi))$ and
\begin{equation}\label{RPDE-flow}
\begin{aligned}
Y(q,s,\xi)_t = S_{t-q} \xi + \int_q^t S_{t-u} f_0(u,Y(q,s,\xi)_u) \, du
\\ + \int_q^t S_{t-u} f(u,Y(q,s,\xi)_u) \, d\mathbf{X}_u, \quad t \in [q,s].
\end{aligned}
\end{equation}
\end{lemma}

\begin{proof}
It is easy to check that $(Y(q,s,\xi),Y(q,s,\xi)') \in \scrd_X^{2\alpha}([q,r],W)$ and $Y(q,s,\xi)' = f(Y(q,s,\xi))$. Furthermore, taking into account (\ref{flow-1}) it is evident that equation (\ref{RPDE-flow}) is satisfied for all $t \in [q,r]$. Moreover, for each $t \in [r,s]$ we obtain
\begin{align*}
Y(q,s,\xi)_t &= S_{t-r} Y(q,r,\xi)_r + \int_r^t S_{t-u} f_0(u,Y(r,s,Y(q,r,\xi)_r)_u) du
\\ &\quad + \int_r^t S_{t-u} f(u,Y(r,s,Y(q,r,\xi)_r)_u) d \mathbf{X}_u
\\ &= S_{t-r} \bigg( S_{r-q} \xi + \int_q^r S_{r-u} f_0(u,Y(q,r,\xi)_u) du + \int_q^r S_{r-u} f(u,Y(q,r,\xi)_u) d \mathbf{X}_u \bigg)
\\ &\quad + \int_r^t S_{t-u} f_0(u,Y(r,s,Y(q,r,\xi)_r)_u) du
\\ &\quad + \int_r^t S_{t-u} f(u,Y(r,s,Y(q,r,\xi)_r)_u) d \mathbf{X}_u
\\ &= S_{t-q} \xi + \int_q^r S_{t-u} f_0(u,Y(q,s,\xi)_u) du + \int_q^r S_{t-u} f(u,Y(q,s,\xi)_u) d\mathbf{X}_u
\\ &\quad + \int_r^t S_{t-u} f_0(u,Y(r,s,Y(q,r,\xi)_r)_u) du
\\ &\quad + \int_r^t S_{t-u} f(u,Y(r,s,Y(q,r,\xi)_r)_u) d \mathbf{X}_u,
\end{align*}
completing the proof.
\end{proof}

\subsection{Existence and uniqueness of global mild solutions}

In this subsection we present a result concerning existence and uniqueness of global mild solutions to the RPDE (\ref{RPDE}).

\begin{theorem}\label{thm-RDGL-main}
Let $\bfx = (X,\bbx) \in \scrc^{\beta}([0,T],V)$ be a rough path for some index $\beta \in (\frac{1}{3},\frac{1}{2}]$, and let $f_0 \in \Lip([0,T] \times W,D(A))$ and $f \in C_b^{2\beta,3}([0,T] \times W, L(V,D(A^2)))$ be mappings such that
\begin{align}\label{f0-restriction}
&f_0|_{[0,T] \times D(A)} \in \Lip([0,T] \times D(A),D(A^2)),
\\ \label{f-restriction} &f|_{[0,T] \times D(A)} \in C_b^{2\beta,3}([0,T] \times D(A), L(V,D(A^3))).
\end{align}
Then for every $\xi \in D(A^2)$ there exists a unique mild solution $(Y,Y') \in \scrd_X^{2 \beta}([0,T],W)$ to the RPDE (\ref{RPDE}) with $Y_0 = \xi$, which is also a strong solution.
\end{theorem}

\begin{proof}
We choose $\alpha \in (\frac{1}{3}, \beta)$ such that $\beta \leq \frac{3}{2} \alpha$. Since $\alpha < \beta$, we have $\bfx \in \scrc^{\alpha}([0,T],V)$. Taking into account (\ref{f0-restriction}), (\ref{f-restriction}) and Lemmas \ref{lemma-int-cont-embedded spaces}, \ref{lemma-restricted-semigroup}, \ref{lemma-graph-norms-eq}, by Proposition \ref{prop-estimate} for every mild solution $(Y,Y') \in \scrd_X^{2\alpha}([0,T],W)$ to the RPDE (\ref{RPDE}) with $Y_0 = \xi$ we have
\begin{align}\label{estimate-solution}
| Y_t |_{D(A)} \leq K \quad \text{for all $t \in [0,T]$,}
\end{align}
where the constant $K > 0$ depends on $\xi$, $\alpha$, $T$, $\mathbf{X}$, $\| f_0 \|_{\Lip}$, $\| f \|_{C_b^{2\alpha,3}}$ and $M$, $\omega$. For $t \in [0,T]$ we define the mapping $\Phi_t : \scrd_X^{2\alpha}([0,t],W) \to \scrd_X^{2\alpha}([0,t],W)$ according to (\ref{def-Phi}). By Propositions \ref{prop-fixed-point-1}, \ref{prop-fixed-point-2} and Lemma \ref{lemma-alpha-beta}, for every $t \in [0,T]$ with $t \leq 1$ and all $(Y,Y'), (Z,Z') \in \bbb_t$ we have
\begin{align*}
&\| \Phi_t(Y,Y') \|_{X,2\alpha;[0,t]} \leq C ( 1 + |\xi|_{D(A)} ) t^{1-2\alpha}
\\ &\quad + C \big( \| X \|_{\beta,[0,t]} t^{\beta-\alpha} + \| \bbx \|_{2\beta,[0,t]} t^{2(\beta-\alpha)} + t^{\alpha} \big),
\\ &\| \Phi(Y,Y') - \Phi(Z,Z') \|_{X,2\alpha;[0,t]} \leq C \interleave (Y,Y') - (Z,Z') \interleave_{X,2\alpha;[0,t]}
\\ &\quad ( \| X \|_{\beta,[0,t]} t^{\beta-\alpha} + \| \bbx \|_{2\beta,[0,t]} t^{2(\beta-\alpha)} + t^{\alpha} + t^{1-2\alpha} ),
\end{align*}
where the constant $C > 0$ depends on $\alpha$, $\mathbf{X}$, $\| f_0 \|_{\Lip}$, $\| f \|_{C_b^{2\alpha,3}}$ and $M$, $\omega$. We choose $t \in (0,1]$ such that $T = nt$ for some $n \in \bbn$ and
\begin{align*}
C ( 1 + K ) t^{1-2\alpha} + C \big( \| X \|_{\beta,[0,t]} t^{\beta-\alpha} + \| \bbx \|_{2\beta,[0,t]} t^{2(\beta-\alpha)} + t^{\alpha} \big) \leq 1,
\\ C \big( \| X \|_{\beta,[0,t]} t^{\beta-\alpha} + \| \bbx \|_{2\beta,[0,t]} t^{2(\beta-\alpha)} + t^{\alpha} + t^{1-2\alpha} \big) \leq \frac{1}{2}.
\end{align*}
Then $\Phi_t$ is a contraction, and by the Banach fixed point theorem there is a unique mild solution $(Y(1),Y(1)') \in \scrd_X^{2\alpha}([0,t],W)$ to the RPDE (\ref{RPDE}) with $Y(1)_0 = \xi$. Taking into account (\ref{estimate-solution}), we can inductively apply this argument to the intervals $[(k-1)t,kt]$ for all $k=2,\ldots,n$ to obtain a unique mild solution $(Y(k),Y(k)') \in \scrd_X^{2\alpha}([(k-1)t,kt])$ to the RPDE (\ref{RPDE}) with $Y(k)_{(k-1)t} = Y(k-1)_{(k-1)t}$. Using Lemma \ref{lemma-flow-property} we can concatenate these solutions and deduce the existence of a unique mild solution $(Y,Y') \in \scrd_X^{2\alpha}([0,T],W)$. By Lemma \ref{lemma-spaces-alpha-beta} we even have $(Y,Y') \in \scrd_X^{2\beta}([0,T])$, showing that $(Y,Y')$ is the unique mild solution to the RPDE (\ref{RPDE}) with $Y_0 = \xi$. By (\ref{estimate-solution}) the path $AY : [0,T] \to W$ is bounded. Hence, applying Proposition \ref{prop-mild-strong} concludes the proof.
\end{proof}

Under the conditions of Theorem \ref{thm-RDGL-main} we can consider the \emph{It\^{o}-Lyons map}
\begin{align}\label{IL-map}
(\xi,\mathbf{X}) \mapsto (Y,Y').
\end{align}
With certain adaptations, the following result regarding the continuity of the It\^{o}-Lyons map can be proven similarly as for rough ordinary differential equations; see, for example \cite[Sec. 8.6]{Friz-Hairer-2020}.

\begin{proposition}\label{prop-Ito-Lyons}
Let $\beta \in (\frac{1}{3},\frac{1}{2}]$ be an index, and let $f_0 \in \Lip([0,T] \times W,D(A))$ and $f \in C_b^{2\beta,3}([0,T] \times W, L(V,D(A^2)))$ be mappings such that (\ref{f0-restriction}) and (\ref{f-restriction}) are fulfilled. Then the It\^{o}-Lyons map
\begin{align*}
D(A^2) \times \scrc^{\beta}([0,T],V) \to \calc^{\beta}([0,T],W) \oplus \calc^{2\beta}([0,T]^2,W \otimes W)
\end{align*}
given by (\ref{IL-map}) is locally Lipschitz continuous.
\end{proposition}

Now, consider a rough ordinary differential equation (RODE) of the form
\begin{align}\label{RODE}
\left\{
\begin{array}{rcl}
dY_t & = & f_0(t,Y_t) dt + f(t,Y_t) d \bfx_t
\\ Y_0 & = & \xi.
\end{array}
\right.
\end{align}

\begin{corollary}
Let $\bfx = (X,\bbx) \in \scrc^{\beta}([0,T],V)$ be a rough path for some index $\beta \in (\frac{1}{3},\frac{1}{2}]$, and let $f_0 \in \Lip([0,T] \times W,W)$ and $f \in C_b^{2\beta,3}([0,T] \times W, L(V,W))$ be mappings. Then for every $\xi \in W$ there exists a unique solution $(Y,Y') \in \scrd_X^{2 \beta}([0,T],W)$ to the RODE (\ref{RODE}) with $Y_0 = \xi$.
\end{corollary}

In particular, the previous result applies to time-homogeneous RODEs of the form
\begin{align}\label{RODE-th}
\left\{
\begin{array}{rcl}
dY_t & = & f_0(Y_t) dt + f(Y_t) d \bfx_t
\\ Y_0 & = & \xi.
\end{array}
\right.
\end{align}

\begin{corollary}
Let $\bfx = (X,\bbx) \in \scrc^{\beta}([0,T],V)$ be a rough path for some index $\beta \in (\frac{1}{3},\frac{1}{2}]$, and let $f_0 : W \to W$ be Lipschitz continuous and $f \in C_b^{3}(W, L(V,W))$. Then for every $\xi \in W$ there exists a unique solution $(Y,Y') \in \scrd_X^{2 \beta}([0,T],W)$ to the time-homogeneous RODE (\ref{RODE-th}) with $Y_0 = \xi$.
\end{corollary}

\section{Stochastic partial differential equations driven by infinite dimensional Wiener processes}\label{sec-Wiener}

In this section we will apply our existence and uniqueness result for RPDEs (see Theorem \ref{thm-RDGL-main}) to stochastic partial differential equations driven by infinite dimensional Wiener processes. For this purpose, we start with some preparations. In the sequel, $(\Omega,\calf,(\calf_t)_{t \in \bbr_+},\bbp)$ denotes a filtered probability space satisfying the usual conditions.

\subsection{Infinite dimensional Wiener process as a rough path}\label{subsec-BM-enhanced}

In this subsection we demonstrate how typical sample paths of an infinite dimensional Wiener process can be realized as rough paths. Let $U$ be a separable Hilbert space, and let $X$ be an $U$-valued $Q$-Wiener process for some nuclear, self-adjoint, positive definite linear operator $Q \in L_1^{++}(U)$; see \cite[Def. 4.2]{Da_Prato}. There exist an orthonormal basis $\{ e_k \}_{k \in \bbn}$ of $U$ and a sequence $(\lambda_k)_{k \in \bbn} \subset (0,\infty)$ with $\sum_{k=1}^{\infty} \lambda_k < \infty$ such that
\begin{align*}
Q e_k = \lambda_k e_k \quad \text{for all $k \in \bbn$.}
\end{align*}
The space $U_0 := Q^{1/2}(U)$, equipped with the inner product
\begin{align*}
\langle u,v \rangle_{U_0} := \langle Q^{-1/2}u, Q^{-1/2}v \rangle_U, \quad u,v \in U_0
\end{align*}
is another separable Hilbert space. Note that
\begin{align*}
Q^{1/2} : (U,\| \cdot \|_U) \to (U_0,\| \cdot \|_{U_0})
\end{align*}
is an isometric isomorphism, and that $\{ \sqrt{\lambda_k} e_k \}_{k \in \bbn}$ is an orthonormal basis of $U_0$. Now, let $H$ be another separable Hilbert space. We denote by $L_2(U_0,H)$ the space of all Hilbert-Schmidt operators from $U_0$ into $H$. Furthermore, we define
\begin{align*}
L(U,H)_0 := \{ \Phi|_{U_0} : \Phi \in L(U,H) \}.
\end{align*}
The next result shows that $L(U,H)$ is continuously embedded in $L_2(U_0,H)$.

\begin{lemma}\label{lemma-HS-embedded}
We have $L(U,H)_0 \subset L_2(U_0,H)$ and
\begin{align*}
| \Phi|_{U_0} |_{L_2(U_0,H)} \leq \sqrt{\tr(Q)} | \Phi |_{L(U,H)} \quad \text{for all $\Phi \in L(U,H)$.}
\end{align*}
\end{lemma}

\begin{proof}
Recalling that $\{ \sqrt{\lambda_k} e_k \}_{k \in \bbn}$ is an orthonormal basis of $U_0$, we have
\begin{align*}
| \Phi |_{L_2(U_0,H)}^2 = \sum_{k=1}^{\infty} | \Phi(\sqrt{\lambda_k} e_k) |^2 = \sum_{k=1}^{\infty} \lambda_k | \Phi(e_k) |^2 \leq \tr(Q) | \Phi |_{L(U,H)}^2,
\end{align*}
completing the proof.
\end{proof}

For what follows, we fix a time horizon $T \in \bbr_+$. The $H$-valued \emph{It\^{o} integral} $\int_0^T \Phi_s d X_s$ can be defined for every predictable $L_2(U_0,H)$-valued process $\Phi$ such that
\begin{align}\label{integrable-Ito}
\bbe \bigg[ \int_0^T |\Phi_s|_{L_2(U_0,H)}^2 ds \bigg] < \infty,
\end{align}
and for each such process $\Phi$ we have the \emph{It\^{o} isometry}
\begin{align*}
\bbe \Bigg[ \bigg| \int_0^T \Phi_s d X_s \bigg|^2 \Bigg] = \bbe \bigg[ \int_0^T |\Phi_s|_{L_2(U_0,H)}^2 ds \bigg].
\end{align*}
We refer to \cite{Atma-book}, \cite{Da_Prato} or \cite{Liu-Roeckner} for further details.

\begin{remark}
Actually, the It\^{o} integral $\int_0^T \Phi_s d X_s$ can even by defined for every predictable $L_2(U_0,H)$-valued process $\Phi$ satisfying
\begin{align*}
\bbp \bigg( \int_0^T |\Phi_s|_{L_2(U_0,H)}^2 ds < \infty \bigg) = 1,
\end{align*}
but this is not required here.
\end{remark}

The following Burkholder-type inequality is a consequence of \cite[Lemma 3.3.b]{Atma-book}.

\begin{proposition}\label{prop-Burkholder}
For each $p \geq 2$ there exists a constant $c_p > 0$ such that for every predictable $L_2(U_0,H)$-valued process $\Phi$ satisfying
\begin{align*}
\bbe \Bigg[ \bigg( \int_0^T | \Phi_s |_{L_2(U_0,H)}^{2} ds \bigg)^{p/2} \Bigg] < \infty
\end{align*}
we have the estimate
\begin{align*}
\bbe \Bigg[ \sup_{t \in [0,T]} \bigg| \int_0^t \Phi_s dX_s \bigg|^{p} \Bigg] \leq c_p \, \bbe \Bigg[ \bigg( \int_0^T | \Phi_s |_{L_2(U_0,H)}^{2} ds \bigg)^{p/2} \Bigg].
\end{align*}
\end{proposition}

According to \cite[Prop. 4.3]{Da_Prato} the sequence $(\beta^k)_{k \in \bbn}$ defined as
\begin{align}\label{beta-j}
\beta^k := \frac{1}{\sqrt{\lambda_k}} \langle X,e_k \rangle_U, \quad k \in \bbn
\end{align}
is a sequence of independent real-valued standard Wiener processes, and according to \cite[Prop. 2.1.10]{Liu-Roeckner} the $Q$-Wiener process $X$ admits the series representation
\begin{align}\label{series-Wiener}
X_t = \sum_{k=1}^{\infty} \sqrt{\lambda_k} \beta_t^k e_k, \quad t \in [0,T].
\end{align}
More generally, we have the following consequence of \cite[Prop. 2.4.5]{Liu-Roeckner}.

\begin{proposition}\label{prop-series-integral}
For every predictable $L_2(U_0,H)$-valued process $\Phi$ satisfying (\ref{integrable-Ito}) we have the series representation
\begin{align*}
\int_0^t \Phi_s \, dX_s = \sum_{k=1}^{\infty} \sqrt{\lambda_k} \int_0^t \Phi_s(e_k) \, d \beta_s^k, \quad t \in [0,T].
\end{align*}
\end{proposition}

Moreover, the series representation (\ref{series-Wiener}) shows that
\begin{align}\label{norm-Wiener}
|X_{s,t}|^2 = \sum_{k=1}^{\infty} \lambda_k |\beta_{s,t}^k|^2, \quad s,t \in [0,T].
\end{align}

\begin{lemma}\label{lemma-Q-BM-1}
For each $p \geq 2$ there is a constant $C_p > 0$ such that
\begin{align*}
\bbe \big[ |X_{s,t}|^p \big] \leq C_p |t-s|^{p/2}, \quad s,t \in [0,T].
\end{align*}
\end{lemma}

\begin{proof}
Let $s,t \in [0,T]$ with $s < t$ be arbitrary. We define the predictable $L(U)$-valued process $\Phi = \Id \cdot \mathbf{1}_{(s,t]}$. Then we have
\begin{align*}
X_{s,t} = \int_0^T \Phi_r|_{U_0} dX_r.
\end{align*}
By Proposition \ref{prop-Burkholder} and Lemma \ref{lemma-HS-embedded} we obtain
\begin{align*}
\bbe \big[ |X_{s,t}|^p \big] &= \bbe \Bigg[ \bigg| \int_0^T \Phi_r|_{U_0} dX_r \bigg|^p \Bigg]
\\ &\leq c_p \, \bbe \Bigg[ \bigg( \int_0^T |\Phi_r |_{U_0}|_{L_2(U_0,U)}^2 dr \bigg)^{p/2} \Bigg]
\\ &\leq c_p \tr(Q)^{p/2} \, \bbe \Bigg[ \bigg( \int_0^T |\Phi_r|_{L(U)}^2 dr \bigg)^{p/2} \Bigg] = c_p |t-s|^{p/2},
\end{align*}
where the constant $c_p > 0$ stems from Proposition \ref{prop-Burkholder}.
\end{proof}

Since $U$ is a separable Hilbert space, we have $U \otimes U \simeq L_2(U)$. We define the predictable $L(U,L_2(U))$-valued process $\Psi$ as
\begin{align}\label{def-op-ONB}
\Psi_{s,r}(e_j)e_k = \sqrt{\lambda_k} \beta_{s,r}^k e_j, \quad s,r \in [0,T] \text{ and } j,k \in \bbn.
\end{align}
According to Lemma \ref{lemma-HS-embedded} the process $\Psi|_{U_0}$ is $L_2(U_0,L_2(U))$-valued.

\begin{lemma}\label{lemma-int-second-order}
For all $s,r \in [0,T]$ we have
\begin{align*}
| \Psi_{s,r}|_{U_0} |_{L_2(U_0,L_2(U))} = \sqrt{\tr(Q)} \, |X_{s,r}|.
\end{align*}
\end{lemma}

\begin{proof}
Recalling that $\{ \sqrt{\lambda_j} e_j \}_{j \in \bbn}$ is an orthonormal basis of $U_0$ and that $\{ e_k \}_{k \in \bbn}$ is an orthonormal basis of $U$, by (\ref{norm-Wiener}) we have
\begin{align*}
| \Psi_{s,r}|_{U_0} |_{L_2(U_0,L_2(U))}^2 &= \sum_{j=1}^{\infty} | \Psi_{s,r}(\sqrt{\lambda_j} e_j) |_{L_2(U)}^2 = \sum_{j=1}^{\infty} \lambda_j \sum_{k=1}^{\infty} |\Psi_{s,r}(e_j)e_k|_U^2
\\ &= \sum_{j=1}^{\infty} \lambda_j \sum_{k=1}^{\infty} \lambda_k |\beta_{s,r}^k|^2 = \tr(Q) \, |X_{s,r}|^2,
\end{align*}
completing the proof.
\end{proof}

Consequently, we can define the $L_2(U)$-valued second order process $\bbx$ as the It\^{o} integral
\begin{align*}
\bbx_{s,t} := \int_s^t \Psi_{s,r}|_{U_0} d X_r, \quad s,t \in [0,T].
\end{align*}

\begin{lemma}\label{lemma-second-order-series}
For all $s,t \in [0,T]$ we have
\begin{align*}
\bbx_{s,t} &= \sum_{j,k=1}^{\infty} \sqrt{\lambda_j \lambda_k} \bigg( \int_s^t \beta_{s,r}^j d \beta_r^k \bigg) (e_j \otimes e_k).
\end{align*}
\end{lemma}

\begin{proof}
By Proposition \ref{prop-series-integral} we have
\begin{align*}
\bbx_{s,t} = \sum_{k=1}^{\infty} \sqrt{\lambda_k} \int_s^t \Phi_{s,r}(e_k) d\beta_r^k.
\end{align*}
Thus, taking into account (\ref{def-op-ONB}), for each $j \in \bbn$ we obtain
\begin{align*}
\bbx_{s,t}e_j = \sum_{k=1}^{\infty} \sqrt{\lambda_k} \int_s^t \Phi_{s,r}(e_k)e_j \, d\beta_r^k = \sum_{k=1}^{\infty} \sqrt{\lambda_k} \int_s^t \sqrt{\lambda_j} \beta_{s,r}^j e_k \, d\beta_r^k,
\end{align*}
completing the proof.
\end{proof}

\begin{lemma}\label{lemma-Q-BM-2}
For each $p \geq 2$ there is a constant $C_p > 0$ such that
\begin{align*}
\bbe \big[ |\bbx_{s,t}|^p \big] \leq C_p |t-s|^p, \quad s,t \in [0,T].
\end{align*}
\end{lemma}

\begin{proof}
Applying Proposition \ref{prop-Burkholder} twice as well as H\"{o}lder's inequality and Lemma \ref{lemma-int-second-order}, we obtain
\begin{align*}
\bbe \big[ |\bbx_{s,t}|^p \big] &= \bbe \Bigg[ \bigg| \int_s^t \Psi_{s,r}|_{U_0} d X_r \bigg|^p \Bigg] \leq c_p \, \bbe \Bigg[ \bigg| \int_s^t |\Psi_{s,r}|_{U_0}|_{L_2(U_0,L_2(U))}^2 dr \bigg|^{p/2} \Bigg]
\\ &\leq c_p \, \bbe \Bigg[ \bigg| \int_s^t |\Psi_{s,r}|_{U_0}|_{L_2(U_0,L_2(U))}^p dr \bigg| \Bigg] |t-s|^{\frac{p}{2} - 1}
\\ &= c_p \, \bbe \Bigg[ \bigg| \int_s^t \big( |\Psi_{s,r}|_{U_0}|_{L_2(U_0,L_2(U))}^2 \big)^{p/2} dr \bigg| \Bigg] |t-s|^{\frac{p}{2} - 1}
\\ &\leq c_p \tr(Q)^{p/2} \, \bbe \Bigg[ \bigg| \int_s^t |X_{s,r}|^p dr \bigg| \Bigg] |t-s|^{\frac{p}{2} - 1}
\\ &\leq c_p \tr(Q)^{p/2} \, \bbe \bigg[ \sup_{r \in [s,t]} |X_{s,r}|^p \bigg] |t-s|^{\frac{p}{2}} \leq c_p^2 |t-s|^p,
\end{align*}
where the constant $c_p > 0$ stems from Proposition \ref{prop-Burkholder}.
\end{proof}

Now, we define the It\^{o}-enhanced $Q$-Wiener process $\mathbf{X} := (X,\bbx)$.

\begin{proposition}\label{prop-Wiener-rough-path}
For each $\alpha \in (\frac{1}{3},\frac{1}{2})$ we have $\bbp$-almost surely $\mathbf{X} \in \scrc^{\alpha}([0,T],U)$.
\end{proposition}

\begin{proof}
Using Lemma \ref{lemma-second-order-series}, it follows that the process $\mathbf{X}$ satisfies Chen's relation (\ref{Chen-relation}). Let $\alpha \in (\frac{1}{3},\frac{1}{2})$ be arbitrary. We choose $\beta = \frac{1}{2}$ and $q > 6$ arbitrary. Then we have $\beta - \frac{1}{q} > \frac{1}{3}$. By Lemmas \ref{lemma-Q-BM-1} and \ref{lemma-Q-BM-2} we have
\begin{align*}
| X_{s,t} |_{\call^q} &\leq C_q |t-s|^{1/2} = C_q |t-s|^{\beta}, \quad s,t \in [0,T],
\\ | \bbx_{s,t} |_{\call^{q/2}} &\leq C_q |t-s| = C_q |t-s|^{2 \beta}, \quad s,t \in [0,T],
\end{align*}
with a constant $C_q > 0$. Consequently, applying \cite[Thm. 3.1]{Friz-Hairer-2020} completes the proof.
\end{proof}

\subsection{Coincidence of the two integrals}

In this subsection we show that the rough integral and the It\^{o} integral coincide. Let $\mathbf{X} = (X,\bbx)$ be the It\^{o}-enhanced $Q$-Wiener process, as introduced in Subsection \ref{subsec-BM-enhanced}. We start with two auxiliary results.

\begin{lemma}\label{lemma-approx-integral}
Let $Y$ be a continuous, adapted and bounded $L(U,H)$-valued process. Furthermore, let $(\Pi_n)_{n \in \bbn}$ be a sequence of partitions of the interval $[0,T]$ such that $|\Pi_n| \to 0$. Then we have
\begin{align*}
\sum_{[u,v] \in \Pi_n} Y_u X_{u,v} \overset{L^2}{\to} \int_0^T Y_s|_{U_0} \, dX_s \quad \text{as $n \to \infty$.}
\end{align*}
\end{lemma}

\begin{proof}
Let $\Pi = \{ 0 = t_1 < \ldots < t_{K+1} = T \}$ be an arbitrary partition of the interval $[0,T]$. We define the elementary $L(U,H)$-valued process
\begin{align*}
Y^{\Pi} := \sum_{i=1}^{K} Y_{t_i} \mathbf{1}_{(t_i,t_{i+1}]}.
\end{align*}
Note that $Y^{\Pi}$ is predictable, because $Y$ is adapted. By the definition of the elementary It\^{o} integral we have
\begin{align}\label{elem-integral}
\sum_{[u,v] \in \Pi} Y_u X_{u,v} = \sum_{i=1}^{K} Y_{t_i} X_{t_i,t_{i+1}} = \int_0^T Y^{\Pi} dX_s.
\end{align}
Now, we set $Y^n := Y^{\Pi_n}$ for each $n \in \bbn$. By the continuity of $Y$ we have $Y^n \to Y$ pointwise. Therefore, and since $Y$ is bounded, by Lemma \ref{lemma-HS-embedded} and Lebesgue's dominated convergence theorem we obtain
\begin{align*}
\bbe \bigg[ \int_0^T | Y_s|_{U_0} - Y_s^n|_{U_0} |_{L_2(U_0,H)}^2 ds \bigg] \leq \tr(Q) \, \bbe \bigg[ \int_0^T | Y_s - Y_s^n |_{L(U,H)}^2 ds \bigg] \to 0.
\end{align*}
Consequently, it follows that
\begin{align*}
\int_0^T Y_s^n \, dX_s \overset{L^2}{\to} \int_0^T Y_s|_{U_0} \, dX_s \quad \text{as $n \to \infty$.}
\end{align*}
In view of (\ref{elem-integral}), this completes the proof.
\end{proof}

\begin{lemma}\label{lemma-mart}
Let $(\calg_n)_{n=0,\ldots,K}$ be a time-discrete filtration for some $K \in \bbn$, and let $(Y_n)_{n=1,\ldots,K}$ be a predictable, bounded  $L(U,H)$-valued process. Furthermore, let $(X_n)_{n=1,\ldots,K}$ be an adapted, square-integrable $U$-valued process such that $X_n$ is independent of $\calg_{n-1}$ and $\bbe[X_n] = 0$ for all $n=1,\ldots,K$. We define the $H$-valued process $S = (S_n)_{n=0,\ldots,K}$ as $S_0 := 0$ and
\begin{align*}
S_n := \sum_{j=1}^n Y_j X_j, \quad n=1,\ldots,K.
\end{align*}
Then the following statements are true:
\begin{enumerate}
\item[(a)] $S$ is a square-integrable $(\calg_n)_{n=0,\ldots,K}$-martingale.

\item[(b)] For all $n=1,\ldots,K$ we have
\begin{align*}
\bbe \Bigg[ \bigg| \sum_{j=1}^{n} (S_{j} - S_{j-1}) \bigg|^2 \Bigg] = \bbe \bigg[ \sum_{j=1}^{n} |S_{j} - S_{j-1}|^2 \bigg].
\end{align*}
\end{enumerate}
\end{lemma}

\begin{proof}
It is clear that the process $S$ is square-integrable. Let $n=1,\ldots,K$ be arbitrary. Then we have
\begin{align*}
\bbe[S_n - S_{n-1} | \calg_{n-1}] = \bbe[Y_n X_n | \calg_{n-1}] = Y_n \, \bbe[X_n | \calg_{n-1}] = Y_n \, \bbe[X_n] = 0,
\end{align*}
proving that $S$ is a $(\calg_n)_{n=0,\ldots,K}$-martingale. Furthermore, we obtain
\begin{align*}
&\bbe \Bigg[ \bigg| \sum_{j=1}^{n} (S_{j} - S_{j-1}) \bigg|^2 \Bigg] = \bbe \Bigg[ \sum_{j=1}^{n} |S_{j} - S_{j-1}|^2 + 2 \sum_{j,k=1 \atop j < k}^{n} \la S_{j} - S_{j-1}, S_{k} - S_{k-1} \ra \Bigg].
\end{align*}
Moreover, by the martingale property of $S$ we have
\begin{align*}
&\bbe \Bigg[ \sum_{j,k=1 \atop j < k}^{n} \la S_{j} - S_{j-1}, S_{k} - S_{k-1} \ra \Bigg] = \sum_{j,k=1 \atop j < k}^{n} \bbe \big[ \bbe [ \la S_{j} - S_{j-1}, S_{k} - S_{k-1} \ra \,|\, \calf_{k-1} ] \big]
\\ &= \sum_{j,k=1 \atop j < k}^{n} \bbe \big[ \la S_{j} - S_{j-1}, \underbrace{\bbe [ S_{k} - S_{k-1} \,|\, \calf_{k-1} ]}_{= 0} \ra \big] = 0,
\end{align*}
completing the proof.
\end{proof}

For the following result we recall that $U \otimes U \simeq L_2(U)$ and $L(U,L(U,H)) \hookrightarrow L(L_2(U),H)$. Furthermore, recall that $L(U,H)$ is continuously embedded in $L_2(U_0,H)$, as stated in Lemma \ref{lemma-HS-embedded}. Let $\alpha \in (\frac{1}{3},\frac{1}{2})$ be an arbitrary index. According to Proposition \ref{prop-Wiener-rough-path} there is a $\bbp$-nullset $N_1$ such that $\mathbf{X} \in \scrc^{\alpha}([0,T],U)$ on $N_1^c$.

\begin{proposition}\label{prop-integrals-coincide}
Let $Y$ be a continuous $L(U,H)$-valued process, and let $Y'$ be a continuous $L(U,L(U,H))$-valued processes such that $(Y,Y') \in \scrd_X^{2 \alpha}([0,T],L(U,H))$ on $N_1^c$. Then the following statements are true:
\begin{enumerate}
\item[(a)] The $H$-valued rough integral
\begin{align}\label{rough-int-Wiener}
\int_0^T Y_s \, d \mathbf{X}_s = \lim_{|\Pi| \to 0} \sum_{[u,v] \in \Pi} ( Y_u X_{u,v} + Y_u' \bbx_{u,v} )
\end{align}
exists on $N_1^c$.

\item[(b)] If $Y$ and $Y'$ are adapted and bounded, then there is a $\bbp$-nullset $N$ with $N_1 \subset N$ such that
\begin{align*}
\int_0^T Y_s \, d \mathbf{X}_s = \int_0^T Y_s|_{U_0} \, dX_s \quad \text{on $N^c$.}
\end{align*}
\end{enumerate}
\end{proposition}

\begin{proof}
The first statement is a consequence of Theorem \ref{thm-Gubinelli}. We proceed with the proof of the second statement. By hypothesis we have $|Y|, |Y'| \leq M$ for some constant $M > 0$. Let $(\Pi_n)_{n \in \bbn}$ be a sequence of partitions of the interval $[0,T]$ with $|\Pi_n| \to 0$. By Lemma \ref{lemma-approx-integral} we have
\begin{align*}
\sum_{[u,v] \in \Pi_n} Y_u X_{u,v} \overset{\bbp}{\to} \int_0^t Y_s|_{U_0} \, dX_s \quad \text{as $n \to \infty$.}
\end{align*}
Passing to a subsequence, if necessary, we obtain
\begin{align}\label{a-s-conv}
\sum_{[u,v] \in \Pi_n} Y_u X_{u,v} \overset{\text{a.s.}}{\to} \int_0^t Y_s \, dX_s \quad \text{as $n \to \infty$.}
\end{align}
Hence, there is a $\bbp$-nullset $N_2$ such that the above convergence (\ref{a-s-conv}) holds true on $N_2^c$. Together with (\ref{rough-int-Wiener}) we obtain
\begin{align*}
\lim_{n \to \infty} \sum_{[u,v] \in \Pi_n} Y_u' \bbx_{u,v} &= \lim_{n \to \infty} \sum_{[u,v] \in \Pi_n} ( Y_u X_{u,v} + Y_u' \bbx_{u,v} ) -  \lim_{n \to \infty} \sum_{[u,v] \in \Pi_n} Y_u X_{u,v}
\\ &= \int_0^t Y_s \, d \mathbf{X}_s - \int_0^t Y_s|_{U_0} \, dX_s
\end{align*}
on $(N_1 \cup N_2)^c$. Let $\Pi = \{ 0 = t_1 < \ldots < t_{K+1} = T \}$ be an arbitrary partition of the interval $[0,T]$. By Lemma \ref{lemma-mart} (with $U$ replaced by $L_2(U)$) the $H$-valued process $S = (S_n)_{n=0,\ldots,K}$ given by $S_0 := 0$ and
\begin{align*}
S_n := \sum_{j=1}^{n} Y_{t_j}' \bbx_{t_j,t_{j+1}}, \quad n=1,\ldots,K
\end{align*}
is a time-discrete square-integrable $(\calf_{t_{n+1}})_{n=0,\ldots,K}$-martingale. Furthermore, by Lemma \ref{lemma-mart} and Lemma \ref{lemma-Q-BM-2} we have
\begin{align*}
\bbe \Bigg[ \bigg| \sum_{[u,v] \in \Pi} Y_u' \bbx_{u,v} \bigg|^2 \Bigg] &= \bbe \big[ |S_K|^2 \big] = \bbe \Bigg[ \bigg| \sum_{j=1}^{K} (S_{j} - S_{j-1}) \bigg|^2 \Bigg]
\\ &= \bbe \bigg[ \sum_{j=1}^{K} |S_{j} - S_{j-1}|^2 \bigg] \leq M^2 \sum_{j=1}^{K} \bbe \big[ |\bbx_{t_{j},t_{j+1}}|^2 \big]
\\ &\leq M^2 C \sum_{j=1}^{K} |t_{j+1} - t_j|^2 \leq M^2 C T |\Pi|,
\end{align*}
where the constant $C > 0$ stems from Lemma \ref{lemma-Q-BM-2}. Therefore, we have
\begin{align*}
\sum_{[u,v] \in \Pi_n} Y_u' \bbx_{u,v} \overset{L^2}{\to} 0 \quad \text{as $n \to \infty$.}
\end{align*}
Thus, passing to a subsequence, if necessary, we obtain
\begin{align}\label{a-s-conv-2}
\sum_{[u,v] \in \Pi_n} Y_u' \bbx_{u,v} \overset{\text{a.s.}}{\to} 0 \quad \text{as $n \to \infty$.}
\end{align}
Hence, there is a $\bbp$-nullset $N_3$ such that the above convergence (\ref{a-s-conv-2}) holds true on $N_3^c$. Consequently, setting $N := N_1 \cup N_2 \cup N_3$ completes the proof.
\end{proof}

\subsection{Stochastic partial differential equations}

Now, we are ready for our study of stochastic partial differential equations (SPDEs). Let $\mathbf{X} = (X,\bbx)$ be the It\^{o}-enhanced $Q$-Wiener process, as introduced in Subsection \ref{subsec-BM-enhanced}. Furthermore, let $W$ be a Banach space, and let $A$ be the generator of a $C_0$-semigroup $(S_t)_{t \geq 0}$ on $W$. Consider the random RPDE
\begin{align}\label{random-SPDE}
\left\{
\begin{array}{rcl}
dY_t & = & (A Y_t + f_0(t,Y_t)) dt + f(t,Y_t) d \mathbf{X}_t
\\ Y_0 & = & \xi
\end{array}
\right.
\end{align}
with mappings $f_0 : [0,T] \times W \to W$ and $f : [0,T] \times W \to L(U,W))$. Let $\beta \in (\frac{1}{3},\frac{1}{2})$ be an index. By Proposition \ref{prop-Wiener-rough-path} there is a $\bbp$-nullset $N$ such that $\mathbf{X} \in \scrc^{\beta}([0,T],U)$ on $N^c$.

\begin{theorem}
Let $f_0 \in \Lip([0,T] \times W,D(A))$ and $f \in C_b^{2\beta,3}([0,T] \times H, L(U,D(A^2)))$ be mappings such that
\begin{align*}
&f_0|_{[0,T] \times D(A)} \in \Lip([0,T] \times D(A),D(A^2)),
\\ &f|_{[0,T] \times D(A)} \in C_b^{2\beta,3}([0,T] \times D(A), L(U,D(A^3))).
\end{align*}
Then for every $\xi \in D(A^2)$ the following statements are true:
\begin{enumerate}
\item[(a)] There exists a unique mild solution $(Y,Y') \in \scrd_X^{2 \beta}([0,T],W)$ to the random RPDE (\ref{random-SPDE}) with $Y_0 = \xi$ on $N^c$, which is also a strong solution.

\item[(b)] If $W$ is a separable Hilbert space $H$, then there is a $\bbp$-nullset $N_1$ with $N \subset N_1$ such that the associated stochastic process $Y$ restricted to $N_1^c$ is also a mild and strong solution to the It\^{o} SPDE
\begin{align*}
\left\{
\begin{array}{rcl}
dY_t & = & (A Y_t + f_0(t,Y_t)) dt + f(t,Y_t)|_{U_0} d X_t
\\ Y_0 & = & \xi.
\end{array}
\right.
\end{align*}
\end{enumerate}
\end{theorem}

\begin{proof}
The first statement is an immediate consequence of Theorem \ref{thm-RDGL-main}. For the proof of the second statement, note that $\calf_t^X = \calf_t^{\mathbf{X}} \subset \calf_t$ for all $t \in [0,T]$. By the continuity of the It\^{o}-Lyons maps (see Proposition \ref{prop-Ito-Lyons}) it follows that the stochastic process $(Y,Y')$ is adapted. Therefore, applying Proposition \ref{prop-integrals-coincide} concludes the proof.
\end{proof}

\begin{remark}
Compared to other SPDE results in the literature (see, for example \cite{Da_Prato} or \cite{Atma-book}) we have imposed stronger conditions on the coefficients; in particular, usually only Lipschitz type assumptions on $f$ are required. On the other hand, we immediately obtain the \emph{flow} of solutions $Y^{\xi}(\omega)$ for $\xi \in D(A^2)$ and $\omega \in N^c$.
\end{remark}

\section{Stochastic partial differential equations driven by infinite dimensional fractional Brownian motion}\label{sec-fractional}

In this section we will apply our existence and uniqueness result for RPDEs (see Theorem \ref{thm-RDGL-main}) to stochastic partial differential equations driven by infinite dimensional fractional Brownian motion. Fractional Brownian motion in Hilbert spaces has been studied, for example, in \cite{Duncan, Grecksch}.

Let $(\Omega,\calf,(\calf_t)_{t \in \bbr_+},\bbp)$ be a filtered probability space satisfying the usual conditions. Furthermore, let $U$ be a separable Hilbert space. Recall that an $U$-valued process $X$ is called a \emph{Gaussian process} if for all $n \in \bbn$ and all $t_1,\ldots,t_n \in \bbr_+$ with $t_1 < \ldots < t_n$ the $U^n$-valued random variable $(X_{t_1},\ldots,X_{t_n})$ is a Gaussian random variable. An $U$-valued Gaussian process $X$ is called \emph{centered} if $\bbe[X_t] = 0$ for all $t \in \bbr_+$. For an $U$-valued Gaussian process $X$ we introduce the \emph{covariance function}
\begin{align*}
R : \bbr_+^2 \to L_2(U), \quad (s,t) \mapsto \bbe[X_s \otimes X_t].
\end{align*}
Let $Q \in L_1^{++}(U)$ be a nuclear, self-adjoint, positive definite linear operator.

\begin{definition}
A centered Gaussian process $X$ is called a \emph{$Q$-fractional Brownian motion} with Hurst index $H \in (\frac{1}{3},\frac{1}{2}]$ if its covariance function is given by
\begin{align*}
R(s,t) =  \frac{1}{2} \Big( s^{2H} + t^{2H} - |t-s|^{2H} \Big) Q, \quad s,t \in \bbr_+.
\end{align*}
\end{definition}

For what follows, let $X$ be a $Q$-fractional Brownian motion with Hurst index $H \in (\frac{1}{3},\frac{1}{2}]$. We also fix a time horizon $T \in \bbr_+$. There exist an orthonormal basis $\{ e_k \}_{k \in \bbn}$ of $U$ and a sequence $(\lambda_k)_{k \in \bbn} \subset (0,\infty)$ with $\sum_{k=1}^{\infty} \lambda_k < \infty$ such that
\begin{align*}
Q e_k = \lambda_k e_k \quad \text{for all $k \in \bbn$.}
\end{align*}
The $Q$-fractional Brownian motion $X$ admits the series representation
\begin{align*}
X_t = \sum_{k=1}^{\infty} \sqrt{\lambda_k} \beta_t^k e_k, \quad t \in [0,T]
\end{align*}
with a sequence of independent real-valued fractional Brownian motions $(\beta^k)_{k \in \bbn}$ with Hurst index $H$; see, for example  \cite{Grecksch-Anh}, \cite{Duncan-Maslowski} or \cite{Grecksch}.

\begin{proposition}\label{prop-frac-rough-path}
There exists a L\'{e}vy area $\bbx$ such that $\bbp$-almost surely $\mathbf{X} := (X,\bbx) \in \scrc_g^{\alpha}([0,T],U)$ for each $\alpha \in (\frac{1}{3},H)$.
\end{proposition}

\begin{proof}
Let $\alpha \in (\frac{1}{3},H)$ be arbitrary. According to \cite[Lemma 2.4]{Hesse-Neamtu-local} there exist a L\'{e}vy area $\bbx$ and a sequence $(X^n)_{n \in \bbn}$ of continuous, piecewise linear functions $X^n : [0,T] \to U$ starting in zero such that $\bbp$-almost surely $\interleave \mathbf{X}^n - \mathbf{X} \interleave_{\alpha} \to 0$, where we have set $\mathbf{X}^n := (X^n,\bbx^n)$, $n \in \bbn$, and the second order processes are given by
\begin{align}\label{Levy-area-n}
\bbx_{s,t}^n = \int_s^t X_{s,r}^n \otimes dX_r^n, \quad s,t \in [0,T].
\end{align}
By Lemma \ref{lemma-geom-1} we have $\mathbf{X}^n \in \scrc_g^{\alpha}([0,T],U)$ for each $n \in \bbn$. Therefore, by Lemma \ref{lemma-geom-2} we deduce that $\bbp$-almost surely $\mathbf{X} \in \scrc_g^{\alpha}([0,T],U)$.
\end{proof}

\begin{remark}
If $H = \frac{1}{2}$, then $X$ is a $Q$-Wiener process (see Subsection \ref{subsec-BM-enhanced}), and the $\mathbf{X} = (X,\bbx)$ according to Proposition \ref{prop-frac-rough-path} is a Stratonovich-enhanced $Q$-Wiener process.
\end{remark}

\begin{remark}
For each $n \in \bbn$ the L\'{e}vy area $\bbx^n$ given by (\ref{Levy-area-n}) can be expressed as
\begin{align*}
\bbx_{s,t}^n = \sum_{j,k=1}^{\infty} \sqrt{\lambda_j \lambda_k} \bigg( \int_s^t \beta_{s,r}^{j,n} \, d\beta_r^{k,n} \bigg) (e_j \otimes e_k),
\end{align*}
where $\beta^{j,n}$ denotes the corresponding $n$-th linear interpolation of $\beta^j$.
\end{remark}

Now, let $W$ be a Banach space, and let $A$ be the generator of a $C_0$-semigroup $(S_t)_{t \geq 0}$ on $W$. Consider the random RPDE
\begin{align}\label{RPDE-fractional}
\left\{
\begin{array}{rcl}
dY_t & = & (A Y_t + f_0(t,Y_t)) dt + f(t,Y_t) d \mathbf{X}_t
\\ Y_0 & = & \xi
\end{array}
\right.
\end{align}
with mappings $f_0 : [0,T] \times W \to W$ and $f : [0,T] \times W \to L(U,W))$. Let $\beta \in (\frac{1}{3},H)$ be an index. By Proposition \ref{prop-frac-rough-path} there is a $\bbp$-nullset $N$ such that $\mathbf{X} \in \scrc_g^{\beta}([0,T],U)$ on $N^c$. As an immediate consequence of Theorem \ref{thm-RDGL-main} we obtain the following result.

\begin{theorem}\label{thm-frac}
Let $f_0 \in \Lip([0,T] \times W,D(A))$ and $f \in C_b^{2\beta,3}([0,T] \times H, L(U,D(A^2)))$ be mappings such that
\begin{align*}
&f_0|_{[0,T] \times D(A)} \in \Lip([0,T] \times D(A),D(A^2)),
\\ &f|_{[0,T] \times D(A)} \in C_b^{2\beta,3}([0,T] \times D(A), L(U,D(A^3))).
\end{align*}
Then for every $\xi \in D(A^2)$ there exists a unique mild solution $(Y,Y') \in \scrd_X^{2 \beta}([0,T],W)$ to the random RPDE (\ref{RPDE-fractional}) with $Y_0 = \xi$ on $N^c$, which is also a strong solution.
\end{theorem}

\begin{remark}
As for random RPDEs driven by infinite dimensional Wiener processes, we obtain the \emph{flow} of solutions $Y^{\xi}(\omega)$ for $\xi \in D(A^2)$ and $\omega \in N^c$.
\end{remark}

\begin{remark}
If $W$ is a separable Hilbert space $H$, then there exists a theory of stochastic integration (see \cite{Duncan}), and one may check that the associated stochastic process $Y$ from Theorem \ref{thm-frac} is also a mild and strong solution to the fractional SPDE
\begin{align*}
\left\{
\begin{array}{rcl}
dY_t & = & (A Y_t + f_0(t,Y_t)) dt + f(t,Y_t) d X_t
\\ Y_0 & = & \xi.
\end{array}
\right.
\end{align*}
\end{remark}

\bibliographystyle{plain}

\bibliography{RPDEs}

\end{document}